\documentclass[a4paper,11pt]{article}
\usepackage{stmaryrd}
\usepackage{graphicx}
\usepackage{amsfonts}
\usepackage{amsmath,amsthm}
\usepackage{amssymb}
\usepackage{mathrsfs}
\usepackage{indentfirst}
\usepackage{titlesec}
\usepackage{caption2}
\usepackage{color}
\usepackage[normalem]{ulem}
\usepackage{algorithm}
\usepackage{algorithmic}
\usepackage{lineno}
\numberwithin{equation}{section}
\usepackage[english]{babel}
\usepackage{cite}
\usepackage{bbm}
\usepackage[hypertexnames=false]{hyperref}
\usepackage{comment}

\newtheorem{theorem}{Theorem}
\newtheorem{definition}{Definition}

\newtheorem{assumption}{Assumption}

\newtheorem{remark}{Remark}
\newtheorem{lemma}{Lemma}[section]
\newtheorem{corollary}{Corollary}[section]

\providecommand{\torus}{\mathbb{T}}
\providecommand{\sphere}{\mathbb{S}}
\providecommand{\su}{\mathfrak{su}}

\begin{document}

\title{Coadjoint averaging and cross-scale fluxes in a fast-slow stochastic Euler-Arnold system on $SU(N)$}

\author{
Wenqing Hu
\thanks{Department of Mathematics and Statistics, Missouri University of Science and Technology
(formerly University of Missouri, Rolla), Rolla, MO, 65409, USA. Email: \texttt{huwen@mst.edu}}
}

\date{}

\maketitle

\begin{abstract}
We study the emergence of (cross-scale) fluxes associated with energy and enstrophy in a stochastic version of Zeitlin's $SU(N)$ approximation of 2-d Euler dynamics. Motivated by the Euler-Arnold
formulation, we interpret the nonlinear transport as motion along coadjoint
orbits, reflecting an underlying symmetry that preserves all Casimir invariants. We introduce a fast-slow stochastic framework in which rapid mixing is modeled
by a structured fast stochastic forcing acting along selected directions. We show that when the fast dynamics preserves the full coadjoint orbit symmetry given by the natural symplectic structure, the averaged system exhibits no nontrivial flux. In contrast, when this symmetry is broken by tangential non-Hamiltonian vector fields on the coadjoint orbit, we identify conditions that yield nonzero fluxes carried by the Euler-Arnold nonlinearity. Thus, coadjoint symmetry suppresses averaged nonlinear flux, whereas its breaking can select a preferred direction of cross-scale transfer.
\end{abstract}



\tableofcontents

\section{Introduction}\label{Sec:Intro}

Understanding the mechanism of energy and enstrophy transfer in $2$-d fluid flows remains an important problem, particularly in connection with Kraichnan's 2-d turbulence theory that involves phenomenological statements such as constant flux and dual cascade at large Reynolds numbers (see \cite{[Kraichnan1967]}, \cite[Lecture 39]{[SverakLectureNotesHydrodynamics]} and Section \ref{Sec:Background} for background introduction). While the physical picture is well established, a mathematically tractable framework that captures the interplay between nonlinear transport, forcing, and dissipation is still far from complete (see, for example,  \cite{[BedrossianCascadeLawsSchrodinger]}, \cite{[BedrossianZelatiPaulFluxSufficient]}, \cite{[DudleyFluxLaws]} for some recent progress in this direction). One major difficulty is that the Kraichnan picture concerns a non-equilibrium limiting regime in which the Hamiltonian/Lie-Poisson structure of $2$-d Euler flows, the invariant measures of the forced-dissipative dynamics, and the inviscid large-scale limit must be understood simultaneously.

A natural approach is to approximate the $2$-d Euler dynamics by finite-dimensional systems. Among the available approximations, the $SU(N)$ model introduced by Zeitlin (see \cite{[Zeitlin1991]}) is particularly appealing. Take the torus $\torus^2$ as an example. Schematically, the $2$-d Euler equation on $\torus^2$ written in terms of vorticity $\omega=\omega(x,t)$ takes the form 
\begin{equation}\label{Intro:Eq:2dEuler-vorticity}
\dfrac{\partial \omega}{\partial t} + u \cdot \nabla \omega = 0 \ , 
\end{equation}
where the velocity $u = \nabla^\perp \Delta^{-1}\omega$ and $\nabla^\perp = \left(-\dfrac{\partial}{\partial x_2}, \dfrac{\partial}{\partial x_1}\right)$. This equation can be viewed as an Euler-Arnold equation describing the geodesic flows on the infinite-dimensional group $SDiff(\torus^2)$ \footnote{The group of smooth,
orientation-preserving, volume-preserving diffeomorphisms of $\mathbb{T}^2$.} equipped with the identity metric (see \cite{[Arnold1966JFourier]}, \cite{[Arnold-KhesinBook]}, \cite{[ArnoldClassicalMechanics]}, \cite{[MarsdenWeistein1983]}). In Zeitlin's approximation, equation (\ref{Intro:Eq:2dEuler-vorticity}) is approximated by the following Euler-Arnold equation on $SU(N)$:
\begin{equation}\label{Intro:Eq:EulerArnoldZeitlinTruncated}
\dot{\Omega} = [\Psi, \Omega] \ ,
\end{equation}
where $\Omega \in \mathfrak{su}(N)$ is on the Lie algebra of $SU(N)$, $[\bullet, \bullet]$ is the Lie bracket on $\su(N)$, and $\Psi=-A^{-1}\Omega$ for a linear positively-definite symmetric operator $A: \su(N)\rightarrow \su(N)$ (see (\ref{Eq:InertiaOperatorZeitlinCase})) mimicking the negative Laplace operator $-\Delta$ on $\torus^2$. One can imagine that there is a projection map $T_N \omega = \Omega$, and it can be shown that as $N\rightarrow \infty$, for identical initial conditions $T_N\omega(0)=\Omega(0)$, the approximation of (\ref{Intro:Eq:EulerArnoldZeitlinTruncated}) to (\ref{Intro:Eq:2dEuler-vorticity}) is in the sense that for every fixed $t$ we have $\|T_N^* \Omega(t)-\omega(t)\|_{L^2(\torus^2)}\rightarrow 0$ as $N\rightarrow \infty$ (see \cite{[GallagherZeitlinApproximation]} and the discussions in Section \ref{Sec:Background}).

As an Euler-Arnold equation on the infinite-dimensional group $SDiff(\torus^2)$, equation (\ref{Intro:Eq:2dEuler-vorticity}) admits a natural decomposition into geometric symmetry and
mechanical evolution.  The Lie-Poisson structure encodes the underlying
relabeling symmetry of the fluid, while the Hamiltonian determines the
specific dynamics within this geometric framework. In particular, the transport nature of (\ref{Intro:Eq:2dEuler-vorticity}) reflects an intrinsic symmetry of the flow:
the vorticity is advected by area-preserving diffeomorphisms, so that its evolution is restricted to the \textit{coadjoint orbit} $\mathcal{O}(\omega_0)=\{\omega_0\circ \Phi, \Phi\in SDiff(\torus^2)\}$, i.e., all possible rearrangements of initial vorticity $\omega_0$ under the group action. As a consequence, the \textit{Casimir functional} of any smooth function $\gamma(\omega)$ taking the form $c_\gamma(\omega)=\displaystyle{\int_{\torus^2}\gamma(\omega(x,t))dx}$ is independent of $t$. In particular, the moments $
\displaystyle{\int_{\mathbb{T}^2} \omega^2 dx} , \dots , \displaystyle{\int_{\mathbb{T}^2} \omega^k}dx, \dots$
are conserved along the flow. These form an infinite family of Casimir invariants for the Lie-Poisson dynamics on $SDiff(\mathbb{T}^2)$, and they determine the coadjoint orbit on which the dynamics evolves.

The Zeitlin's approximation preserves this geometry, rather than merely truncating Fourier modes. Its dynamics yields the conservation of the Casimir quantities taking the form $
\operatorname{tr}(\Omega^2),\ \operatorname{tr}(\Omega^3),\ \dots,\ \operatorname{tr}(\Omega^N)$. Their common level sets determine the corresponding coadjoint orbit
\begin{equation}\label{Intro:Eq:CoadjointOrbitZeitlinEulerArnold}
\mathcal{O}(c_1, ..., c_{N-1}) = \{\operatorname{tr}(\Omega^2)=c_1, \ \operatorname{tr}(\Omega^3) = c_2, \dots , \operatorname{tr}(\Omega^N) = c_{N-1} \} \ .
\end{equation}

For either group $SDiff(\torus^2)$ or $SU(N)$, when restricted to a fixed coadjoint orbit, the Euler-Arnold dynamics becomes a Hamiltonian
flow with respect to the natural symplectic structure (the Kirillov-Kostant-Souriau, or KKS form, see Section \ref{Sec:Background} and the references cited there) and therefore the Liouville measure, while the
Hamiltonian is given by the energy. In this sense, the Casimir invariants determine
the coadjoint orbit, while the Hamiltonian governs the motion within this orbit.

Therefore the Zeitlin approximation provides a finite-dimensional model
that preserves the geometric constraints of the Euler dynamics as well as its Hamiltonian structure, while replacing
the nonlinear transport by a Lie bracket interaction, which makes it a natural
framework for studying flux mechanisms.

In Kraichnan's picture of 2-d turbulence theory, the major concern is that for the 2-d Euler equation (\ref{Intro:Eq:2dEuler-vorticity}) subject to small viscosity $\nu\Delta \omega$, damping term $-\alpha \omega$ and random forcing $f$ concentrated at some intermediate scale, as $\nu\rightarrow 0, \alpha=\alpha(\nu)\rightarrow 0$, one expects the so-called ``inertial range" consisting of a spectrum of Fourier modes within which the nonlinearity $u\cdot \nabla\omega$ is the only dominant term that carries energy and enstrophy between these modes. The average flux of energy and enstrophy remains approximately independent of wavenumber within this inertial range. Given Zeitlin's finite-dimensional approximation scheme via $SU(N)$, a natural first question is to identify possible flux mechanisms for Euler-Arnold system (\ref{Intro:Eq:EulerArnoldZeitlinTruncated}). 

The goal of this work is to provide an initial step towards identifying a possible geometric mechanism within the
Euler-Arnold dynamics on $SU(N)$ that can produce nontrivial fluxes
across scales. This is partly motivated by the problem of constructing mechanical models for cascade, see \cite[Section 15]{[khesin_misiolek_shnirelman_2023_geometric_hydrodynamics_open_problems]}. However, the present work only establishes a local finite-time averaging principle on regular coadjoint strata  for flux dynamics, while convergence of stationary measures and stationary fluxes is not addressed here.

In fact, Kraichnan's theory of 2-d turbulence fundamentally relies on the injection of energy at some intermediate scale together with dissipation mechanisms at large and small scales. Such a regime is inherently non-equilibrium in nature, and the inertial range corresponds to some non-equilibrium steady state (NESS), creating stationary fluxes across that range. This kind of steady state should be a result of mutual interactions between the nonlinearity, the dissipation and the random forcing.
 Mimicking the situation considered in 2-d turbulence, we can add dissipation and forcing term to (\ref{Intro:Eq:EulerArnoldZeitlinTruncated}), to modify it into the following equation
\begin{equation}\label{Intro:Eq:EulerArnoldZeitlinTruncatedDissipatedForced}
\dot{\Omega}=[\Psi,\Omega]-\alpha \Omega + \nu (-A) \Omega + f \ ,
\end{equation}
where $\alpha>0$ is the damping constant, $\nu>0$ is the viscosity constant, $f$ is an external random forcing injecting energy only at some scale $k_f>\!\!>1$, and $(-A)$ plays the role as the Laplace operator (see (\ref{Eq:InertiaOperatorZeitlinCase}) below). In this case, once the unique ergodicity is established (for example, via H\"{o}rmander's bracket computations, see \cite{[HuSverak2018DynamicsGroup]}), we obtain an invariant measure $\mu^{(N)}_{\alpha, \nu}$ under which we can discuss the emergence of nontrivial flux, as first $N\rightarrow \infty$ and then $\nu\rightarrow 0, \alpha=\alpha(\nu)\rightarrow 0$. However, when $N$ is large, the invariant measure $\mu^{(N)}_{\alpha, \nu}$ would be on a high-dimensional space, and accurate computations or estimates with respect to this measure for quantities like energy (or enstrophy) flux are facing the ``curse of dimensionality" problem and become almost mathematically intractable.

To overcome this difficulty, let us go back to the original 2-d Euler equation (\ref{Intro:Eq:2dEuler-vorticity}). From a dynamical point of view, there are at least two roles played by the nonlinear term $u\cdot \nabla \omega$ in (\ref{Intro:Eq:2dEuler-vorticity}):

\begin{itemize}
\item[(A)] Mixing: The nonlinear term \(u \cdot \nabla \omega\) can be viewed as a vector field
on the phase space that couples different Fourier modes.
When combined with the external stochastic forcing, this nonlinear coupling tends to spread
the influence of the noise across a large number of degrees of freedom via H\"ormander-type iterated Lie brackets (see \cite{[Hormander1967]}).
As a result, randomness injected at a limited set of modes via the stochastic forcing may propagate to
other modes over time, and when all modes are reached one arrives at the unique ergodicity of the system (see \cite{[HairerMattingly2006]}, \cite{[HairerMattinglyTheoryHypoellipticitySPDE]}, \cite{[HuSverak2018DynamicsGroup]}). The bracket structure suggests a hierarchy of directions along which
stochastic influence propagates. Directions appearing at lower bracket
depth are natural candidates for faster mixing (compare this with the small-time asymptotics
of hypoelliptic heat kernels e.g.,  \cite{[BenArousHypoellipticCutLocus]},
\cite{[BenArousHypoelliptic]}, \cite{[LeandreDegenrateDiffusion1]}, \cite{[LeandreDegenrateDiffusion2]}).

\item[(B)] Energy and enstrophy transfer: In Fourier basis, the term $u\cdot \nabla \omega$ is written as $3$-mode (triadic) interaction via the pattern $k+l=m$ (see equation (\ref{Eq:2dEulerFourierComponents}) below). This means the nonlinear term can transfer energy and enstrophy between these interacting modes, and when combined with the forcing term and the dissipation/damping terms this creates complicated energy and enstrophy flows between different Fourier modes. Unlike the rapid propagation of stochastic influence described above, this
energy and enstrophy transfer mechanism is not purely mixing. It reflects a subtle statistical
bias arising from the asymmetric interplay between forcing, nonlinear coupling,
and dissipation, which is expected to be responsible for the emergence of inertial-range behavior.
\end{itemize} 

During the evolution of fluid flows, the above two effects take place simultaneously, but analyzing them together seems to be very hard, even for the finite dimensional model systems like (\ref{Intro:Eq:EulerArnoldZeitlinTruncatedDissipatedForced}). Heuristically, the net inter-scale flux is a delicate third-order quantity arising from a small statistical bias among many cancelling triad contributions from the nonlinear term, and therefore may evolve on a slower time scale, when compared with the fast mixing created by the interaction of the nonlinear term with random forcing. This separation of time scales becomes visible in the turbulent regime (high Reynolds number), where strong nonlinearity and large scale separation allow rapid phase mixing to coexist with a slowly emerging net flux selected by forcing and dissipation.
This means at the turbulent regime we might expect to see a separation of scale between these two effects: mixing is happening much faster than energy transfer. This is only a physical heuristic, but it seems to have been supported by some computational theories of $2$-d flows (see for example the theory of stochastic superparametrization \cite{[GroomsMajdaSuperParametrization]}) and partly from the theory of inviscid damping of $2$-dimensional flows (see \cite{[BedrossianMasmoudiInviscidDamping]}).

If we can assume, in the turbulence regime,  that (A) Mixing happens much faster than (B) Energy and enstrophy transfer, then we can model (A) using a fast noise. This means, at the finite-dimensional level, we can investigate a system modified from (\ref{Intro:Eq:EulerArnoldZeitlinTruncatedDissipatedForced}), which is schematically written as
\begin{equation}\label{Intro:Eq:EulerArnoldZeitlinTruncatedDissipatedForcedWithFastNoise}
\dot{\Omega}=\dfrac{1}{\varepsilon} \cdot \mathscr{F} + [\Psi, \Omega]-\alpha \Omega + \nu(-A) \Omega + f \ , \ 0<\varepsilon <\!\!<1 \ ,
\end{equation}
where $\mathscr{F}$ is an additional stochastic forcing term that represents the fast mixing created by the interaction between the nonlinearity $[\Psi, \Omega]$ and the external stochastic forcing $f$. In this case, one can discuss the limit of the invariant measures $\mu^{(N)}_{\varepsilon, \alpha, \nu}$ as $\varepsilon \rightarrow 0$, denoted as $\mu^{(N)}_{0, \alpha, \nu}$. At this limit measure $\mu^{(N)}_{0, \alpha, \nu}$, we are able to focus on the effect of (B) which directly connects to mechanisms of fluxes, without facing the difficulty of simultaneously dealing with (A) and (B). We can further estimate, from $\mu^{(N)}_{0, \alpha, \nu}$, as first $N\rightarrow\infty$ and then $\nu\rightarrow 0, \alpha=\alpha(\nu)\rightarrow 0$, the fluxes that are related to the inertial range phenomenon via a more approachable mathematics, where the key question is whether a structured fast stochastic forcing can isolate the slow energy and enstrophy transfer mechanism.

From the point of view of asymptotic analysis, the above limiting procedure can be viewed as a fast-slow averaging (homogenization) procedure (see \cite{[FWBook]}), which is essentially some kind of model reduction that helps to circumvent the difficulty arising from ``curse of dimensionality" as $N\rightarrow\infty$. Since we work with a finite-dimensional Lie group $SU(N)$, our proposed system (\ref{Intro:Eq:EulerArnoldZeitlinTruncatedDissipatedForcedWithFastNoise}) with fast stochastic forcing can also be regarded as a toy model in Stochastic Geometric Mechanics (see \cite{[BismutGeometricMechanics]}, \cite{[Cami-OrtegaStochasticMechanics]}), and so it may be of independent interest to investigate.

From the above discussion of (A) Mixing, the fast stochastic forcing term $\mathscr{F}$ in (\ref{Intro:Eq:EulerArnoldZeitlinTruncatedDissipatedForcedWithFastNoise}) itself may have separation of scales based on the hierarchy and interaction coefficients calculated from Lie bracket operations. These Lie brackets are computed iteratively between the nonlinear term $[\Psi, \Omega]$ and the external stochastic forcing term $f$. However,
deriving the precise architecture of scale separation in $\mathscr{F}$ via iterated Lie brackets appears to be hard to track. Instead of attempting such a derivation, we adopt a
geometric viewpoint based on the coadjoint orbit of the Euler-Arnold equation. Actually, $\mathscr{F}$ can be understood as an effective closure of the rapid stochastic mixing generated by the interaction between $[\Psi, \Omega]$ and $f$, and the coadjoint geometry can be taken as the organizing principle for this closure. This is because the Euler-Arnold nonlinearity $[\Psi, \Omega]$ gives a vector field tangent to coadjoint orbits, and the first few brackets that it has with $f$ will create directions not far away from the tangential directions to the coadjoint orbit. We therefore model the leading fast mixing effect by stochastic vector fields tangent to coadjoint orbits, plus small transversal perturbations to these tangential directions. That is to say we assume
\begin{equation}\label{Intro:Eq:TangentialFastStochastic}
\mathscr{F}dt=dX+\rho \varepsilon dY \ ,
\end{equation}
where $X$ is tangent to the coadjoint orbit and $Y$ is allowed to be transversal to it. To ensure unique ergodicity, the small perturbation parameter $\rho>0$ is chosen with $\rho^2 \sim  O(\nu+\alpha)$, i.e., $\rho^2$ has the same order as the viscosity and damping parameters (see Theorem \ref{Thm:ErgodicityEulerArnoldZeitlinTruncatedDissipatedForcedWithFastNoise} in Section \ref{Sec:AveragingFramework:EulerArnoldMultiscale}). As we have seen, in the Euler-Arnold formulation, the phase space carries a natural decomposition
induced by symmetry. The coadjoint orbit represents the set of configurations
related by relabeling transformations, i.e., by rearrangements of vorticity
that preserve its distribution. Therefore, modeling fast mixing along coadjoint directions
can be viewed as a stochastic representation of the intrinsic transport
mechanism of the 2-d Euler dynamics (compare with \cite{[HuSverak2018DynamicsGroup]}). However, since numerical and experimental studies of $2$-d turbulence suggest that higher-order Casimirs are effectively dissipated through small-scale mixing (see, e.g. \cite{[BouchetVenaille2012Statistical2dGeophysical]}), including only the tangential fast stochastic forcing direction $X$ in (\ref{Intro:Eq:TangentialFastStochastic}) may be too ideal, while adding the transversal perturbations $\rho \varepsilon Y$ to some extent remedies this issue.

Within this effective coadjoint-tangent closure, there are two different levels of symmetry. The most symmetric choice is to take the $X$ vector to consist of Hamiltonian vector fields on the coadjoint orbit. Such fields preserve the KKS symplectic form and hence the corresponding volume measure $m_{\mathrm{KKS}}$ (which is the normalized Liouville measure) on the orbit. When $X$ propagates stochasticity to all directions on the coadjoint orbit, this gives rapid thermalization with respect to the natural orbit ensemble. Such a regime represents the idealized situation in which the fast stochastic forcing direction $X$ produces complete coadjoint mixing without introducing any statistical bias inside the orbit. It turns out, that under this full symmetry, we can show (see Section \ref{Sec:HeatingUpAllCoAdjoint}) that there will be no nontrivial flux in the averaged system carried by the Euler-Arnold nonlinearity $[\Psi, \Omega]$. 

The next level of symmetry preserves the coadjoint orbit but breaks the invariance of the KKS form. This means that the $X$-induced tangent dynamics preserves some volume measure $m_{\mathrm{KKS}}^{\kappa}$ on the coadjoint orbit with $d m_{\mathrm{KKS}}^{\kappa}(\Omega)=(1+\kappa q(\Omega)+O(\kappa^2)) \cdot dm_{\mathrm{KKS}}(\Omega)$, $0<\kappa<\!\!<1$ for some $q(\Omega)$ integrating to $0$ on the coadjoint orbit. Such a situation can be created from tangential non-Hamiltonian vector $X$, which encodes the statistical bias produced by forcing and dissipation while remaining consistent with the idea that the fast motion is primarily a coadjoint-tangent rearrangement. We emphasize that this symmetry-breaking mechanism is not caused by the transversal perturbation $\rho \varepsilon Y$ in (\ref{Intro:Eq:TangentialFastStochastic}), but by the fact that the effective tangent dynamics $X$ need not preserve the KKS form. We will show in Section \ref{Sec:SymmetryBreakingCoAdjoint} conditions that can yield nontrivial flux produced in this symmetry-breaking situation.

The paper is organized as follows: In Section \ref{Sec:Background} we provide background information regarding the Euler-Arnold equation on $SU(N)$ and Zeitlin's approximation of $2$-d fluid flows. In Section \ref{Sec:GeometryAlgebraSUN} we further develop some geometry and algebra regarding $SU(N)$ coadjoint orbits. In Section \ref{Sec:AveragingFramework} we introduce our fast-slow stochastic system and its coadjoint averaging framework, with discussions on different notions of fluxes that we consider. In Section \ref{Sec:HeatingUpAllCoAdjoint} we show that there is no nontrivial flux carried by the nonlinearity in the fully symmetric coadjoint averaging case. Finally in Section \ref{Sec:SymmetryBreakingCoAdjoint} we discuss the symmetry breaking case and see how it can develop nontrivial flux.

\section{Background: Euler-Arnold equation on $SU(N)$ as an approximation of 2-dimensional fluid flows}\label{Sec:Background}

\subsection{2-d Euler and damped-forced Navier-Stokes equations}\label{Sec:Background:2dEulerNSE}

Consider the $2$-d Navier-Stokes equation on a torus $\mathbb{T}^2=(\mathbb{R}/2\pi \mathbb{Z})^2$ written in the following vorticity form
\begin{equation}\label{Eq:2d-ns-damped-forced:vorticityform}
\dfrac{\partial \omega}{\partial t} + u\cdot \nabla \omega = \nu \Delta \omega - \alpha \omega + f \ .
\end{equation}
Here $\omega=\omega(x,t), x\in \mathbb{T}^2, t\geq 0$ is the vorticity and $u= \mathcal{K} \omega$ is the corresponding velocity field. The Biot-Savart operator $\mathcal{K}=\nabla^\perp \Delta^{-1}$ with $\nabla^\perp = \left(-\dfrac{\partial}{\partial x_2}, \dfrac{\partial}{\partial x_1}\right)$. The viscosity constant $\nu>0$ and the linear damping constant $\alpha>0$. The forcing term $f=f(x,t)$ is stochastic and concentrated at an injection scale $k_f>\!\!> 1$. Define
\begin{equation}\label{Eq:Z2FullSpaceLattice}
\mathcal{K}_{\mathbb{Z}^2}= \mathbb{Z}^2\backslash\{(0,0)\} \ ,
\end{equation}
and
\begin{equation}\label{Eq:Z2HalfSpaceLattice}
\mathcal{K}_{\mathbb{Z}^2}^+=\{k=(k_1, k_2)\in \mathcal{K}_{\mathbb{Z}^2}, k_1>0 \text{ or } (k_1=0, k_2>0)\} \ .
\end{equation}
Then formally, 
\begin{equation}\label{Eq:2d-ns-damped-forced:forcing}
f(x,t)=\sum_{k\in \mathcal{K}_{\mathbb{Z}^2}^+, |k|\approx k_f} q_k
\left[
\cos(k\cdot x)\,\dot\beta_t^{(k,\cos)}
+
\sin(k\cdot x)\,\dot\beta_t^{(k,\sin)}
\right] \ ,
\end{equation}
where for each $k$,  $\beta_t^{(k,\cos)} , \beta_t^{(k,\sin)}$ are independent
standard real Brownian motions defined on the filtered probability space $(\mathbf{\Omega}, \mathcal{F}_t, \mathbf{P})$ and \(q_k\in\mathbb R\) are forcing amplitudes such that $0\leq \sum\limits_{k\in \mathcal{K}_{\mathbb{Z}^2}^+ , |k|\approx k_f} q_k^2<\infty$.

In the special case when $\nu=\alpha=f\equiv 0$, equation (\ref{Eq:2d-ns-damped-forced:vorticityform}) becomes the $2$-d Euler equation written in vorticity form
\begin{equation}\label{Eq:2dEuler-vorticity}
\dfrac{\partial \omega}{\partial t} + u \cdot \nabla\omega = 0 \ , \ u = \nabla^\perp \Delta^{-1}\omega \ .
\end{equation}
If we write the solution $\omega(x,t)$ for the above 2-d Euler equation in Fourier modes, i.e., $\omega(x,t) = \sum\limits_{k \in \mathbb{Z}^2\setminus\{(0,0)\}} \omega_k(t)e^{ik\cdot x}$, then the dynamics of (\ref{Eq:2dEuler-vorticity}) takes the form (see \cite[equation (2.2)]{[HairerMattingly2006]}):
\begin{equation}\label{Eq:2dEulerFourierComponents}
\dot{\omega}_m=\sum\limits_{k, l, k+l=m} \dfrac{m\times k}{|k|^2}\omega_k\omega_l \ .
\end{equation}
From the above, we see that in terms of Fourier modes, the nonlinear term $u\cdot \nabla \omega$ in (\ref{Eq:2dEuler-vorticity}) is a convolution over triads satisfying $k+l=m$,
which reflects the fundamental nonlinear interaction structure of 2-dimensional Euler flows.

\subsection{Energy and enstrophy fluxes in Kraichnan's 2-d turbulence theory}\label{Sec:Background:Kraichnan2dTurbulence}

Ergodic properties of this system (\ref{Eq:2d-ns-damped-forced:vorticityform}) have been discussed a lot in the literature (see, for example \cite{[HairerMattingly2006]}, \cite{[HairerMattinglyTheoryHypoellipticitySPDE]}). These works focus on identifying the structure of the random forcing (\ref{Eq:2d-ns-damped-forced:forcing}), i.e., which modes need to be randomly forced, that can make the system (\ref{Eq:2d-ns-damped-forced:vorticityform}) ergodic and admit a unique invariant measure $\mu_{\nu, \alpha}$. Assuming in advance that system (\ref{Eq:2d-ns-damped-forced:vorticityform}) is already uniquely ergodic with stationary measure $\mu_{\nu, \alpha}$, Kraichnan proposed a phenomenological 2-d turbulence theory (see \cite{[Kraichnan1967]}) regarding the transfer of energy and enstrophy between different Fourier modes as the viscosity $\nu \rightarrow 0$ and the corresponding damping coefficient $\alpha=\alpha_\nu\rightarrow 0$ that scale with $\nu$. Assume $\omega(x, t), x\in \mathbb{T}^2, t\geq 0$ is the solution to (\ref{Eq:2d-ns-damped-forced:vorticityform}) with $\omega(x,0)$ distributed as $\mu_{\nu, \alpha}$, so that by stationarity for each $t\geq 0$ the solution $\omega(x,t)$ is also distributed 
as $\mu_{\nu, \alpha}$. Let the corresponding stationary velocity field be $u(x, t)=\nabla^\perp \Delta^{-1} \omega(x,t)$. If we write $\omega$ and $u$ in Fourier modes, then we get
$$\omega(x, t)=\sum\limits_{k\in \mathbb{Z}^2\backslash\{(0,0)\}} \omega_k(t) e^{ik \cdot x} \ , \ u(x, t)=\sum\limits_{k\in \mathbb{Z}^2\backslash\{(0,0)\}} u_k(t) e^{ik \cdot x} \ .$$
Notice that for each Fourier mode $k\neq 0$, the Biot-Savart law gives
$
u_k = -\dfrac{i\,k^\perp}{|k|^2}\omega_k$, $k^\perp = (-k_2,k_1)$, or equivalently,
$
\omega_k = i\,k^\perp \cdot u_k$.

Based on this, the energy and enstrophy spectra are defined by grouping modes with the same magnitude $|k|$:

\begin{equation}\label{Eq:Kraichnan2dTurbulence:EnergyEnstropySpectrum}
\mathcal{E}(n, t)=\dfrac{1}{2}\sum\limits_{|k|=n}|u_k(t)|^2 \ , \ \mathcal{Z}(n, t)=\dfrac{1}{2}\sum\limits_{|k|=n}|\omega_k(t)|^2 \ .
\end{equation}

The cumulative low-frequency energy and enstrophy are then defined by 
\begin{equation}\label{Eq:Kraichnan2dTurbulence:CumulativeLowFreqEnergyEnstropy}
\mathcal{E}_{\leq K}(t) = \dfrac{1}{2}\sum\limits_{|k|\leq K}|u_k(t)|^2 \ , \ \mathcal{Z}_{\leq K}(t)=\dfrac{1}{2}\sum\limits_{|k|\leq K}|\omega_k(t)|^2 \ .
\end{equation}

Since $\omega$ and $u$ are already distributed at stationarity, the total energy $\mathcal{E}=\sum\limits_{n\geq 1}\mathcal{E}(n,t)$ and enstrophy $\mathcal{Z}=\sum\limits_{n\geq 1}\mathcal{Z}(n,t)$ should remain constant during the evolution of (\ref{Eq:2d-ns-damped-forced:vorticityform}). However, this does not mean that for each $K$, the quantities $\mathcal{E}_{\leq K}(t)$ and $\mathcal{Z}_{\leq K}(t)$ also remain constant. In fact, since the linear damping term $-\alpha \omega$ takes effect mostly for low-frequency dissipation, the viscosity term $\nu \Delta \omega$ takes effect mostly for high-frequency dissipation, while the forcing term $f$ injects energy only around an intermediate scale $k_f$, Kraichnan argued that there exist \textit{inertial ranges} $k\in [k_\alpha, k_f]$ and $k\in [k_f, k_\nu]$ with $k_\alpha=k_\alpha(\alpha_\nu)<\!\!< k_f <\!\!< k_\nu(\nu)=k_\nu$ within which the energy and enstrophy are only transferred by the nonlinear term $u\cdot \nabla \omega$ in (\ref{Eq:2d-ns-damped-forced:vorticityform}). Let the (stationary) \textit{flux} of energy from $\mathcal{E}_{\leq K}(t)$ to $\mathcal{E}_{>K}(t)=\mathcal{E}-\mathcal{E}_{\leq K}(t)$ be defined by $\Pi_\mathcal{E}(K)$, and the (stationary) flux of enstrophy from $\mathcal{Z}_{\leq K}(t)$ to $\mathcal{Z}_{>K}(t)=\mathcal{Z}-\mathcal{Z}_{\leq K}(t)$ be defined by $\Pi_\mathcal{Z}(K)$. Then by using the equation (\ref{Eq:2d-ns-damped-forced:vorticityform}) and directly differentiating (\ref{Eq:Kraichnan2dTurbulence:CumulativeLowFreqEnergyEnstropy}), we get the (stationary) energy flux
\begin{equation}\label{Eq:Kraichnan2dTurbulence:EnergyFlux}
\Pi_\mathcal{E}(K)
:= - \mathbf{E}_{\mu_{\nu,\alpha}} \sum_{|k|\le K}
\Re \Big( \widehat{(u\cdot\nabla u)}_k \cdot \overline{u_k} \Big)
\end{equation}
and the (stationary) enstrophy flux
\begin{equation}\label{Eq:Kraichnan2dTurbulence:EnstrophyFlux}
\Pi_\mathcal{Z}(K)
:= - \mathbf{E}_{\mu_{\nu,\alpha}}
\sum_{|k|\le K}
\Re \big(
\widehat{(u\cdot\nabla \omega)}_k
\overline{\omega_k}
\big) \ .
\end{equation}
Here $\widehat{(u\cdot\nabla u)}_k$ and $\widehat{(u\cdot\nabla \omega)}_k$ are the $k$-th Fourier coefficients of $u\cdot\nabla u$ and $u\cdot\nabla \omega$, respectively.

Based on further physical assumptions in the regime when $\nu\rightarrow 0$, such as local interactions, meaning that the dominant contribution to the flux across scale $k$
comes from triad interactions in (\ref{Eq:2dEulerFourierComponents}) among comparable wavenumbers, the fact that we run (\ref{Eq:2d-ns-damped-forced:forcing}) at the stationary distribution $\mu_{\nu, \alpha}$ then implies that in this case as $\nu \rightarrow 0$, $\Pi_\mathcal{E}(K)$ should be independent of $K$ when $K\in [k_\alpha, k_f]$, and $\Pi_\mathcal{Z}(K)$ should be independent of $K$ when $K\in [k_f, k_\nu]$. Moreover, Kraichnan also conjectures a dual cascade: an inverse energy cascade towards low Fourier modes (large scales) and a direct enstrophy cascade towards high Fourier modes (small scales), driven by the nonlinear interactions. This leads to the conjectured behavior of $\Pi_\mathcal{E}(K)$ and $\Pi_\mathcal{Z}(K)$ as
\begin{equation}\label{Eq:Kraichnan2dTurbulence:FluxLawDualCascade:Energy}
\lim\limits_{\nu\rightarrow 0, \alpha=\alpha_\nu \rightarrow 0} \Pi_\mathcal{E}(K)=-\varepsilon <0 \text{ when } K\in [k_\alpha, k_f]
\end{equation} 
and
\begin{equation}\label{Eq:Kraichnan2dTurbulence:FluxLawDualCascade:Enstrophy}
\lim\limits_{\nu\rightarrow 0, \alpha=\alpha_\nu \rightarrow 0} \Pi_\mathcal{Z}(K)=\eta >0 \text{ when } K\in [k_f, k_\nu] \ .
\end{equation} 
As of the date when this manuscript is prepared, as far as the author is aware of, these physical predictions of 2-d turbulence have not been fully established from a rigorous mathematical perspective. In fact, necessary and sufficient conditions for the emergence of cascade flux laws have been discussed, which are proved to be equivalent conditions of the 2-d turbulence predictions. We refer to \cite{[BedrossianCascadeLawsSchrodinger]}, \cite{[DudleyFluxLaws]}, \cite{[BedrossianZelatiPaulFluxSufficient]}   
for some recent references on the mathematical treatment of the flux law problem. 

\subsection{Euler-Arnold equation on a general Lie group $G$ and its application to fluid motions}\label{Sec:Background:EulerArnoldGeneralLieGroup}

The 2-d Euler equation (\ref{Eq:2dEuler-vorticity}) admits a geometric interpretation that reflects its underlying Hamiltonian and Lie-Poisson structure, which has been introduced in the pioneering works of Arnold \cite{[Arnold1966JFourier]} (see also \cite{[Arnold-KhesinBook]}, \cite{[ArnoldClassicalMechanics]}, \cite{[MarsdenWeistein1983]}). In fact, this equation admits a geometric formulation as an Euler-Arnold equation on the
group of volume-preserving diffeomorphisms of $\mathbb{T}^2$, written as  $SDiff(\mathbb{T}^2)$. 

Let us start from the Euler-Arnold equation for a general real Lie group $G$ (see \cite{[Arnold1966JFourier]}). It takes the form
\begin{equation}\label{Eq:EulerArnoldEquation}
\dfrac{dM_c}{dt}=ad^*_{\omega_c} M_c \ ,
\end{equation}
where $\omega_c=L_{g^{-1}*}\dot{g}\in \mathfrak{g}$ is the ``angular velocity in the body", and $M_c=A \omega_c\in \mathfrak{g}^*$ is the ``angular momentum in the body". Here $\mathfrak{g}$ is the Lie algebra of $G$ and $\mathfrak{g}^*$ is its dual space, $g=g(t)$ is a trajectory on the group $G$ and $\dot{g}$ is its derivative, $L_{g^{-1}*}$ is the induced mappings on $TG_g$ of the left action $h\rightarrow g^{-1}h, h\in G$, and $A: \mathfrak{g} \rightarrow \mathfrak{g}^*$ is a positive-definite linear mapping representing the moment of inertia. The right-hand side of (\ref{Eq:EulerArnoldEquation}) is defined by 
\begin{equation}\label{Eq:EulerArnoldEquationCoadjointOperator} (ad_{\omega_c}^* M_c, \zeta)=(M_c, [\omega_c, \zeta]) \ , \ \zeta\in \mathfrak{g} \ ,
\end{equation}
where $(\bullet, \bullet)$ is the pairing between $\mathfrak{g}^*$ and $\mathfrak{g}$, and $[\bullet, \bullet]$ is the Lie bracket on $\mathfrak{g}$. From the point of view of geometry, the Euler-Arnold equation (\ref{Eq:EulerArnoldEquation}) defines a geodesic curve $g(t)$ on $G$ equipped with the left-invariant metric given by the matrix $A$ (see \cite{[TaoEuler-Arnoldblog]}, \cite{[HuSverak2018DynamicsGroup]}). 

It turns out that, for $G=SDiff(\torus^2)$, the Lie algebra $\mathfrak{g} = S_0\text{Vect}(\torus^2)$ consists of all divergence-free vector fields
$u$ on the torus with single-valued stream functions. With the further assumption that $A=\mathrm{id}$ one can derive that in this case $M_c=\omega_c=u$ and the Euler-Arnold equation of $SDiff(\torus^2)$ can be written as 
\begin{equation}\label{Eq:EulerArnoldSDiffTorus}
\dfrac{\partial \omega}{\partial t}+\{\psi, \omega\} = 0 \ , 
\end{equation}
where $\psi$ is the stream function such that $\omega = \Delta \psi$, and the Poisson bracket
\begin{equation}\label{Eq:EulerArnoldPoissonBracket}
\{\psi,\omega\}
=
\dfrac{\partial \psi}{\partial x_1}\dfrac{\partial \omega}{\partial x_2}
-
\dfrac{\partial \psi}{\partial x_2}\dfrac{\partial \omega}{\partial x_1} \ .
\end{equation}
It is easy to check that (\ref{Eq:EulerArnoldSDiffTorus}) is equivalent to the original Euler's equation on $\torus^2$. Thus the Euler's equation (\ref{Eq:2dEuler-vorticity}) characterizes a geodesic curve on $SDiff(\torus^2)$ with the identity metric. 

For a general real Lie group $G$ we consider the diffeomorphism
$
R_{g^{-1}}L_g : G\to G  \ , \ h\mapsto ghg^{-1}$,
which is the inner automorphism of the group $G$. This diffeomorphism $R_{g^{-1}}L_g$ preserves the identity element $e$, and its
derivative at $e$ defines the \textit{adjoint representation}
\begin{equation}\label{Eq:AdjointRepresentation}
Ad_g:\mathfrak g\to\mathfrak g,
\qquad
Ad_g=(R_{g^{-1}}L_g)_{*e} .
\end{equation}
It is
natural to define the dual operator $Ad_g^*: \mathfrak{g}^*\rightarrow \mathfrak{g}^*$ by the
identity \begin{equation}\label{Eq:CoAdjointRepresentation}
(Ad_g^*\xi, \eta)=(\xi, Ad_g\eta) \ , \ \xi\in \mathfrak{g}^* \ , \  \eta\in \mathfrak{g} \ ,
\end{equation}
which is the \textit{coadjoint operator} defining the \textit{coadjoint representation} of
the group $G$. Physically, $Ad_g^*\eta$ describes how the momentum $\eta \in \mathfrak{g}^*$ is transformed after the configuration is changed by $g\in G$. It can be proved, by (\ref{Eq:EulerArnoldEquationCoadjointOperator}), (\ref{Eq:AdjointRepresentation}) and (\ref{Eq:CoAdjointRepresentation}), that 
\begin{equation}\label{Eq:RelationAdStaradstar}
\left.\dfrac{d}{dt}\right|_{t=0} Ad^*_{e^{t\xi}}= ad_\xi^* \ , \ \xi\in\mathfrak{g} \ .
\end{equation}

Thus, starting from a given momentum state $\eta\in\mathfrak g^*$,
one obtains all equivalent momentum states by conjugation with group
elements. This leads naturally to the \textit{coadjoint orbit}
\begin{equation}\label{Eq:CoAdjointOrbit}
\mathcal O^*(\eta)=\{Ad_g^*\eta,\ g\in G\} .
\end{equation}

Let $G_\eta = \{g\in G: Ad_g^*\eta=\eta\} $
be the stabilizer subgroup of $\eta$, then $\mathcal O^*(\eta)
\cong
G/G_\eta$, i.e. each coadjoint orbit is a homogeneous space of the Lie group $G$.
If $G$ is compact, then $G_\eta$ is compact as well, and therefore the coadjoint orbit $\mathcal O^*(\eta)$
is a compact smooth manifold.

It turns out that $\mathcal{O}^*(\eta)$ is also a symplectic manifold on which there is a canonical $2$-form, called the Kirillov-Kostant-Souriau (KKS) form (see \cite{[ArnoldClassicalMechanics]}, \cite{[KirillovOrbit]}, \cite{[KirillovRepresentationBook]}). Let $\mu\in \mathcal O^*(\eta)$ be written as $\mu=Ad_g^*\eta \in \mathfrak{g}^*$. Using (\ref{Eq:RelationAdStaradstar}), the tangent space is
$T_\mu\mathcal O^*(\eta)
=
\{ad_\xi^*\mu:\xi\in\mathfrak g\}$. Then the KKS form is defined by
\begin{equation}\label{Eq:KKSForm}
\omega_{\mathrm{KKS}}
\left(ad_\xi^*\mu,ad_\zeta^*\mu\right)
=
\left( \mu,[\xi,\zeta]\right) \ , \ 
\xi,\zeta \in \mathfrak{g} \ .
\end{equation}

If $\dim \mathcal O^*(\eta)=2d$,
then the Liouville volume measure associated with the KKS form is given by 
\begin{equation}\label{Eq:LiouvilleMeasureKKSForm}
d\widehat{m}_{\mathrm{KKS}}
\stackrel{\text{def}}{=}
\dfrac{1}{d!}\underbrace{\omega_{\mathrm{KKS}}\wedge...\wedge\omega_{\mathrm{KKS}}}_{d \text{ times}} \ .
\end{equation}
Since the KKS form is invariant under the coadjoint action, the Liouville measure is $G$-invariant. On the other hand, the Haar measure on $G$ induces a natural quotient measure
on the homogeneous space $G/G_\eta$. If $G$ is compact, since compact homogeneous spaces admit a unique $G$-invariant measure up to normalization,
the Liouville measure agrees, after normalization, with the Haar push-forward measure
under the orbit map $g\mapsto Ad_g^*\eta$. Equivalently, for every continuous observable $F$, we have
$$
\int_{\mathcal O^*(\eta)}
F(\mu) dm_{\mathrm{KKS}}(\mu)
=
\int_G
F(Ad_g^*\eta) dg_{\mathrm{Haar}} \ ,
$$
where $dm_{\mathrm{KKS}}$ is the normalized Liouville measure.

Given the positive-definite linear mapping $A: \mathfrak{g} \rightarrow \mathfrak{g}^*$ defining the moment of inertia, we introduce the Hamiltonian as
\begin{equation}\label{Eq:CoAdjointHamiltonian}
\mathscr{H}: \mathcal{O}^*(\eta)\rightarrow \mathbb{R} \ , \ \mathscr{H}(M_c)=\dfrac{1}{2}(M_c, A^{-1}M_c) \ .
\end{equation}
It turns out that the Euler-Arnold equation (\ref{Eq:EulerArnoldEquation}) with $M_c(0)=\eta\in \mathfrak{g}^*$ is a Hamiltonian equation on the coadjoint orbit $\mathcal{O}^*(\eta)$. That is to say, if $M_c(0)=\eta$, then $M_c(t)\in \mathcal{O}^*(\eta)\cap \{x\in \mathfrak{g}^*: \mathscr{H}(x)=\mathscr{H}(\eta)\}$ for all $t\geq 0$ (see \cite{[TaoEuler-Arnoldblog]}). 
Define the Hamiltonian vector field $X_{\mathscr{H}}$ associated with the Hamiltonian $\mathscr{H}$ by
\begin{equation}\label{Eq:HamiltonianEqViaKirillovForm}
\omega_{\mathrm{KKS}}(\xi, X_\mathscr{H}) = (d\mathscr{H}, \xi) \ , \ \text{for every vector } \xi \text{ tangent to } \mathcal{O}^*(\eta) \ .
\end{equation}
Then the Hamiltonian flow on $\mathcal{O}^*(\eta)$ generated by this vector field agrees with the flow generated by the Euler-Arnold
equation (\ref{Eq:EulerArnoldEquation}). In this sense, the coadjoint orbit should be regarded as the intrinsic symplectic phase space
for the momentum dynamics.

It turns out that for 2-dimensional incompressible flows (for example on
$\sphere^2$), see \cite{[ModinViviani2026ARMA]}), the coadjoint orbit associated with an initial
vorticity $\omega_0$ admits a natural fluid-dynamical interpretation.
If $\Phi_t$ denotes the Lagrangian flow map generated by the velocity
field, then the vorticity evolves by transport $\omega_t=\omega_0\circ\Phi_t^{-1}$. Hence the coadjoint orbit through $\omega_0$ consists precisely of all
vorticity fields obtained from $\omega_0$ by volume-preserving
rearrangements. The same interpretation also holds on $\torus^2$, where the
Euler flow evolves on coadjoint orbits of the volume-preserving
diffeomorphism group $SDiff(\mathbb{T}^2)$.

\subsection{Zeitlin's $SU(N)$ approximation}\label{Sec:Background:ZeitlinApproximation}

On $\mathbb{T}^2$, one can ``identify" the stream function $\psi$ with the corresponding divergence-free
vector field $u$, and such a stream function can be assumed to have to have zero mean, i.e., $\displaystyle{\int_{\torus^2}\psi(x)dx=0}$. This gives an identification 
\begin{equation}\label{Eq:S0VectT2IdentificationST2}
S_0\text{Vect}(\torus^2)\cong S(\torus^2)
\end{equation}
of the Lie algebra $S_0\text{Vect}(\torus^2)$ with all mean-zero functions 
\begin{equation}\label{Eq:TorusMeanZeroFunctionsAsLieAlgebraSDiff}
S(\torus^2)\stackrel{\text{def}}{=} \left\{\psi: \torus^2\rightarrow \mathbb{R}: \displaystyle{\int_{\torus^2}\psi(x)dx=0}\right\} \ .
\end{equation}
Given this identification, the basis of the complexified Lie algebra $S_0\text{Vect}(\torus^2; \mathbb{C})$ can be established as the Fourier basis $L_k$ in $S(\torus^2; \mathbb{C})=S(\torus^2)\oplus iS(\torus^2)$:
\begin{equation}\label{Eq:FourierBasis2dTorus}
L_k=e^{i k\cdot x} \ , 
\end{equation}
where $k=(k_1,k_2)\in \mathbb{Z}^2\backslash \{(0,0)\}$, and
whose value at a point $(x_1,x_2)$ is $L_k=e^{i(k_1x_1+k_2x_2)}$. The Lie bracket of $S_0\text{Vect}(\torus^2; \mathbb{C})$ can be computed directly by definition from $SDiff(\mathbb{T}^2)$, and under the Fourier basis the structural constants are actually $k\times l$, i.e.
\begin{equation}\label{Eq:CommutatorSVectT2}
[L_k, L_l]=(k\times
l)L_{k+l}
\end{equation} where $k\times l=k_1l_2-k_2l_1$ is the oriented
area of the parallelogram spanned by $k$ and $l$. We then see that by (\ref{Eq:EulerArnoldPoissonBracket}), \begin{equation}\label{Eq:RelationLieBracketS0VectT2andPossionBracket}
[L_k, L_l]=-\{L_l, L_k\}
\end{equation} 
and the Euler-Arnold equation (\ref{Eq:EulerArnoldSDiffTorus}) on $SDiff(\torus^2)$ can be written as

\begin{equation}\label{Eq:EulerArnoldSDiffTorus:LieBracket}
\dfrac{\partial \omega}{\partial t}=[\psi, \omega] \ .
\end{equation}

In \cite{[Zeitlin1991]}, Zeitlin proposed an approximation of the above Lie bracket relation (\ref{Eq:CommutatorSVectT2}) via the Lie algebra of $SU(N)$ as $N\rightarrow\infty$. By definition we have 
\begin{equation}\label{Eq:SUN}
SU(N) = \{A\in \mathbb{C}^{N\times N}: A^*A=I \ , \ \det A = 1\} \ ,
\end{equation}
where $A^*=\overline{A}^{T}$. Although the matrices in $SU(N)$ have complex entries, the group $SU(N)$ is a real Lie group rather than a complex Lie group. Indeed, it is a smooth submanifold of $\mathbb{C}^{N\times N}\cong \mathbb{R}^{2N^2}$.
Its Lie algebra
\begin{equation}\label{Eq:suN}
\mathfrak{su}(N)=\{X\in\mathbb{C}^{N\times N}:X^\ast=-X,\ \operatorname{tr}(X)=0\}
\end{equation}
is a \textit{real} vector space of dimension $N^2-1$, not a complex vector space.

Actually, for the complexification $\mathfrak{su}(N;\mathbb{C})=\mathfrak{su}(N)\oplus i \mathfrak{su}(N)$, we have $\mathfrak{su}(N; \mathbb{C})=\mathfrak{sl}(N;\mathbb{C})$ since for any $A\in \mathfrak{sl}(N; \mathbb{C})$ we have $A=X+iY$ where $X=\dfrac{A-A^*}{2}\in \mathfrak{su}(N)$ and $Y=\dfrac{A+A^*}{2i}\in \mathfrak{su}(N)$. 
The commutation relations in algebra $\mathfrak{su}(N;
\mathbb{C})$ ``approximate" (\ref{Eq:CommutatorSVectT2}) in the following way.
Fix some odd $N$ and consider the following two matrices on
$\mathfrak{sl}(N; \mathbb{C})$:
\begin{equation}\label{Eq:ZeitlinBasisFH}
F=\text{diag}(1,\omega_N,...,\omega_N^{N-1}) \ , \
H=\begin{pmatrix}
0&1&...&0&0\\
0&0&1&...&0\\
0&0&0&...&0\\
...\\
0&0&0&...&1\\
1&0&0&...&0
\end{pmatrix}_{N\times N}
\end{equation} 
where $\omega_N$ is a primitive $N$-th root of unity and may be taken as, for example, $\omega_N=\exp\left(-4\pi \dfrac{i}{N}\right)$.
These matrices obey the identities 
\begin{equation}\label{Eq:ZeitlinIdentitiesHF}
HF=\omega_N FH \ , \ F^N=H^N=1 \ .
\end{equation} 

We need $N$ to be odd here since this guarantees that $2$ is
invertible in $\mathbb Z/N\mathbb Z$, so that the phase factor $\omega_N^{k_1k_2/2}$ in (\ref{Eq:ZeitlinBasisJ}) below
is well defined for $k=(k_1,k_2)\in \mathbb{Z} \times \mathbb{Z} \ (\text{mod } (N,N))$.
More precisely, $k_1k_2/2$ is understood as
$2^{-1}k_1k_2$ in  $\mathbb Z \ (\text{mod } N)$.

Define $N^2-1$ matrices $J_k$, $k=(k_1,k_2)\in \mathbb{Z}\times \mathbb{Z} \
(\text{mod } (N,N))$ and $(k_1,k_2)\neq (0,0) \ (\text{mod } (N,N))$
by setting \begin{equation}\label{Eq:ZeitlinBasisJ}
J_{(k_1,k_2)}=\omega_N^{k_1\cdot k_2/2}F^{k_1}H^{k_2} \ .
\end{equation}

Let  \begin{equation}\label{Eq:BasisInsuNCorrespondtoSVect}
L_{k, N}=\dfrac{N}{4\pi i}J_k \ . \end{equation} 
Then Zeitlin \cite{[Zeitlin1991]} shows that we have
\begin{equation}\label{Eq:CommutatorsuN} 
[L_{k,N}, L_{l, N}]=c_{k, l, N}L_{k+l, N} \ ,
\end{equation}
where for each fixed $k, l$ we have $c_{k,l,N}=\dfrac{N}{4\pi i}
2i\sin\left(\dfrac{2\pi(k\times l)}{N}\right)$ and the limit $\lim\limits_{N\rightarrow\infty} c_{k, l, N}=k\times l$. In this sense, the Lie algebra $\su(N)$ ``approximates" $S_0\text{Vect}(\torus^2)$ as $N\rightarrow \infty$, by comparing (\ref{Eq:CommutatorSVectT2}) with (\ref{Eq:CommutatorsuN}).

It turns out that Zeitlin's approximation also works at the level of Euler-Arnold equation written in the form (\ref{Eq:EulerArnoldSDiffTorus:LieBracket}). In fact, we can introduce the projection operators as the linear map $T_N: S(\mathbb{T}^2)\cong S_0\text{Vect}(\torus^2; \mathbb{C})\rightarrow \su(N; \mathbb{C})$, such that for every $k=(k_1, k_2)\in \mathbb{Z}^2\backslash\{(0,0)\}$ we have
\begin{equation}\label{Eq:ProjectionTN:ZeitlinTorusCase}
T_N L_k= L_{k,N} \ .
\end{equation}
Here $L_k=e^{ik\cdot x}$ is defined in (\ref{Eq:FourierBasis2dTorus}) and $L_{k, N}=\dfrac{N}{4\pi i}J_k$ is defined in (\ref{Eq:BasisInsuNCorrespondtoSVect}). Using the projection map $T_N$, we derive, for the vorticity $\omega$ and the stream function $\psi$ in (\ref{Eq:EulerArnoldSDiffTorus}), that
\begin{equation}\label{Eq:TorusTNTruncationMapOnVorticityAndStream}
T_N\omega = \Omega \ , \ T_N\psi = \Psi \ ,
\end{equation}
where \begin{equation}\label{Eq:ZeitlinExpansionVorticityAndStreamMatrices}
\Omega = \sum\limits_{k\in (\mathbb{Z}\mod N)^2\backslash \{(0,0)\}} \Omega_k L_{k, N} \ , \ \Psi = \sum\limits_{k\in (\mathbb{Z}\mod N)^2 \backslash \{(0,0)\}} -\dfrac{\Omega_k}{|k|^2} L_{k, N} \ .
\end{equation}
The relation between $\Omega$ and $\Psi$ is given by $\Psi = -A^{-1} \Omega$ where 
\begin{equation}\label{Eq:InertiaOperatorZeitlinCase}
A L_{k,N}=|k|^2 L_{k, N}
\end{equation}
is the analogue of $-\Delta$ as in the $SDiff(\torus^2)$ case, and it also gives the metric on $SU(N)$. One can then apply the Euler-Arnold equation (\ref{Eq:EulerArnoldEquation}) to the group $G=SU(N)$ to obtain the following matrix dynamics
\begin{equation}\label{Eq:EulerArnoldZeitlinTruncated}
\dot{\Omega} = [\Psi, \Omega] \ ,
\end{equation}
which is Zeitlin's Euler-Arnold equation on $SU(N)$. It is easy to see that (\ref{Eq:EulerArnoldZeitlinTruncated}) is symbolically parallel to (\ref{Eq:EulerArnoldSDiffTorus:LieBracket}) via the correspondence $\omega \leftrightarrow \Omega$ and $\psi \leftrightarrow \Psi$. Furthermore, it turns out that if $T_N \omega(0)=\Omega(0)$, then for $t$ fixed, we have \begin{equation}\label{Eq:ZeitlinApproximationL2convergence}
\left\|T_N^* \Omega(t)-\omega(t)\right\|_{L^2(\torus^2)}\rightarrow 0
\end{equation}
as $N\rightarrow\infty$ (see \cite{[GallagherZeitlinApproximation]}). This justifies that Zeitlin's Euler-Arnold equation can be viewed as an approximation of the true Euler-Arnold equation (\ref{Eq:EulerArnoldSDiffTorus}) on $SDiff(\torus^2)$.

Zeitlin's approximation can be viewed as a quantization of the group $SDiff(\mathbb{T}^2)$. When the base space changes from $\mathbb{T}^2$ to $\sphere^2$, a similar approximation scheme has been considered using the quantization idea and representation theory of $SO(3)$. We refer to \cite{[ModinViviani2026ARMA]} for details.

\subsection{Casimir-preserving approximation of 2d Euler flows}\label{Sec:Background:CasimirPreserving2dEulerFlows}

For the 2-d Euler equation on $\mathbb{T}^2$, the vorticity equation (\ref{Eq:EulerArnoldSDiffTorus}) makes the vorticity $\omega$ transported by the area-preserving flow. As a consequence, for every smooth function $\gamma$ we have
$
\dfrac{d}{dt}\displaystyle{\int_{\mathbb{T}^2}} \gamma(\omega)dx = 0$.
In particular, the moments
$\displaystyle{\int_{\mathbb{T}^2} \omega^2 dx \ , \  \int_{\mathbb{T}^2} \omega^3 dx \dots}$
are conserved. These form an infinite family of Casimir invariants for the Euler-Arnold dynamics on $SDiff(\mathbb{T}^2)$.

The $SU(N)$ approximation preserves the corresponding finite-dimensional Casimirs. Indeed, for the matrix dynamics (\ref{Eq:EulerArnoldZeitlinTruncated}) we have
$
\dfrac{d}{dt}\operatorname{tr}(\Omega^k)
=
k\,\operatorname{tr}\bigl(\Omega^{k-1}[\Omega,\Psi]\bigr)
=0 \ , \ 
k=2,\dots,N$.
Hence the quantities
$
c_1=\operatorname{tr}(\Omega^2),\ c_2=\operatorname{tr}(\Omega^3),\ \dots,\ c_{N-1}=\operatorname{tr}(\Omega^N)
$
are conserved along the finite-dimensional flow. These are precisely the Casimir invariants of $\mathfrak{su}(N)$. As we will see in Section \ref{Sec:GeometryAlgebraSUN:GeometryAlgebraSUNCoAdjointOrbit}, their common level sets determine the corresponding coadjoint orbits $\mathcal{O}^*(\Lambda)$, which are the leaves of the symplectic foliation of $\su(N)$. 

Thus Zeitlin's $SU(N)$ approximation respects the Casimir structure of the continuous Euler dynamics. Under the above identification and as $N\to\infty$, the discrete invariants $\operatorname{tr}(\Omega^k)$ converge to the vorticity moments $\displaystyle{\int_{\mathbb{T}^2}\omega^k\,dx}$, showing that the approximation is compatible with the coadjoint/Lie-Poisson geometry at both the finite-dimensional and continuum levels.

\section{The geometry and algebra of $SU(N)$ coadjoint orbits}\label{Sec:GeometryAlgebraSUN}

\subsection{Real and orthonormal basis of $\su(N; \mathbb{C})$}\label{Sec:GeometryAlgebraSUN:RealOrthonormalBasisZeitlinApproximation}

By (\ref{Eq:ProjectionTN:ZeitlinTorusCase}), the Zeitlin's basis $L_{k,N}$ defined in (\ref{Eq:BasisInsuNCorrespondtoSVect}) is a projection of the Fourier basis $L_k=e^{ik\cdot x}\in S_0(\mathbb{T}^2; \mathbb{C})\cong S_0\text{Vect}(\mathbb{T}^2;\mathbb{C})$ onto $\su(N;\mathbb{C})$. In parallel with (\ref{Eq:Z2FullSpaceLattice}) and (\ref{Eq:Z2HalfSpaceLattice}), we define the full space lattice 
\begin{equation}\label{Eq:ZeitlinFullSpaceLattice}
\mathcal{K}_N=(\mathbb{Z}\mod N)^2\backslash\{(0,0)\} \ ,
\end{equation}
and the half-space lattice \begin{equation}\label{Eq:ZeitlinHalfSpaceLattice}
\mathcal{K}_N^+=\{k=(k_1, k_2)\in (\mathbb{Z}\mod N)^2\backslash\{(0,0)\}, k_1>0 \text{ or } (k_1=0, k_2>0)\} \ .
\end{equation} 
Then 
\begin{equation}\label{Eq:ZeitlinHalfSpaceLattice:Cardinality}
|\mathcal{K}_N^+| =\dfrac{N^2-1}{2} \ . \end{equation} 
Since $N$ is odd this is an integer.
Comparing with the fact that 
$$\cos(k\cdot x) = \dfrac{1}{2}(e^{ik \cdot x}+e^{-ik\cdot x}) \ , \ \sin(k\cdot x) = \dfrac{1}{2i}(e^{ik \cdot x}-e^{-ik\cdot x}) \ , $$
the analogous $\su(N)$ version of cos and sin Zeitlin basis are defined by 
\begin{equation}\label{Eq:RealSinCosZeitlinBasis}
L_{k,\cos,N}=\dfrac{1}{2}\left(L_{k,N}+L_{-k,N}\right), L_{k,\sin,N}=\dfrac{1}{2i}\left(L_{k,N}-L_{-k,N}\right) \ , \ k \in \mathcal{K}_N^+ \ .
\end{equation}
In fact, the matrices $F$ and $H$ defined as in  (\ref{Eq:ZeitlinBasisFH}) are unitary matrices, so that the matrix $J_{(k_1, k_2)}$ defined in (\ref{Eq:ZeitlinBasisJ}) satisfies $J_{(k_1, k_2)}^* \stackrel{\text{def}}{=} \overline{J_{(k_1, k_2)}^T}=\omega_N^{-k_1\cdot k_2/2}(\overline{H^T})^{k_2}(\overline{F^T})^{k_1} = \omega_N^{-k_1\cdot k_2/2}H^{-k_2}F^{-k_1} = \omega_N^{-k_1\cdot k_2/2}\omega_N^{k_1\cdot k_2}F^{-k_1}H^{-k_2}=\omega_N^{k_1\cdot k_2/2}F^{-k_1}H^{-k_2}=J_{-(k_1, k_2)}$, where we have used (\ref{Eq:ZeitlinIdentitiesHF}). By (\ref{Eq:BasisInsuNCorrespondtoSVect}), this implies that $L_{k,N}^*\stackrel{\text{def}}{=} \overline{L_{k,N}^T}=-L_{-k, N}$, which further implies that $L_{k, \cos, N}, L_{k, \sin, N}\in \su(N)$, i.e., they are ``real" in the sense of lying in $\su(N)$ (as a real Lie algebra). 

Let us introduce the standard $Ad$-invariant inner product on $\mathfrak{su}(N)$,
\begin{equation}\label{Eq:suNinnerproduct}
(A,B)=-\operatorname{tr}(AB) \ , \ A,B\in\mathfrak{su}(N) \ ,
\end{equation}
which, as a pairing operation, allows us to identify $\mathfrak{su}(N)^*$ with $\mathfrak{su}(N)$. The corresponding norm is defined as 
\begin{equation}\label{Eq:suNL2Norm}
\|A\|^2=(A,A)=-\operatorname{tr}(A^2) \ , \ A\in \su(N) \ .
\end{equation} 
This inner product is symmetric, positive-definite and is the unique (up to a constant factor) such bilinear form on $\su(N)$, and it is diagonal under the Zeitlin's basis $L_{k,N}$ introduced in (\ref{Eq:BasisInsuNCorrespondtoSVect}), i.e., 
\begin{equation}\label{Eq:OrthogonalBasisZeitlinInnerProduct}
(L_{k, N}, L_{l, N})=c_N\delta_{k+l, 0} \ ,
\end{equation}
for a constant $c_N>0$ depending only on the normalization of the basis $L_{k,N}$. 

Using (\ref{Eq:OrthogonalBasisZeitlinInnerProduct}), we set the normalized Zeitlin basis
\[
\widehat L_{k,N}=c_N^{-1/2}L_{k,N}
\]
to yield
\[
(\widehat L_{k,N},\widehat L_{l,N})=\delta_{k+l,0} \ .
\]
Similarly as (\ref{Eq:RealSinCosZeitlinBasis}), the ``real" orthonormal basis in $\su(N)$ are defined as
\begin{equation}\label{Eq:BasisInsuNOrthonormal}
\widehat{L}_{k,\cos,N}=\frac{1}{\sqrt{2}}(\widehat L_{k,N}+\widehat L_{-k,N}) \ , \ \widehat{L}_{k,\sin,N}=\frac{1}{i\sqrt{2}}(\widehat L_{k,N}-\widehat L_{-k,N}) \ , \ k\in \mathcal{K}_N^+  \ ,
\end{equation}
such that
\begin{equation}\label{Eq:OrthonormalBasisObtainedFromZeitlinsuNBasis}
(\widehat{L}_{k,\cos, N}, \widehat{L}_{l,\cos,N})=\delta_{kl} \ , \ (\widehat{L}_{k,\sin,N},\widehat{L}_{l,\sin, N})=\delta_{kl} \ , \ 
(\widehat{L}_{k,\cos, N},\widehat{L}_{l,\sin, N})=0 \ .
\end{equation}

Let \(\mathcal Z^{(N)}=\{Z_{p}\}_{p=1}^{N^2-1}\) denote the above orthonormal Zeitlin basis (\ref{Eq:BasisInsuNOrthonormal}), which includes $\dfrac{N^2-1}{2}$ number of $\widehat{L}_{k, \cos, N}$'s and $\dfrac{N^2-1}{2}$ number of $\widehat{L}_{k, \sin, N}$'s. 
Thus every \(\Omega\in\mathfrak{su}(N)\) can be written as
\begin{equation}\label{Eq:UniqueCoordinateRepresentationOrthonormalZeitlinBasis}
\Omega=\sum_{p=1}^{N^2-1}
\omega_p Z_p \ .
\end{equation}

\subsection{$SU(N)$ coadjoint orbits}\label{Sec:GeometryAlgebraSUN:GeometryAlgebraSUNCoAdjointOrbit}

Using (\ref{Eq:EulerArnoldEquationCoadjointOperator}), we have, for $X, Y, \Omega\in \su(N)$, that
$(ad^*_X \Omega, Y) =(\Omega, [Y, X])
=(\Omega, YX-XY)
=(\Omega, YX) - (\Omega, XY)
= -\operatorname{tr}(\Omega YX)+\operatorname{tr}(\Omega XY)
= -\operatorname{tr}(X\Omega Y)+\operatorname{tr}(\Omega XY)
= (X\Omega, Y)-(\Omega X, Y)
= ([X, \Omega], Y)$,
which gives
\begin{equation}\label{Eq:EulerArnoldEquationCoadjointOperator:suN}
ad^*_X\Omega =[X, \Omega] \ , \ X, \Omega\in \su(N) \ .
\end{equation}

By (\ref{Eq:RelationAdStaradstar}) we have, for $\Omega\in \su(N)$, that
$\displaystyle{\left.\frac{d}{dt}\right|_{t=0}}Ad^*_{\exp(tX)}\Omega=ad_X^*\Omega=[X,\Omega]$,
which gives
$
Ad_{\exp(tX)}^*\Omega=e^{tX}\Omega e^{-tX}$.
Since \(SU(N)\) is connected, this then gives
\begin{equation}\label{Eq:EulerArnoldEquationCoAdjointOperator:suN}
Ad^*_g \Omega = g \Omega g^{-1} \ , \ g\in SU(N) \ , \ \Omega \in \su(N) \ .
\end{equation}
We note that by (\ref{Eq:AdjointRepresentation}) we also have 
\begin{equation}\label{Eq:EulerArnoldEquationAdjointOperator:suN}
Ad_g \Omega = g \Omega g^{-1} \ , \ g\in SU(N) \ , \ \Omega \in \su(N) \ ,
\end{equation}
so that $Ad_g \Omega = Ad^*_g \Omega$.

Thus the coadjoint orbit on $\su(N)^*\cong \su(N)$ through \(\Omega\in\mathfrak{su}(N)\) is
\begin{equation}\label{Eq:CoAdjointOrbitsuN}
\mathcal O^*(\Omega)
=
\{g\Omega g^{-1}:g\in SU(N)\}.
\end{equation}
This is a smooth embedded submanifold of \(\mathfrak{su}(N)\). Since \(SU(N)\) is compact, the orbit $\mathcal O^*(\Omega)$ is also compact. The stabilizer subgroup of \(\Omega\) is
$\mathrm{Stab}(\Omega)
=
\{g\in SU(N):g\Omega g^{-1}=\Omega\}$,
and the coadjoint orbit can be identified with the homogeneous space
$\mathcal{O}^*(\Omega)
\simeq
SU(N)/\mathrm{Stab}(\Omega)$.

Since $\Omega\in \su(N)$ is skew-Hermitian $\Omega^*=-\Omega$, by the spectral decomposition theorem, there exists $g_0\in SU(N)$ and $\lambda_1,...,\lambda_N\in \mathbb{R}$ such that $\lambda_1+...+\lambda_N=0$ and $\Omega = g_0 \Lambda g_0^{-1}$, where $\Lambda = \text{diag}(i\lambda_1, ..., i\lambda_N)$. Furthermore, $\mathrm{Stab}(\Omega)=g_0\text{Stab}(\Lambda)g_0^{-1}$. This gives further that
\begin{equation}\label{Eq:CoAdjointOrbitsuNIsospectral}
\mathcal{O}^*(\Omega)=\mathcal{O}^*(\Lambda)=\{g\Lambda g^{-1}: g\in SU(N)\}\simeq SU(N)/\text{Stab}(\Lambda) \ .
\end{equation}

Assume $\Lambda = \text{diag}(i\lambda_1, ..., i\lambda_N)$ is such that different eigenvalues have multiplicities $m_1,...,m_k$, with $m_1+...+m_k=N$. Then 
\begin{equation}\label{Eq:StabilizerDiagonalSUNMatrix}
\text{Stab}(\Lambda)=S(U(m_1)\times...\times U(m_k))=
\left\{(A_1,...,A_k): A_j\in U(m_j) \ , \  \prod_{j=1}^k \det(A_j)=1 \right\} \ .
\end{equation}
This means 
\begin{equation}\label{Eq:SUNCoAdjointOrbitAsSUNDividedByStabilizerDiagonal}
\mathcal{O}^*(\Lambda)\simeq SU(N)/S(U(m_1)\times...\times U(m_k)) \ .
\end{equation}
If all eigenvalues are distinct ($m_1=...=m_N=1$), then from the above we have $\text{Stab}(\Lambda)=S(U(1)^N)\simeq T^{N-1}$ is the maximal torus of $SU(N)$, and correspondingly $\dim \mathcal{O}^*(\Lambda)=N^2-N$. In general, with eigenvalue multiplicities $m_1,...,m_k$, $m_1+...+m_k=N$, we have \begin{equation}\label{Eq:DimensionCoAdjointOrbitSUN}
\dim \mathcal{O}^*(\Lambda)=N^2-\sum\limits_{j=1}^k m_j^2=2\sum\limits_{1\leq i<j\leq k}m_im_j \ .
\end{equation}

If two diagonal matrices $\Lambda=\text{diag}(i\lambda_1,...,i\lambda_N)$ and $\widetilde{\Lambda}=\text{diag}(i\widetilde{\lambda}_1,...,i\widetilde{\lambda}_N)$ in $\su(N)$ are on the same coadjoint orbit $\mathcal{O}^*(\Lambda)$, then by (\ref{Eq:CoAdjointOrbitsuNIsospectral}) there exists some $g\in SU(N)$ such that $\widetilde{\Lambda}=g\Lambda g^{-1}$. Since unitary conjugation keeps spectrum, we must have the equality of multisets $\{\lambda_1,...,\lambda_N\}=\{\widetilde{\lambda}_1,...,\widetilde{\lambda}_N\}$. Conversely, if these two multisets agree, then there exists a permutation matrix, up to a harmless phase adjustment to make determinant one, which conjugates $\Lambda$ to 
$\widetilde{\Lambda}$. Hence the two matrices lie on the same coadjoint orbit. Therefore
\begin{equation}\label{Eq:CoAdjointOrbitUniquelyDeterminedBySpectrum}
\mathcal{O}^*(\Lambda)=\mathcal{O}^*(\widetilde{\Lambda}) \Longleftrightarrow \{\lambda_1,...,\lambda_N\}=\{\widetilde{\lambda}_1,...,\widetilde{\lambda}_N\} \text{ as multisets} .
\end{equation}
This means that the coadjoint orbit $\mathcal{O}^*(\Omega)$ is precisely the isospectral manifold associated with $\Omega\in \su(N)$.

Using (\ref{Eq:CoAdjointOrbitsuN}), a tangent vector to $\mathcal{O}^*(\Omega)$ at $\widetilde{\Omega}\in \mathcal{O}^*(\Omega)$ can be calculated as
$\displaystyle{\left.\dfrac{d}{dt}\right|_{t=0}} e^{tX}\widetilde{\Omega} e^{-tX}=X\widetilde{\Omega}-\widetilde{\Omega} X=[X, \widetilde{\Omega}]$, $X\in \su(N)$.
Thus the tangent space of the orbit $\mathcal{O}^*(\Omega)$ at $\widetilde{\Omega}$ is given by
\begin{equation}\label{Eq:CoAdjointOrbitTangetSpace}
T_{\widetilde{\Omega}}\mathcal{O}^*(\Omega)
=
\{[X,\widetilde{\Omega}]:X\in\mathfrak{su}(N)\} \ .
\end{equation}

Combining (\ref{Eq:KKSForm}) and (\ref{Eq:EulerArnoldEquationCoadjointOperator:suN}), the KKS symplectic form on $\mathcal{O}^*(\Omega)$ in this case is given by
\begin{equation}\label{Eq:KKSFormsuN}
\omega_{\mathrm{KKS}}([X,\Omega],[Y,\Omega])
=
(\Omega,[X,Y])=-\text{tr}(\Omega [X, Y]) \ , \ X, Y, \Omega \in \su(N) \ .
\end{equation}
The corresponding Liouville measure associated with this KKS form is defined as in (\ref{Eq:LiouvilleMeasureKKSForm}), which, after normalization,
is the same as the Haar-induced orbital measure on $\mathcal{O}^*(\Lambda)$.

By (\ref{Eq:HamiltonianEqViaKirillovForm}), for a Hamiltonian \(\mathscr{H}:\mathfrak{su}(N)^*\to \mathbb R\), the Hamiltonian vector field $X_{\mathscr{H}}=[A, \Omega]$ satisfies
$\omega_{\text{KKS}}([Y, \Omega], [A, \Omega])= (\Omega, [Y, A])=(d\mathscr{H}, [Y, \Omega])=(\Omega, [\nabla \mathscr{H}, Y])$, where $\nabla \mathscr{H}$ is defined via $\left.\dfrac{d}{d\varepsilon}\right|_{\varepsilon=0}\mathscr{H}(\Omega + \varepsilon \delta_\Omega)=(\nabla \mathscr{H}(\Omega), \delta_ \Omega)$. This gives $A=-\nabla \mathscr{H}(\Omega)$ and the Hamiltonian vector field 
\begin{equation}\label{Eq:HamiltonianVectorFieldExplicitFormCoAdjointOrbitSUN}
X_{\mathscr{H}}=[\Omega, \nabla \mathscr{H}(\Omega)] \ .
\end{equation}

The coadjoint orbits provide a natural (stratified) symplectic foliation (see \cite[Section 29]{[NovikovModernGeometryII]}) of $\mathfrak{su}(N)$ with its decomposition into disjoint isospectral manifolds
\begin{equation}\label{Eq:suNCoAdjointFoliation}
\mathfrak{su}(N)
=
\bigcup_{\Lambda}
\mathcal O^*(\Lambda) \ ,
\end{equation}
where $\Lambda =\mathrm{diag}(i\lambda_1,...,i\lambda_N)$ ranges over diagonal
skew-Hermitian matrices with trace zero, modulo permutation of the eigenvalues. Each leaf in this foliation, which is a coadjoint orbit, can be parametrized by the multiset $\{\lambda_1,...,\lambda_N\}$. We note that the foliation is not a globally smooth
foliation, but rather a stratification by orbit type. On each stratum with fixed
eigenvalue multiplicity, the coadjoint orbits form a locally smooth invariant
foliation. The singular strata correspond to eigenvalue collisions. 

On a regular stratum with simple spectrum, one may work
locally using coordinates $(y,\boldsymbol{\lambda})$, where $y$ denotes the position along the
coadjoint leaf and $\boldsymbol{\lambda}=(\lambda_1,...,\lambda_N)$ denotes the transverse spectral variables. 
The collection of coadjoint orbit parameters $\boldsymbol{\lambda}$ can be described as the set
$$\mathfrak{L}=\left\{\lambda_1\geq \lambda_2\geq ...\geq \lambda_N \ , \ \sum\limits_{i=1}^N \lambda_i = 0\right\}\subset \mathbb{R}^N \ .$$
This is a closed set of conic shape. The eigenvalue parameters of the regular strata correspond to the relative interior of this set where all eigenvalues are distinct, while the singular strata correspond to points on the boundary of this set where eigenvalues may collide.

An equivalent parameterization of the coadjoint orbit $\mathcal{O}^*(\Omega)=\mathcal{O}^*(\Lambda)$, where $\Lambda =\text{diag}(i\lambda_1,...,i\lambda_N)$ is to use the power sums \begin{equation}\label{Eq:suNCoAdjointFoliationCasimirFunctions} c_{k-1}=c_{k-1}(\Omega)\stackrel{\text{def}}{=} \operatorname{tr}(\Omega^k)=\operatorname{tr}(\Lambda^k)=\sum\limits_{j=1}^N (i\lambda_j)^k
\end{equation} 
for $k=2,...,N$ (in the case $k=1$ we have $\operatorname{tr}(\Lambda)=0$), which are the Casimir invariants (see Section \ref{Sec:Background:CasimirPreserving2dEulerFlows}). By the Newton's identities (see \cite{[WikiNewtonIdentities]}), the power sums $(c_1,....,c_{N-1})$ uniquely determine the characteristic polynomial $p(x)=\det(\Omega-xI)$, and thus the multiset of eigenvalue parameters $\{\lambda_1,...,\lambda_N\}$. Based on this observation, the foliation (\ref{Eq:suNCoAdjointFoliation}) can be equivalently written as \begin{equation}\label{Eq:suNCoAdjointFoliationPowerSums} 
\su(N)=\bigcup_{z\in \mathcal{A}}\mathcal{F}_z
\end{equation}
where each foliation parameter $z\in \mathcal{A}$ can be identified as the Casimir variables $z=(c_1,...,c_{N-1})$. In this way the coadjoint orbit $\mathcal{O}^*(\Lambda)$ can be written in the form of (\ref{Intro:Eq:CoadjointOrbitZeitlinEulerArnold}), with $\mathcal{F}_z=\mathcal{O}^*(z)$. The discriminant (Vandermonde determinant) square $\Delta^2(\lambda_1,...,\lambda_N)=\prod\limits_{1\leq i <j\leq N}(\lambda_i-\lambda_j)^2$ can be written as a polynomial (discriminant locus) $D(c_1,...,c_{N-1})$ of $(c_1,...,c_{N-1})$. Thus the singular strata of this foliation where eigenvalues collide can be described by the hypersurface \begin{equation}\label{Eq:SingularStataCoAdjointFoliationsuNDiscriminantHypersurface}
\mathcal{A}_{\text{singular}}\stackrel{\text{def}}{=} \{z=(c_1,...,c_{N-1})\in \mathcal{A}: D(z)=0\} \ .
\end{equation}
Away from this hypersurface, any other foliation parameter 
\begin{equation}\label{Eq:RegularStataCoAdjointFoliationsuNDiscriminantHypersurface}
z\in \mathcal{A}_{\text{regular}}\stackrel{\text{def}}{=} \mathcal{A}\backslash \mathcal{A}_{\text{singular}}
\end{equation} corresponds to a regular stratum, that is some coadjoint orbit $\mathcal{O}^*(\Lambda)$ with $\Lambda$ having simple spectrum. We introduce the projection map $\pi: \su(N)\rightarrow \mathcal{A}$ such that 
\begin{equation}\label{Eq:suNCoAdjointFoliationProjectionMap}
\pi(x) = z = (c_1,...,c_{N-1})\in \mathcal{A} \ , \ x\in \su(N) \ .
\end{equation}
Equivalently, any point $x\in \su(N)$ can be defined by the pair 
\begin{equation}\label{Eq:suNCoadjointFoliationLeafDecomposition}
(y, z)
\end{equation}
such that $z\in \mathcal{A}$ and $y\in \mathcal{O}^*(z)$. This decomposition forms a well-defined local coordinate of $x\in \su(N)$ when $z\in \mathcal{A}_{\text{regular}}$, that is within a small neighborhood of a fixed regular stratum of the $\mathfrak{su}(N)$ foliation (\ref{Eq:suNCoAdjointFoliationPowerSums}), where the coadjoint orbits form a smooth foliation of constant dimension.  Under this local coordinate we have \begin{equation}\label{Eq:suNCoAdjointFoliationProjectionMapInverse}
\pi^{-1}(z)=\{(y, z): y\in \mathcal{O}^*(z)\} \cong \mathcal{O}^*(z) \ , \ z\in \mathcal{A}_{\text{regular}} \ .
\end{equation}

\section{The Euler-Arnold equation on $SU(N)$ with fast-slow stochastic dynamics and its fluxes: an averaging framework}\label{Sec:AveragingFramework}

\subsection{Euler-Arnold equation on $SU(N)$ with fast-slow stochastic dynamics}\label{Sec:AveragingFramework:EulerArnoldMultiscale}

We now introduce specific noise structure that fits the scheme  (\ref{Intro:Eq:EulerArnoldZeitlinTruncatedDissipatedForcedWithFastNoise}). As in (\ref{Intro:Eq:TangentialFastStochastic}), we split the fast stochastic forcing term in (\ref{Intro:Eq:EulerArnoldZeitlinTruncatedDissipatedForcedWithFastNoise}) into $2$ parts: the $X$ part is tangent to the coadjoint orbit and the $Y$ part is multiplied by $\varepsilon\rho$ to provide transversal perturbations to the coadjoint orbit. In addition to the terms introduced in the scheme (\ref{Intro:Eq:EulerArnoldZeitlinTruncatedDissipatedForcedWithFastNoise}), we also add a non-degenerate background noise scaled with parameter $\delta$ such that $0<\delta<\!\!<1$, acting in all directions of $\su(N)$, whose role is to ensure, when combined with the dissipation and the damping term, irreducibility and the existence of
a unique invariant measure (see Theorem \ref{Thm:ErgodicityEulerArnoldZeitlinTruncatedDissipatedForcedWithFastNoise} below). 

Using the symbol conventions in (\ref{Intro:Eq:EulerArnoldZeitlinTruncatedDissipatedForcedWithFastNoise}), the system is then written as
\begin{equation}\label{Eq:EulerArnoldZeitlinTruncatedDissipatedForcedWithFastNoise:Specific}
d\Omega_t^{\varepsilon, \delta} = \dfrac{1}{\varepsilon}\underbrace{\left(dX(\Omega_t^{\varepsilon, \delta})+\varepsilon\rho dY(\Omega_t^{\varepsilon,\delta}) \right)}_{\text{fast stochastic forcing } \mathscr{F}(\Omega^{\varepsilon,\delta}_t)dt} + B(\Omega_t^{\varepsilon, \delta})dt + f(t)dt + \delta d\mathfrak{b}_t \ ,
\end{equation}
where 
\begin{itemize}
\item[(a)] The fast stochastic forcing is given by the following stochastic dynamics
\begin{equation}\label{Eq:EulerArnoldZeitlinTruncatedDissipatedForcedWithFastNoise:FastPart}
\mathscr{F}(\Omega)dt=\underbrace{\sum\limits_{j=1}^m X_j(\Omega)\circ dW_t^{(j)}}_{dX(\Omega)} + \varepsilon \rho \underbrace{\sum\limits_{k=1}^l Y_k(\Omega)\circ dB_t^{(k)}}_{dY(\Omega)} \ ,
\end{equation} 
which is driven by smooth vector fields $X_j(\Omega), j=1,...,m$ and $Y_k(\Omega), k=1,...,l$ on $\su(N)$, written in Stratonovich integral form, with $W_t^{(j)}, j=1,...,m$, $B_t^{(k)}, k=1,...,l$ being independent 1-d Brownian motions. The vector fields $X_j(\Omega), Y_k(\Omega)$ and the parameter $\rho \geq 0$ are subject to conditions in Assumption \ref{Assumption:FastVectorEnergyPreservingPerturbation} below;
\item[(b)] The background noise $d\mathfrak{b}_t$ is defined through the ``full-$\su(N)$" Brownian process
\begin{equation}\label{Eq:EulerArnoldZeitlinTruncatedDissipatedForcedWithFastNoise:BackGroundPart}
\mathfrak{b}_t=\sum\limits_{k\in \mathcal{K}_N^+} \left[L_{k, \cos, N}\widetilde{W}_t^{(k, \cos)}+L_{k, \sin, N}\widetilde{W}_t^{(k, \sin)}\right] \ ,
\end{equation}
where $\mathcal{K}_N^+$ is defined as (\ref{Eq:ZeitlinHalfSpaceLattice}), the real basis $L_{k, \cos, N}$ and $L_{k, \sin, N}$ are defined in (\ref{Eq:RealSinCosZeitlinBasis}). The $1$-d Brownian motions $\widetilde{W}_t^{(k, \cos)}$ and $\widetilde{W}_t^{(k, \sin)}$ are independent and independent of the $W_t^{(j)}, B_t^{(k)}$'s defined in (\ref{Eq:EulerArnoldZeitlinTruncatedDissipatedForcedWithFastNoise:FastPart}). It is clear, by our discussion in Section \ref{Sec:GeometryAlgebraSUN:RealOrthonormalBasisZeitlinApproximation}, that $\mathfrak{b}_t\in \su(N)$;

\item[(c)] The term 
\begin{equation}\label{Eq:EulerArnoldZeitlinTruncatedDissipatedForcedWithFastNoise:B}
B(\Omega)=[\Psi, \Omega]-\alpha \Omega + \nu(-A)\Omega \ , \  \Psi = -A^{-1}\Omega \ ,
\end{equation}
with the operator $A$ defined as in (\ref{Eq:InertiaOperatorZeitlinCase}). Assume $\Omega\in \su(N)$, then (\ref{Eq:InertiaOperatorZeitlinCase}) implies that the operator $A$ preserves \(\mathfrak{su}(N)\) and \(\Psi=-A^{-1}\Omega\in\mathfrak{su}(N)\). Since the commutator of two
anti-Hermitian matrices is again anti-Hermitian, we have $[\Psi, \Omega]\in \su(N)$, and the linear terms
\(-\alpha\Omega\) and \(\nu(-A)\Omega\) also lie in \(\mathfrak{su}(N)\).
Thus when $\Omega\in \su(N)$ we have \(B(\Omega)\in\mathfrak{su}(N)\);

\item[(d)] In parallel with (\ref{Eq:2d-ns-damped-forced:forcing}), we set the external stochastic forcing term
\begin{equation}\label{Eq:EulerArnoldZeitlinTruncatedDissipatedForcedWithFastNoise:Forcing}
f(t)=\sum\limits_{k\in \mathcal{K}_N^+, |k| \approx k_f}q_k \left[ 
L_{k, \cos, N} \dot{\beta}^{(k, \cos)}_t
+
L_{k, \sin, N} \dot{\beta}^{(k, \sin)}_t
\right] \ ,
\end{equation}
with a set of independent $1$-d Brownian motions $\beta_t^{(k, \cos)}, \beta_t^{(k, \sin)}$, which are also independent of the Brownian motions introduced in (\ref{Eq:EulerArnoldZeitlinTruncatedDissipatedForcedWithFastNoise:FastPart}) and (\ref{Eq:EulerArnoldZeitlinTruncatedDissipatedForcedWithFastNoise:BackGroundPart}). The quantities $q_k\in \mathbb{R}$ are forcing amplitudes, with $0\leq \sum\limits_{k\in \mathcal{K}_N^+, |k|\approx k_f}q_k^2<\infty$ satisfied automatically for each finite $N$. By our discussions in Section \ref{Sec:GeometryAlgebraSUN:RealOrthonormalBasisZeitlinApproximation}, it is also clear that $f(t)\in \su(N)$.
\end{itemize}

The above conditions ensure existence and uniqueness of a solution $\Omega_t^{\varepsilon,\delta}$ to (\ref{Eq:EulerArnoldZeitlinTruncatedDissipatedForcedWithFastNoise:Specific}). It also turns out that, since every $X_j(\Omega), Y_k(\Omega)$ are in $\su(N)$, when the initial condition $\Omega_0^{\varepsilon, \delta}\in \su(N)$, it is guaranteed that $\Omega_t^{\varepsilon, \delta}\in \su(N)$ for all $t\geq 0$. 

For the generating directions $X_j(\Omega), j=1,...,m$, $Y_k(\Omega), k=1,...,l$ of the fast stochastic forcing in (\ref{Eq:EulerArnoldZeitlinTruncatedDissipatedForcedWithFastNoise:FastPart}), as we have indicated in item (a) above, we put the following assumption:

\begin{assumption}[Fast stochastic forcing within the perturbative regime]\label{Assumption:FastVectorEnergyPreservingPerturbation}
We make the following assumptions on the $X$ and $Y$ terms in the fast stochastic forcing defined in \emph{(\ref{Eq:EulerArnoldZeitlinTruncatedDissipatedForcedWithFastNoise:FastPart})}:

\begin{enumerate}
\item[\emph{(1)}] Each $X_j(\Omega)$ is a smooth vector field on $\su(N)$ that is tangent to the coadjoint orbit, so that it preserves the second Casimir $c_1(\Omega)=-(\Omega, \Omega)=\mathrm{tr}(\Omega^2)$ (see \emph{(\ref{Eq:suNinnerproduct})}), i.e., 
\begin{equation}\label{Assumption:FastVectorEnergyPreservingPerturbation:Eq:TangentialVector}
(\Omega, X_j(\Omega))=0 \ .
\end{equation}

\item[\emph{(2)}] Each $Y_k(\Omega)$ is a smooth vector field not identically tangent to the coadjoint orbit, so it is possible that
\begin{equation}\label{Assumption:FastVectorEnergyPreservingPerturbation:Eq:PerturbationVector}
(\Omega, Y_k(\Omega))\not \equiv 0 \ , \ \text{ for at least one } k \ .
\end{equation}

\item[\emph{(3)}] The Frech\'{e}t derivative $DY_k$, defined as $[DY_k(\Omega)](\Theta)=\left.\dfrac{d}{ds}\right|_{s=0} Y_k(\Omega+s\Theta)$,  has bounded norm \begin{equation}\label{Assumption:FastVectorEnergyPreservingPerturbation:Eq:BoundedFirstDerivativePerturbative}
\|DY_k(\Omega)\|\leq C \ ,
\end{equation}
for all $k=1,...,l$ and some $C>0$.

\item[\emph{(4)}] The perturbation parameter $\rho\geq 0$ satisfies \begin{equation}\label{Assumption:FastVectorEnergyPreservingPerturbation:Eq:PerturbationParameter}
0\leq \rho\leq K \ ,
\end{equation}
where the constant $K$ satisfies
\begin{equation}\label{Assumption:FastVectorEnergyPreservingPerturbation:Eq:PerturbationParameter:ConstraintK}
0<K<\left[\dfrac{2(\alpha+\nu)}{l\left(2C^2+\dfrac{1}{2}C(1+2C^2)\right)}\right]^{1/2} \ ,
\end{equation}
with the constant $C>0$ as defined in \emph{(\ref{Assumption:FastVectorEnergyPreservingPerturbation:Eq:BoundedFirstDerivativePerturbative})}.
\end{enumerate}
\end{assumption}

It is clear that, by (\ref{Assumption:FastVectorEnergyPreservingPerturbation:Eq:BoundedFirstDerivativePerturbative}), that we have the growth constraint \begin{equation}\label{Assumption:FastVectorEnergyPreservingPerturbation:Eq:GrowthCondition}
\|Y_k(\Omega)\|\leq C(1+\|\Omega\|) \ ,
\end{equation}
for all $k=1,...,l$ and the same constant $C>0$ as in (\ref{Assumption:FastVectorEnergyPreservingPerturbation:Eq:BoundedFirstDerivativePerturbative}).

\begin{remark}[Unique ergodicity only needs $X$ to preserve the second Casimir]\label{Rmk:Assumption:UniqueErgodicityOnlyNeedsSecondCasimir}
For the unique ergodicity to hold as in \emph{Theorem \ref{Thm:ErgodicityEulerArnoldZeitlinTruncatedDissipatedForcedWithFastNoise}} below, $X_j$ only needs to conserve the second Casimir $c_1(\Omega)$, instead of the full coadjoint orbit, i.e., all Casimirs $c_1(\Omega), ..., c_{N-1}(\Omega)$. 
\end{remark}

\begin{theorem}[Unique ergodicity]\label{Thm:ErgodicityEulerArnoldZeitlinTruncatedDissipatedForcedWithFastNoise}
Assume that $\delta, \alpha>0$. Then, for every fixed $\varepsilon>0$ and $N\in \mathbb{N}$, 
 the process $\Omega^{\varepsilon,\delta}_t$ defined in \emph{(\ref{Eq:EulerArnoldZeitlinTruncatedDissipatedForcedWithFastNoise:Specific})} admits a unique invariant probability measure
$\mu^{(N)}_{\varepsilon,\delta, \alpha , \nu}$ on $\su(N)$.
\end{theorem}

\begin{proof}
We use the framework of \cite{[NathanUniqueErgodicityNotes]}. The equation (\ref{Eq:EulerArnoldZeitlinTruncatedDissipatedForcedWithFastNoise:Specific}) is a finite-dimensional diffusion on the real vector space
\(\mathfrak{su}(N)\). By construction, the background noise driven by $\mathfrak{b}_t$ is additive and
non-degenerate, acting on all directions of $\su(N)$.
Hence the diffusion matrix is uniformly elliptic for every fixed
\(\delta>0\). In particular, the associated Markov semigroup is weakly irreducible and strong Feller
 on \(\mathfrak{su}(N)\). So by \cite[Theorem 4.1]{[NathanUniqueErgodicityNotes]}, we have uniqueness of invariant measure. 

It remains to prove existence of invariant measure through a Krylov-Bogoliubov type argument (see \cite[Section 3]{[NathanUniqueErgodicityNotes]}). To this end we define the Lyapunov function
\begin{equation}\label{Thm:ErgodicityEulerArnoldZeitlinTruncatedDissipatedForcedWithFastNoise:Eq:LyapunovFunction}
V(\Omega)=1+\|\Omega\|^2 \ ,
\end{equation}
so that 
\begin{equation}\label{Thm:ErgodicityEulerArnoldZeitlinTruncatedDissipatedForcedWithFastNoise:Eq:LyapunovFunctionGradient}
\nabla V(\Omega)=2 \Omega \ , \ [D^2 V(\Omega)](\Theta, \Theta)=2\|\Theta\|^2 \ , \ \Omega, \Theta \in \su(N) \ ,
\end{equation}
where $[DV(\Omega)](\Theta)=(\nabla V (\Omega), \Theta)=\left.\dfrac{d}{dt}\right|_{t=0} V(\Omega+t\Theta)$ is the Fr\'{e}chet derivative of $V(\Omega)$; the gradient is with respect to the inner product $(\bullet, \bullet)$ on $\su(N)$, as defined in (\ref{Eq:suNinnerproduct}); and $[D^2V(\Omega)](\Theta_1, \Theta_2)=\left.\dfrac{\partial^2}{\partial s \partial t}\right|_{s=t=0} V(\Omega+s\Theta_1+t\Theta_2)$ is the Hessian of $V(\Omega)$. Using (\ref{Eq:InertiaOperatorZeitlinCase}) we have, for any $\Omega\in \su(N)$, 
\begin{equation}\label{Thm:ErgodicityEulerArnoldZeitlinTruncatedDissipatedForcedWithFastNoise:Eq:DampingDissipationL2Estimate}
(-\alpha\Omega+\nu(-A)\Omega,\Omega)
\leq -c_{\alpha,\nu}(\Omega, \Omega)
\end{equation}
where the constant $c_{\alpha,\nu}>0$ can be taken as \begin{equation}\label{Thm:ErgodicityEulerArnoldZeitlinTruncatedDissipatedForcedWithFastNoise:Eq:DampingDissipationL2Estimatecalphanu}
c_{\alpha, \nu}=\alpha+\nu \ .
\end{equation} Moreover, the commutator term
\([\Psi,\Omega]\) is energy-preserving, i.e., \begin{equation}\label{Thm:ErgodicityEulerArnoldZeitlinTruncatedDissipatedForcedWithFastNoise:Eq:CommutatorEnergyPreserving}
([\Psi, \Omega], \Omega)=0 \ .
\end{equation} 
By It\^{o}'s formula, we derive
\begin{equation}\label{Thm:ErgodicityEulerArnoldZeitlinTruncatedDissipatedForcedWithFastNoise:Eq:ItoFormula}
\begin{array}{ll}
& dV(\Omega^{\varepsilon, \delta}_t) 
\\
= & \dfrac{1}{\varepsilon} \sum\limits_{j=1}^m (\nabla V(\Omega_t^{\varepsilon,\delta}), X_j(\Omega_t^{\varepsilon,\delta}))\circ dW_t^{(j)} + \rho\sum\limits_{k=1}^l (\nabla V(\Omega_t^{\varepsilon, \delta}), Y_k(\Omega_t^{\varepsilon, \delta}))\circ dB_t^{(k)}
\\
& \ +\delta \left(\nabla V(\Omega_t^{\varepsilon, \delta}), \sum\limits_{k\in \mathcal{K}^+_N}\left[L_{k, \cos, N}d\widetilde{W}_t^{(k, \cos)}+L_{k, \sin, N}d\widetilde{W}_t^{(k, \sin)}\right]\right)
\\
& \ + \dfrac{\delta^2}{2}\sum\limits_{k\in \mathcal{K}_N^+}\left([\nabla^2 V(\Omega_t^{\varepsilon, \delta})](L_{k, \cos N}, L_{k, \cos N})+[\nabla^2 V(\Omega_t^{\varepsilon, \delta})](L_{k, \sin, N}, L_{k, \sin, N})\right)dt
\\
& \ + (\nabla V(\Omega_t^{\varepsilon, \delta}), [A^{-1}\Omega_t^{\varepsilon, \delta}, \Omega_t^{\varepsilon, \delta}]-\alpha\Omega_t^{\varepsilon, \delta} + \nu(-A)\Omega_t^{\varepsilon, \delta})dt 
\\
& \ + \left(\nabla V(\Omega_t^{\varepsilon, \delta}), \sum\limits_{k\in \mathcal{K}_N^+, |k| \approx k_f}q_k \left[ 
L_{k, \cos, N} d\beta^{(k, \cos)}_t
+
L_{k, \sin, N} d\beta^{(k, \sin)}_t
\right]\right)
\\
& \ + \dfrac{1}{2}\sum\limits_{k\in \mathcal{K}_N^+, |k|\approx k_f}q_k^2\left([\nabla^2 V(\Omega_t^{\varepsilon, \delta})](L_{k, \cos N}, L_{k, \cos N})+[\nabla^2 V(\Omega_t^{\varepsilon, \delta})](L_{k, \sin, N}, L_{k, \sin, N})\right)dt \ .
\end{array}
\end{equation}
By (\ref{Eq:RealSinCosZeitlinBasis}) and (\ref{Eq:OrthogonalBasisZeitlinInnerProduct}), we see that $\|L_{k, \cos, N}\|^2=\|L_{k, \sin, N}\|^2=\dfrac{c_N}{2}$. Combining this with  (\ref{Thm:ErgodicityEulerArnoldZeitlinTruncatedDissipatedForcedWithFastNoise:Eq:LyapunovFunctionGradient}), taking expectation on the above It\^{o}'s formula (\ref{Thm:ErgodicityEulerArnoldZeitlinTruncatedDissipatedForcedWithFastNoise:Eq:ItoFormula}), and making use of condition (\ref{Assumption:FastVectorEnergyPreservingPerturbation:Eq:TangentialVector}) in Assumption \ref{Assumption:FastVectorEnergyPreservingPerturbation} we have
\begin{equation}\label{Thm:ErgodicityEulerArnoldZeitlinTruncatedDissipatedForcedWithFastNoise:Eq:ExpectedItoLyapunov} 
\begin{array}{ll}
d\mathbf{E} V(\Omega_t^{\varepsilon,\delta}) & = 2\mathbf{E} \left[\rho \sum\limits_{k=1}^l (\Omega_t^{\varepsilon,\delta}, Y_k(\Omega_t^{\varepsilon,\delta}))\circ dB_t^{(k)}\right]
\\
& \qquad + 2\mathbf{E} \left(\Omega_t^{\varepsilon, \delta}, [A^{-1}\Omega_t^{\varepsilon, \delta}, \Omega_t^{\varepsilon, \delta}]-\alpha\Omega_t^{\varepsilon, \delta} + \nu(-A)\Omega_t^{\varepsilon, \delta}\right)dt
\\
& \qquad + \delta^2 c_N |\mathcal{K}_N^+|dt + c_N \sum\limits_{k\in \mathcal{K}_N^+, |k|\approx k_f}q_k^2dt
\\
& \leq 2\rho \mathbf{E} \left[\sum\limits_{k=1}^l (\Omega_t^{\varepsilon,\delta}, Y_k(\Omega_t^{\varepsilon,\delta}))\circ dB_t^{(k)}\right]
\\
& \qquad - 2c_{\alpha, \nu}\mathbf{E}\|\Omega_t^{\varepsilon, \delta}\|^2 dt + \delta^2 c_N \dfrac{N^2-1}{2} dt + c_N \sum\limits_{k\in \mathcal{K}_N^+, |k|\approx k_f}q_k^2 dt
\\
& = 2\rho \mathbf{E} \left[\sum\limits_{k=1}^l (\Omega_t^{\varepsilon,\delta}, Y_k(\Omega_t^{\varepsilon,\delta}))\circ dB_t^{(k)}\right]
- 2c_{\alpha, \nu}\mathbf{E} V(\Omega_t^{\varepsilon, \delta}) dt + D dt \ ,
\end{array}
\end{equation}
where
\begin{equation}\label{Thm:ErgodicityEulerArnoldZeitlinTruncatedDissipatedForcedWithFastNoise:Eq:ConstantD}
D=D(\alpha, \nu, \delta, N, k_f, q)=\left(2 c_{\alpha, \nu}+ \delta^2 c_N \dfrac{N^2-1}{2} + c_N \sum\limits_{k\in \mathcal{K}_N^+, |k|\approx k_f}q_k^2\right) > 0 \ .
\end{equation}
It remains to estimate the term $$\displaystyle{\mathbf{E} \left[\sum\limits_{k=1}^l (\Omega_t^{\varepsilon,\delta}, Y_k(\Omega_t^{\varepsilon,\delta}))\circ dB_t^{(k)}\right]} \ .$$ 
By combining (\ref{Thm:ErgodicityEulerArnoldZeitlinTruncatedDissipatedForcedWithFastNoise:Eq:LyapunovFunctionGradient}) with (\ref{Lm:StratonovichToIto:Eq:StratonovichToIto}) in Lemma \ref{Lm:StratonovichToIto}, we can estimate 
\begin{equation}\label{Thm:ErgodicityEulerArnoldZeitlinTruncatedDissipatedForcedWithFastNoise:Eq:ExpectationFastPart}
\begin{array}{ll}
& \mathbf{E} \left[\sum\limits_{k=1}^l (\Omega_t^{\varepsilon,\delta}, Y_k(\Omega_t^{\varepsilon,\delta}))\circ dB_t^{(k)}\right]
\\
= &  \dfrac{\rho}{2}
\sum\limits_{k=1}^l \left[\mathbf{E}\|Y_k(\Omega_t^{\varepsilon, \delta})\|^2+\mathbf{E} (\Omega_t^{\varepsilon, \delta}, [D Y_k(\Omega_t^{\varepsilon, \delta})](Y_k(\Omega_t^{\varepsilon, \delta})))\right]dt
\\
\leq & \dfrac{\rho}{2}\sum\limits_{k=1}^l \left[2C^2(1+\mathbf{E}\|\Omega_t^{\varepsilon, \delta}\|^2)+\dfrac{1}{2}C\mathbf{E}(\|\Omega_t^{\varepsilon, \delta}\|^2+\|Y_k(\Omega_t^{\varepsilon, \delta})\|^2)\right]dt
\\
\leq & \dfrac{\rho}{2}\sum\limits_{k=1}^l \left[2C^2(1+\mathbf{E}\|\Omega_t^{\varepsilon, \delta}\|^2)+\dfrac{1}{2}C(1+2C^2)(1+\mathbf{E}\|\Omega_t^{\varepsilon, \delta}\|^2)\right]dt
\\
= & \dfrac{\rho}{2}l\left(2C^2+\dfrac{1}{2}C(1+2C^2)\right)\mathbf{E} V(\Omega_t^{\varepsilon,\delta})dt \ ,
\end{array}
\end{equation}
where $C>0$ is given in (\ref{Assumption:FastVectorEnergyPreservingPerturbation:Eq:BoundedFirstDerivativePerturbative}). Inserting (\ref{Thm:ErgodicityEulerArnoldZeitlinTruncatedDissipatedForcedWithFastNoise:Eq:ExpectationFastPart}) back into (\ref{Thm:ErgodicityEulerArnoldZeitlinTruncatedDissipatedForcedWithFastNoise:Eq:ExpectedItoLyapunov}), using (\ref{Assumption:FastVectorEnergyPreservingPerturbation:Eq:PerturbationParameter}), we get
\begin{equation}\label{Thm:ErgodicityEulerArnoldZeitlinTruncatedDissipatedForcedWithFastNoise:Eq:ExpectedItoLyapunovFinalEstimate}
d\mathbf{E} V(\Omega_t^{\varepsilon,\delta})\leq \left[K^2 l\left(2C^2+\dfrac{1}{2}C(1+2C^2)\right)-2c_{\alpha, \nu}\right]\mathbf{E} V(\Omega_t^{\varepsilon,\delta})dt+D dt \ .
\end{equation}
Using (\ref{Thm:ErgodicityEulerArnoldZeitlinTruncatedDissipatedForcedWithFastNoise:Eq:DampingDissipationL2Estimatecalphanu}) and (\ref{Assumption:FastVectorEnergyPreservingPerturbation:Eq:PerturbationParameter:ConstraintK}), this yields the desired compactness estimate that fits into \cite[Section 3]{[NathanUniqueErgodicityNotes]}, which proves the existence of invariant measure.
\end{proof}

\begin{lemma}\label{Lm:StratonovichToIto}
The Stratonovich stochastic differential terms $(\nabla V(\Omega_t^{\varepsilon,\delta}), Y_k(\Omega_t^{\varepsilon,\delta}))\circ dB^{(k)}_t$, $1\leq k \leq l$ in \emph{(\ref{Thm:ErgodicityEulerArnoldZeitlinTruncatedDissipatedForcedWithFastNoise:Eq:ItoFormula})} can be expanded in It\^{o} form as
\begin{equation}\label{Lm:StratonovichToIto:Eq:StratonovichToIto}
\begin{array}{ll}
& (\nabla V(\Omega_t^{\varepsilon, \delta}), Y_k(\Omega_t^{\varepsilon,\delta}))\circ dB_t^{(k)}
\\
= & (\nabla V(\Omega_t^{\varepsilon, \delta}), Y_k(\Omega_t^{\varepsilon,\delta})) dB_t^{(k)}
\\
& \qquad +
\dfrac{\rho}{2} \left[[D^2 V(\Omega_t^{\varepsilon,\delta})](Y_k(\Omega_t^{\varepsilon,\delta}), Y_k(\Omega_t^{\varepsilon,\delta}))+(\nabla V(\Omega_t^{\varepsilon,\delta}), [D Y_k(\Omega_t^{\varepsilon,\delta})](Y_k(\Omega_t^{\varepsilon,\delta})))\right]dt \ .
\end{array}
\end{equation}
\end{lemma}

\begin{proof}
By \cite[Chapter IV, Exercise (2.18)]{[RevuzYorContinuousMartingaleBM]} we have, for every $1\leq k \leq l$, that
\begin{equation}\label{Lm:StratonovichToIto:Eq:StratonovichToItoExpandViaCrossVariation}
(\nabla V(\Omega_t^{\varepsilon, \delta}), Y_k(\Omega_t^{\varepsilon,\delta}))\circ dB_t^{(k)}=(\nabla V(\Omega_t^{\varepsilon, \delta}), Y_k(\Omega_t^{\varepsilon,\delta})) dB_t^{(k)}+
\dfrac{1}{2}d\left\langle
(\nabla V(\Omega_t^{\varepsilon, \delta}), Y_k(\Omega_t^{\varepsilon,\delta})), B^{(k)}\right\rangle_t \ ,
\end{equation}
where $\left\langle \bullet, \bullet \right\rangle_t$ stands for the cross-variation between two continuous semimartingales. Define $F_k(\Omega)=(\nabla V(\Omega), Y_k(\Omega))$ and $[DF_k(\Omega)](\Theta)=\left.\dfrac{d}{dt}\right|_{t=0}F_k(\Omega+t\Theta)$ be the Fr\'{e}chet derivative of $F_k(\Omega)$. Then by the Stratonovich chain rule, using (\ref{Eq:EulerArnoldZeitlinTruncatedDissipatedForcedWithFastNoise:Specific}),  (\ref{Eq:EulerArnoldZeitlinTruncatedDissipatedForcedWithFastNoise:FastPart}) we have
$$\begin{array}{ll}
& dF_k(\Omega_t^{\varepsilon,\delta})
\\
= & [DF_k(\Omega_t^{\varepsilon,\delta})]\circ d\Omega_t^{\varepsilon,\delta}
\\
= & [DF_k(\Omega_t^{\varepsilon,\delta})]\left(\rho Y_k(\Omega_t^{\varepsilon,\delta})\right)\circ dB_t^{(k)}+ \text{ directions having vanishing cross variation with } B_t^{(k)} \ ,
\end{array}$$
so we get
\begin{equation}\label{Lm:StratonovichToIto:Eq:CrossVariationInTermsOfFrechetDerivative} 
d\left\langle
(\nabla V(\Omega_t^{\varepsilon, \delta}), Y_k(\Omega_t^{\varepsilon,\delta})), B^{(k)}\right\rangle_t=[DF_k(\Omega_t^{\varepsilon, \delta})]\left(\rho Y_k(\Omega_t^{\varepsilon, \delta})\right)dt \ .
\end{equation}

It can be further calculated that 
$$\begin{array}{ll}
[DF_k(\Omega)](\Theta) & = \left.\dfrac{d}{dt}\right|_{t=0} (\nabla V(\Omega+t\Theta), Y_k(\Omega+t\Theta))
\\
& =\left(\left.\dfrac{d}{dt}\right|_{t=0}\nabla V(\Omega+t\Theta), Y_k(\Omega)\right)+(\nabla V(\Omega), [DY_k(\Omega)](\Theta))
\\
& = ([\nabla^2 V(\Omega)](\Theta), Y_k(\Omega))+(\nabla V(\Omega), [D Y_k(\Omega)](\Theta))
\\
& = [D^2 V(\Omega)](\Theta, Y_k(\Omega))+(\nabla V(\Omega), [D Y_k(\Omega)](\Theta)) \ .
\end{array}$$
Setting $\Theta=\rho Y_k(\Omega)$ in the above, we get
\begin{equation}\label{Lm:StratonovichToIto:Eq:FrechetDerivativeInDirection}
[D F_k(\Omega)]\left(\rho Y_k(\Omega)\right)=\rho \left[[D^2 V(\Omega)](Y_k(\Omega), Y_k(\Omega))+(\nabla V(\Omega), [D Y_k(\Omega)](Y_k(\Omega)))\right] \ .
\end{equation}
Combining (\ref{Lm:StratonovichToIto:Eq:StratonovichToItoExpandViaCrossVariation}), (\ref{Lm:StratonovichToIto:Eq:CrossVariationInTermsOfFrechetDerivative}) and (\ref{Lm:StratonovichToIto:Eq:FrechetDerivativeInDirection}) we get 
$$\begin{array}{ll}
& (\nabla V(\Omega_t^{\varepsilon, \delta}), Y_k(\Omega_t^{\varepsilon,\delta}))\circ dB_t^{(k)}
\\
= & (\nabla V(\Omega_t^{\varepsilon, \delta}), Y_k(\Omega_t^{\varepsilon,\delta})) dB_t^{(k)}+
\dfrac{1}{2}[DF_k(\Omega_t^{\varepsilon, \delta})]\left(\rho Y_k(\Omega_t^{\varepsilon, \delta})\right)dt
\\
= & (\nabla V(\Omega_t^{\varepsilon, \delta}), Y_k(\Omega_t^{\varepsilon,\delta})) dB_t^{(k)}
\\
& \qquad +
\dfrac{\rho}{2} \left[[D^2 V(\Omega_t^{\varepsilon,\delta})](Y_k(\Omega_t^{\varepsilon,\delta}), Y_k(\Omega_t^{\varepsilon,\delta}))+(\nabla V(\Omega_t^{\varepsilon,\delta}), [D Y_k(\Omega_t^{\varepsilon,\delta})](Y_k(\Omega_t^{\varepsilon,\delta})))\right]dt \ ,
\end{array}$$
which is (\ref{Lm:StratonovichToIto:Eq:StratonovichToIto}).
\end{proof}

\begin{lemma}\label{Lm:StratonovichToIto:X}
The Stratonovich stochastic differential terms $(\nabla V(\Omega_t^{\varepsilon,\delta}), X_j(\Omega_t^{\varepsilon,\delta}))\circ dW^{(j)}_t$, $1\leq j \leq m$ in \emph{(\ref{Thm:ErgodicityEulerArnoldZeitlinTruncatedDissipatedForcedWithFastNoise:Eq:ItoFormula})} can be expanded in It\^{o} form as
\begin{equation}\label{Lm:StratonovichToIto:X:Eq:StratonovichToIto}
\begin{array}{ll}
& (\nabla V(\Omega_t^{\varepsilon, \delta}), X_j(\Omega_t^{\varepsilon,\delta}))\circ dW_t^{(j)}
\\
= & (\nabla V(\Omega_t^{\varepsilon, \delta}), X_j(\Omega_t^{\varepsilon,\delta})) dW_t^{(j)}
\\
& \qquad +
\dfrac{1}{2\varepsilon} \left[[D^2 V(\Omega_t^{\varepsilon,\delta})](X_j(\Omega_t^{\varepsilon,\delta}), X_j(\Omega_t^{\varepsilon,\delta}))+(\nabla V(\Omega_t^{\varepsilon,\delta}), [D X_j(\Omega_t^{\varepsilon,\delta})](X_j(\Omega_t^{\varepsilon,\delta})))\right]dt \ .
\end{array}
\end{equation}
\end{lemma}

\begin{proof}
The proof of this Lemma is very similar to that of Lemma \ref{Lm:StratonovichToIto} with $\rho Y_k$ replaced by $\dfrac{1}{\varepsilon}X_j$ and $B_t^{(k)}$ replaced by $W_t^{(j)}$, so we omit the details.
\end{proof}

\subsection{The averaging framework}\label{Sec:AveragingFramework:Framework}

We define the auxiliary processes

\begin{equation}\label{Eq:EulerArnoldZeitlinTruncatedDissipatedForcedWithFastNoise:FastPart:X}
d\Omega^X_t=\sum\limits_{j=1}^m X_j(\Omega^X_t)\circ dW_t^{(j)} \ ,
\end{equation}
and
\begin{equation}\label{Eq:EulerArnoldZeitlinTruncatedDissipatedForcedWithFastNoise:FastPart:Y}
d\Omega^Y_t=\sum\limits_{k=1}^l Y_k(\Omega^Y_t)\circ dB_t^{(k)} \ ,
\end{equation}
which are the $X$ and $Y$ parts of the fast stochastic forcing $\mathscr{F}(\Omega_t^{\varepsilon,\delta})dt$ in (\ref{Eq:EulerArnoldZeitlinTruncatedDissipatedForcedWithFastNoise:FastPart}). 

Based on (\ref{Eq:EulerArnoldZeitlinTruncatedDissipatedForcedWithFastNoise:Specific}), we make the following

\begin{assumption}[Fast mixing on coadjoint orbits and transversal perturbation]\label{Assumption:FastVectorMixingFullCoadjoint}
We assume that the fast vector fields $X_j(\Omega)$ in
\emph{Assumption \ref{Assumption:FastVectorEnergyPreservingPerturbation}} are tangent
to coadjoint orbits $\mathcal{O}^*(\Omega)$ and preserve all Casimir invariants. In particular, since the second Casimir $c_1(\Omega)$ is conserved, \emph{Assumption \ref{Assumption:FastVectorEnergyPreservingPerturbation}} is automatically satisfied under \emph{Assumption \ref{Assumption:FastVectorMixingFullCoadjoint}}. 

We further assume that on each coadjoint orbit under consideration, the
tangential vector fields \(X_j\) satisfy the bracket-generating condition
\begin{equation}\label{Assumption:FastVectorMixingFullCoadjoint:Eq:LieBracketGenerating}
\operatorname{Lie}\{X_1,\ldots,X_m\}(\Omega)\stackrel{\mathrm{def}}{=} \emph{\text{span}}(X_1,...,X_m, [X_1, X_2], ..., [X_1,[X_2, X_3]], ...)=
T_\Omega\mathcal O^*(\Omega) \ .
\end{equation}
The bracket-generating condition \emph{(\ref{Assumption:FastVectorMixingFullCoadjoint:Eq:LieBracketGenerating})} then implies that the process $\Omega^X_t$ defined in \emph{(\ref{Eq:EulerArnoldZeitlinTruncatedDissipatedForcedWithFastNoise:FastPart:X})}
is hypoelliptic and ergodic on \(\mathcal O^*(\Omega^X_0)\), with invariant measure
given by some normalized measure $dm_{\mathcal{O}^*(\Omega_0^X)}(\Omega)$, $\Omega \in \mathcal{O}^*(\Omega_0^X)$. This represents rapid mixing within symmetry classes determined
by the Euler-Arnold geometry.

The process $\Omega_t^Y$ in \emph{(\ref{Eq:EulerArnoldZeitlinTruncatedDissipatedForcedWithFastNoise:FastPart:Y})} provides a transversal perturbation that breaks the coadjoint orbit conservation of the $\Omega_t^X$ dynamics.
\end{assumption}

By Theorem \ref{Thm:ErgodicityEulerArnoldZeitlinTruncatedDissipatedForcedWithFastNoise}, for any given $\varepsilon, \delta, \alpha, \nu>0$, system (\ref{Eq:EulerArnoldZeitlinTruncatedDissipatedForcedWithFastNoise:Specific}) admits a unique invariant measure $\mu^{(N)}_{\varepsilon, \delta, \alpha, \nu}$. The limiting procedure of $\mu^{(N)}_{\varepsilon, \delta, \alpha, \nu}$ as $\varepsilon\rightarrow 0$ is closely related to the  \textit{averaging principle} (see, for example, \cite[Chapter 7]{[FWBook]}, \cite{[Khasminskii1968]}, \cite{[XuemeiLi2008]}, \cite{[AveragingDiffusionFoliatedSpacesAOP2016]}), which roughly speaking is a dimension reduction process, that replaces the motion of $\Omega_t^{\varepsilon,\delta}$ by fast mixing on the coadjoint orbit associated with slow evolution on the Casimir parameter space $\mathcal{A}$ defined in (\ref{Eq:suNCoAdjointFoliationProjectionMap}). One of the classical results of the averaging principle asserts that as $\varepsilon\rightarrow 0$, the process $Z_t^{\varepsilon, \delta}=\pi(\Omega_t^{\varepsilon, \delta})$ converges (weakly) to a limiting Markov process $Z_t^\delta$ in the space of continuous trajectories on $\mathcal{A}$. Here $\Omega_t^{\varepsilon, \delta}$ is defined in (\ref{Eq:EulerArnoldZeitlinTruncatedDissipatedForcedWithFastNoise:Specific}) and the projection $\pi: \su(N)\rightarrow \mathcal{A}$ is defined in (\ref{Eq:suNCoAdjointFoliationProjectionMap}). The limiting process $Z_t^\delta$ is characterized by averaging out (\ref{Eq:EulerArnoldZeitlinTruncatedDissipatedForcedWithFastNoise:Specific}) with respect to fast mixing on coadjoint orbits. For the rest of this subsection we'll make this averaging procedure more precise.

As in (\ref{Eq:suNCoadjointFoliationLeafDecomposition}), for any point $\Omega\in \su(N)$, we can introduce the Casimir variables $z=\pi(\Omega)=(c_1(\Omega),...,c_{N-1}(\Omega))$ where $c_k(\Omega)=\operatorname{tr}(\Omega^{k+1})$, $k=1,...,N-1$ are defined in (\ref{Eq:suNCoAdjointFoliationCasimirFunctions}). 
Then $\pi^{-1}(z)$ is a coadjoint foliation leaf $\mathcal{F}_z=\mathcal{O}^*(z)$ with Casimir variable $z$, as introduced in (\ref{Eq:suNCoAdjointFoliationPowerSums}). The averaging operator can then be defined on each of the foliation leaf $\mathcal{F}_z=\mathcal{O}^*(z)$. Let us assume that $z\in \mathcal{A}_{\text{regular}}$ corresponds to a regular stratum, as defined in (\ref{Eq:RegularStataCoAdjointFoliationsuNDiscriminantHypersurface}). 
For each such leaf $\mathcal{F}_z$, let $dm_z(\Omega)$ be some normalized measure defined on $\mathcal{F}_z$ with $\Omega\in \mathcal{F}_z$. This corresponds to the invariant measure defined in Assumption \ref{Assumption:FastVectorMixingFullCoadjoint}. Then for any function (observable)
$\phi: \mathfrak{su}(N)\to \mathbb{R}$, we define the leafwise averaging operator by
\begin{equation}\label{Eq:AveragingLeafwiseCoAdjointsuN}
\overline{\phi}(z)
=
\int_{\mathcal O^*(z)}
\phi(\Omega)\,
dm_z(\Omega)\ , \ z\in \mathcal{A}_{\text{regular}} \ ,
\end{equation}
i.e., the averaging procedure consists of integrating out the fast motion
along each regular stratum (coadjoint leaf) while keeping the transverse Casimir variables $z$ fixed. 

Define the slow process
\begin{equation}\label{Eq:EulerArnoldZeitlinTruncatedDissipatedForcedWithFastNoise:EssentialSlowMotionPart}
d\Omega_t^{\text{slow}} =  \rho \sum\limits_{k=1}^l Y_k(\Omega_t^{\text{slow}})\circ dB_t^{(k)}+B(\Omega_t^{\text{slow}})dt+f(t)dt+
\delta d\mathfrak{b}_t \ ,
\end{equation}
which is the part of the process (\ref{Eq:EulerArnoldZeitlinTruncatedDissipatedForcedWithFastNoise:Specific}) without the $\dfrac{1}{\varepsilon}\sum\limits_{j=1}^m X_j(\Omega)\circ dW_t^{(j)}$ term. The slow process (\ref{Eq:EulerArnoldZeitlinTruncatedDissipatedForcedWithFastNoise:EssentialSlowMotionPart}) is the part of (\ref{Eq:EulerArnoldZeitlinTruncatedDissipatedForcedWithFastNoise:Specific}) that is ``essentially slow" at the asymptotic when $\varepsilon\rightarrow 0$. Using the It\^{o}'s formula for (\ref{Eq:EulerArnoldZeitlinTruncatedDissipatedForcedWithFastNoise:Specific}), i.e., the equation (\ref{Thm:ErgodicityEulerArnoldZeitlinTruncatedDissipatedForcedWithFastNoise:Eq:ItoFormula}) applied to a general function $V\in \mathbf{C}^{(2)}_{\su(N)}(\mathbb{R})$, as well as Lemma \ref{Lm:StratonovichToIto} to expand Stratonovich stochastic differential into It\^{o} form, we can write down the generator of $\Omega_t^{\text{slow}}$ as

\begin{equation}\label{Eq:EulerArnoldZeitlinTruncatedDissipatedForcedWithFastNoise:EssentialSlowMotionPartGenerator}
\begin{array}{ll}
& L_{\text{slow}} V(\Omega) 
\\
= &  \dfrac{\rho^2}{2}\sum\limits_{k=1}^l\left[[D^2 V(\Omega)](Y_k(\Omega), Y_k(\Omega))+(\nabla V(\Omega), [DY_k(\Omega)](Y_k(\Omega)))\right]
\\
& \qquad + \left(\nabla V(\Omega), [-A^{-1}\Omega, \Omega]-\alpha \Omega +\nu(-A)\Omega\right)
\\
& \qquad + \dfrac{\delta^2}{2}\sum\limits_{k\in \mathcal{K}_N^+} \left([D^2 V(\Omega)](L_{k, \cos N}, L_{k, \cos N})+[D^2 V(\Omega)](L_{k, \sin, N}, L_{k, \sin, N})\right)
\\
& \qquad + \dfrac{1}{2}\sum\limits_{k\in \mathcal{K}_N^+, |k|\approx k_f}q_k^2\left([D^2 V(\Omega)](L_{k, \cos N}, L_{k, \cos N})+[D^2 V(\Omega)](L_{k, \sin, N}, L_{k, \sin, N})\right) \ .
\end{array}
\end{equation}

Define the fast process
\begin{equation}\label{Eq:EulerArnoldZeitlinTruncatedDissipatedForcedWithFastNoise:FastMotionPart}
d\Omega_t^{\text{fast}}=\sum\limits_{j=1}^m X_j(\Omega_t^{\text{fast}})\circ dW_t^{(j)} \ ,
\end{equation}
which is the driving process of the term $\dfrac{1}{\varepsilon}\sum\limits_{j=1}^m X_j(\Omega)\circ dW_t^{(j)}$ in (\ref{Eq:EulerArnoldZeitlinTruncatedDissipatedForcedWithFastNoise:Specific}). When $\varepsilon\rightarrow 0$, this term gives the fast motion of (\ref{Eq:EulerArnoldZeitlinTruncatedDissipatedForcedWithFastNoise:Specific}). Using the It\^{o}'s formula for (\ref{Eq:EulerArnoldZeitlinTruncatedDissipatedForcedWithFastNoise:Specific}), i.e., the equation (\ref{Thm:ErgodicityEulerArnoldZeitlinTruncatedDissipatedForcedWithFastNoise:Eq:ItoFormula}) applied to a general function $V\in \mathbf{C}^{(2)}_{\su(N)}(\mathbb{R})$, as well as Lemma \ref{Lm:StratonovichToIto:X} to expand Stratonovich stochastic differential into It\^{o} form, we can write down the generator of $\Omega_t^{\text{fast}}$ as

\begin{equation}\label{Eq:EulerArnoldZeitlinTruncatedDissipatedForcedWithFastNoise:FastMotionPartGenerator}
L_{\text{fast}}V(\Omega)=\dfrac{1}{2}\sum\limits_{j=1}^m \left[[D^2 V(\Omega)](X_j(\Omega), X_j(\Omega))+(\nabla V(\Omega), [DX_j(\Omega)](X_j(\Omega)))\right] \ .
\end{equation}

Combining (\ref{Eq:EulerArnoldZeitlinTruncatedDissipatedForcedWithFastNoise:FastMotionPartGenerator}) with (\ref{Eq:EulerArnoldZeitlinTruncatedDissipatedForcedWithFastNoise:EssentialSlowMotionPartGenerator}), the generator of the process $\Omega_t^{\varepsilon,\delta}$ in (\ref{Eq:EulerArnoldZeitlinTruncatedDissipatedForcedWithFastNoise:Specific}) is written as

\begin{equation}\label{Eq:EulerArnoldZeitlinTruncatedDissipatedForcedWithFastNoise:GeneratorFastSlow}
L^{\varepsilon,\delta}=\dfrac{1}{\varepsilon^2}L_{\text{fast}}+L_{\text{slow}} \ .
\end{equation}

Given (\ref{Eq:EulerArnoldZeitlinTruncatedDissipatedForcedWithFastNoise:EssentialSlowMotionPartGenerator}), the averaged process $Z_t^\delta$ on the foliation parameter space $\mathcal{A}$, before its first hitting of $\mathcal{A}_\text{singular}$, is characterized by the generator
\begin{equation}\label{Eq:AveragingPrincipleEffectiveSlowGenerator}
\bar{L}\varphi(z)
= \overline{L_\text{slow}(\varphi\circ \pi)}(z) \stackrel{\text{ by } (\ref{Eq:AveragingLeafwiseCoAdjointsuN})}{=}
\int_{\mathcal{O}^*(z)}
L_{\text{slow}}(\varphi\circ\pi)(\Omega)dm_z(\Omega) \ ,
\end{equation}
where $\varphi\in \mathbf{C}^{(2)}_{\mathcal{A}}(\mathbb{R})$. 

\begin{lemma}\label{Lm:FirstAndSecondFrechetDerivativeCasimirProjection} 
For any $\Omega\in \su(N)$ and $\Theta, \Theta_1, \Theta_2\in \su(N)$ we have
\begin{equation}\label{Lm:FirstAndSecondFrechetDerivativeCasimirProjection:Eq:FirstDerivative}
[D c_k(\Omega)](\Theta)=(k+1)\operatorname{tr}(\Omega^k \Theta) \ , \ k = 1,...,N-1 \ ,
\end{equation}
and
\begin{equation}\label{Lm:FirstAndSecondFrechetDerivativeCasimirProjection:Eq:SecondDerivative}
[D^2 c_k(\Omega)](\Theta_1, \Theta_2)=(k+1)\sum\limits_{r=0}^{k-1} \operatorname{tr}(\Omega^r \Theta_2 \Omega^{k-1-r} \Theta_1) \ , \ k = 1,...,N-1 \ .
\end{equation}

\end{lemma}

\begin{proof}
We calculate
$[D c_k(\Omega)](\Theta) = \left.\dfrac{d}{dt}\right|_{t=0} \operatorname{tr}(\Omega+t\Theta)^{k+1}
= \operatorname{tr}\left(\left.\dfrac{d}{d\varepsilon}\right|_{\varepsilon=0}(\Omega+\varepsilon \Theta)^{k+1}\right)
= \operatorname{tr}\left(\sum\limits_{r=0}^k \Omega^r \Theta \Omega^{k-r}\right)
= \sum\limits_{r=0}^k\operatorname{tr}\left( \Omega^r \Theta \Omega^{k-r}\right)$ which is (\ref{Lm:FirstAndSecondFrechetDerivativeCasimirProjection:Eq:FirstDerivative}) due to $\operatorname{tr}(AB)=\operatorname{tr}(BA)$. Based on this, we further calculate $[D^2 c_k(\Omega)](\Theta_1, \Theta_2)=\left.\dfrac{d}{ds}\right|_{s=0} [D c_k(\Omega+s\Theta_2)](\Theta_1)=\left.\dfrac{d}{ds}\right|_{s=0} (k+1)\operatorname{tr}((\Omega+s\Theta_2)^k \Theta_1)=(k+1)\operatorname{tr}\left(\left.\dfrac{d}{d\varepsilon}\right|_{\varepsilon=0}(\Omega+\varepsilon \Theta_2)^k \Theta_1\right)$ $=(k+1)\operatorname{tr}\left(\sum\limits_{r=0}^{k-1} \Omega^r \Theta_2 \Omega^{k-1-r} \Theta_1\right)$ which is (\ref{Lm:FirstAndSecondFrechetDerivativeCasimirProjection:Eq:SecondDerivative}) due to $\operatorname{tr}(A+B)=\operatorname{tr}(A)+\operatorname{tr}(B)$.
\end{proof} 

By (\ref{Eq:suNCoAdjointFoliationProjectionMap}), we have $\pi(\Omega)=(c_1(\Omega),...,c_{N-1}(\Omega))=z$. Since $$(\varphi\circ \pi)(\Omega)=\varphi(c_1(\Omega),...,c_{N-1}(\Omega)) \ ,$$
therefore for $\Omega, \Theta, \Theta_1, \Theta_2\in \su(N)$ we have that
\begin{equation}\label{Eq:FirstFrechetDerivativephiComposeCasimirProjection}
[D(\varphi\circ \pi)(\Omega)](\Theta) = \sum\limits_{k=1}^{N-1} \dfrac{\partial \varphi(z)}{\partial z_k}[D c_k(\Omega)](\Theta)
\ , \ z = \pi(\Omega) \ ,
\end{equation}
and
\begin{equation}\label{Eq:SecondFrechetDerivativephiComposeCasimirProjection}
\begin{array}{ll}
[D^2(\varphi\circ \pi)(\Omega)](\Theta_1, \Theta_2) & = \sum\limits_{k_1, k_2=1}^{N-1} \dfrac{\partial^2 \varphi(z)}{\partial z_{k_1}\partial z_{k_2}}[D c_{k_1}(\Omega)](\Theta_1)[D c_{k_2}(\Omega)](\Theta_2)
\\
& \qquad + \sum\limits_{k=1}^{N-1} \dfrac{\partial \varphi(z)}{\partial z_k}[D^2 c_k(\Omega)](\Theta_1, \Theta_2) \ , \ z= \pi(\Omega) \ .
\end{array}
\end{equation}

We put $V(\Omega)=(\varphi \circ \pi)(\Omega)$ in (\ref{Eq:EulerArnoldZeitlinTruncatedDissipatedForcedWithFastNoise:EssentialSlowMotionPartGenerator}) and make use of (\ref{Eq:FirstFrechetDerivativephiComposeCasimirProjection}), (\ref{Eq:SecondFrechetDerivativephiComposeCasimirProjection}), to have the explicit form of the operator $\bar{L}$ in (\ref{Eq:AveragingPrincipleEffectiveSlowGenerator}) as
\begin{equation}\label{Eq:AveragingPrincipleEffectiveSlowGeneratorExplicit}
\bar{L}\varphi (z)= \sum\limits_{k=1}^{N-1} \overline{b}_k(z)\dfrac{\partial \varphi(z)}{\partial z_k} + \dfrac{1}{2}\sum\limits_{k_1, k_2=1}^{N-1} \overline{a}_{k_1k_2}(z) \dfrac{\partial^2 \varphi(z)}{\partial z_{k_1}\partial z_{k_2}} \ .
\end{equation}
Here
\begin{equation}\label{Eq:AveragingPrincipleEffectiveSlowGeneratorExplicit:Drift}
\overline{b}_k(z)= \int_{\mathcal{O}^*(z)} L_{\text{slow}} c_k(\Omega) dm_z(\Omega) \ ,
\end{equation}
and
\begin{equation}\label{Eq:AveragingPrincipleEffectiveSlowGeneratorExplicit:Diffusion}
\overline{a}_{k_1k_2}(z)= \int_{\mathcal{O}^*(z)} \Gamma_{\text{slow}} (c_{k_1}, c_{k_2})(\Omega) dm_z(\Omega) \ .
\end{equation}
The integrand $L_{\text{slow}} c_k (\Omega)$ is as in (\ref{Eq:EulerArnoldZeitlinTruncatedDissipatedForcedWithFastNoise:EssentialSlowMotionPartGenerator}) with $V(\Omega)=c_k(\Omega)$. The integrand $\Gamma_{\text{slow}}(c_{k_1}, c_{k_2})(\Omega)$ is defined as
\begin{equation}\label{Eq:AveragingPrincipleEffectiveSlowGeneratorExplicit:GammaSlow}
\begin{array}{ll}
& \Gamma_{\text{slow}}(c_{k_1}, c_{k_2})(\Omega) 
\\
= & \rho^2 \sum\limits_{k=1}^{l} [D c_{k_1}(\Omega)](Y_k(\Omega))[D c_{k_2}(\Omega)](Y_k(\Omega))
\\
& \qquad + \delta^2\sum\limits_{k\in \mathcal{K}_N^+} \left([D c_{k_1}(\Omega)](L_{k, \cos N})[D c_{k_2}(\Omega)](L_{k, \cos N})\right.
\\
& \qquad \qquad \qquad \qquad \left.+[D c_{k_1}(\Omega)](L_{k, \sin, N})[D c_{k_2}(\Omega)](L_{k, \sin, N})\right)
\\
& \qquad + \sum\limits_{k\in \mathcal{K}_N^+, |k|\approx k_f}q_k^2\left([D c_{k_1}(\Omega)](L_{k, \cos N})[D c_{k_2}(\Omega)](L_{k, \cos N})\right.
\\
& \qquad \qquad \qquad \qquad \qquad \left.+[D c_{k_1}(\Omega)](L_{k, \sin, N})[D c_{k_2}(\Omega)](L_{k, \sin, N})\right) \ .
\end{array}
\end{equation}
One can use equations (\ref{Lm:FirstAndSecondFrechetDerivativeCasimirProjection:Eq:FirstDerivative}) and (\ref{Lm:FirstAndSecondFrechetDerivativeCasimirProjection:Eq:SecondDerivative}) in Lemma \ref{Lm:FirstAndSecondFrechetDerivativeCasimirProjection} to replace the terms of $Dc_k(\Omega)$, $D c_{k_1}(\Omega)$, $D c_{k_2}(\Omega)$ in $L_{\text{slow}}$ and $\Gamma_{\text{slow}}$ to further obtain more explicit formulae for $\bar{L}\varphi (z)$ in terms of $\Omega$. 

By (\ref{Eq:AveragingPrincipleEffectiveSlowGeneratorExplicit}), the averaged
process $Z_t^\delta=(Z_{1,t}^{\delta}, ..., Z_{N-1, t}^\delta)$ can be equivalently written in an It\^{o}'s stochastic differential equation form as
\begin{equation}\label{Eq:AveragingPrincipleEffectiveSlowProcessItoFormExplicit}
dZ_t^\delta=\overline{b}(Z_t^\delta)dt+\overline{\sigma}(Z_t^\delta)dw_t \ ,
\end{equation}
where $\overline{b}(z)=(\overline{b}_k(z))_{1\leq k\leq N-1}$ and  $\overline{\sigma}(z)=(\overline{\sigma}_{k_1k_2}(z))_{1\leq k_1, k_2\leq N-1}$ is such that   $\overline{a}(z)=\overline{\sigma(z)}\cdot \overline{\sigma(z)^T}$ with $\overline{a}(z)=(\overline{a}_{k_1k_2}(z))_{1\leq k_1, k_2\leq N-1}$; and $w_t$ is an $(N-1)$-dimensional Brownian motion. The terms $\overline{b}_k(z)$ and $\overline{a}_{k_1k_2}(z)$ are defined as in (\ref{Eq:AveragingPrincipleEffectiveSlowGeneratorExplicit:Drift}), (\ref{Eq:AveragingPrincipleEffectiveSlowGeneratorExplicit:Diffusion}), respectively.

Consider the process $\Omega_t^{\varepsilon, \delta}$ defined in (\ref{Eq:EulerArnoldZeitlinTruncatedDissipatedForcedWithFastNoise:Specific}) and let $\Omega_0^{\varepsilon, \delta}=\Omega_0$ be such that $Z_0=\pi(\Omega_0) \in \mathcal{A}_{\text{regular}}^\circ$, which is the interior of $\mathcal{A}_{\text{regular}}$, i.e., there is an open neighborhood $\mathcal{U}_0$ of $Z_0$ such that $\mathcal{U}_0 \subset \mathcal{A}_{\text{regular}}$. Since $\Omega_t^{\varepsilon, \delta}$ is a strong Markov process on $\su(N)$, we can define  the hitting time 
\begin{equation}\label{Eq:HittingTimeToBoundaryOpenNeighborhoodRegularStrata}
\tau^{\varepsilon, \delta}=\inf\{t\geq 0: \Omega_t^{\varepsilon,\delta}\in \partial (\pi^{-1}(\mathcal{U}_0))\}
\end{equation} 
to be the first time $\Omega_t^{\varepsilon, \delta}$ hits $\partial (\pi^{-1}(\mathcal{U}_0))$.
Similarly, for the process $Z_t^\delta$ defined by the generator (\ref{Eq:AveragingPrincipleEffectiveSlowGeneratorExplicit}) or the It\^{o}'s stochastic differential equation (\ref{Eq:AveragingPrincipleEffectiveSlowProcessItoFormExplicit}), we define the hitting time
\begin{equation}\label{Eq:HittingTimeToBoundaryOpenNeighborhoodRegularCasimir}
\tau^\delta = \inf\{t\geq 0: Z_t^\delta \in \partial \mathcal{U}_0 \}
\end{equation}
to be the first time $Z_t^\delta$ hits $\partial \mathcal{U}_0$. Let $\mathbf{C}_{[0,T]}(\mathcal{A})$ be the space of continuous trajectories on $\mathcal{A}$.

\begin{theorem}[Averaging principle via weak convergence of projected processes]\label{Thm:AveragingPrinciple}
As $\varepsilon \rightarrow 0$, the family of (non-Markov) processes $\zeta_t^{\varepsilon,\delta}=\pi(\Omega^{\varepsilon,\delta}_{t\wedge \tau^{\varepsilon,\delta}})$ converges weakly to $Z_{t\wedge \tau^\delta}^\delta$ in $\mathbf{C}_{[0,T]}(\mathcal{A})$. In other words, for any bounded continuous functional $F$ on $\mathbf{C}_{[0,T]}(\mathcal{A})$ we have, as $\varepsilon \rightarrow 0$, \begin{equation}\label{Thm:AveragingPrinciple:Eq:WeakConvergenceBoundedFunctional}
\mathbf{E}_{\Omega^{\varepsilon,\delta}_0=\Omega_0} F(\pi(\Omega^{\varepsilon,\delta}_{t\wedge \tau^{\varepsilon,\delta}}))\rightarrow \mathbf{E}_{Z_0^\delta=\pi(\Omega_0)} F(Z_{t\wedge \tau^\delta}^\delta) \ .
\end{equation}
\end{theorem}

\begin{proof}
The result is a standard application of the averaging principle with fast motion on period components, see \cite[Theorem 3]{[FWAveragingOpenBook]}, \cite{[Khasminskii1968]}, \cite{[AveragingDiffusionFoliatedSpacesAOP2016]}, \cite{[XuemeiLi2008]}, which can be made using martingale problem formulation of Markov processes. Here we apply the averaging principle to the stopped process $\Omega^{\varepsilon,\delta}_{t\wedge \tau^{\varepsilon,\delta}}$ upon hitting $\partial(\pi^{-1}(\mathcal{U}_0))$. Note that since $\mathcal{U}_0\subset \mathcal{A}_{\text{regular}}$, we do not have to consider the averaging behavior near $\mathcal{A}_{\text{singular}}$. 
\end{proof}

We define the process $Z_t$ as $Z_t^\delta$ with $\delta=0$ and set the stopping time \begin{equation}\label{Eq:HittingTimeToBoundaryOpenNeighborhoodRegularCasimir:DeltaVanish}
\tau = \inf\{t\geq 0: Z_t \in \partial \mathcal{U}_0 \}
\end{equation} 
to be the first time $Z_t$ hits $\partial \mathcal{U}_0$. Then we have

\begin{corollary}\label{Cor:WeakConvergenceToZ}
For any bounded continuous functional $F$ on $\mathbf{C}_{[0,T]}(\mathcal{A})$ we have, as $\varepsilon \rightarrow 0$ and then $\delta\rightarrow 0$, \begin{equation}\label{Cor:WeakConvergenceToZ:Eq:WeakConvergenceBoundedFunctional}
\mathbf{E}_{\Omega^{\varepsilon,\delta}_0=\Omega_0} F(\pi(\Omega^{\varepsilon,\delta}_{t\wedge \tau^{\varepsilon,\delta}}))\rightarrow \mathbf{E}_{Z_0=\pi(\Omega_0)} F(Z_{t\wedge \tau}) \ .
\end{equation}
\end{corollary}

\begin{proof}
By (\ref{Eq:EulerArnoldZeitlinTruncatedDissipatedForcedWithFastNoise:EssentialSlowMotionPartGenerator}), (\ref{Eq:AveragingPrincipleEffectiveSlowGeneratorExplicit}), (\ref{Eq:AveragingPrincipleEffectiveSlowGeneratorExplicit:Drift}), (\ref{Eq:AveragingPrincipleEffectiveSlowGeneratorExplicit:Diffusion}), (\ref{Eq:AveragingPrincipleEffectiveSlowGeneratorExplicit:GammaSlow}) it is clear that as $\delta \rightarrow 0$, the process $Z_t^\delta$ converges weakly to $Z_t$ before the first hitting time to $\partial \mathcal{U}_0$. Thus (\ref{Cor:WeakConvergenceToZ:Eq:WeakConvergenceBoundedFunctional}) is a direct consequence of (\ref{Thm:AveragingPrinciple:Eq:WeakConvergenceBoundedFunctional}).
\end{proof}

Moreover we have

\begin{theorem}[Averaging of observables along the fast coadjoint leaves]\label{Thm:BoundedFunctionConvergenceProcessToProjectedAverage}
For any smooth observable function $G: \su(N)\rightarrow \mathbb{R}$, any $T\geq 0$ we have
\begin{equation}\label{Thm:BoundedFunctionConvergenceProcessToProjectedAverage:Eq}
\lim\limits_{\varepsilon\rightarrow 0}\sup\limits_{0\leq t \leq T}\left|\mathbf{E}_{\Omega_0^{\varepsilon,\delta}=\Omega_0} \int_0^{t\wedge \tau^{\varepsilon,\delta}}G(\Omega_s^{\varepsilon,\delta})ds-\mathbf{E}_{\Omega_0^{\varepsilon,\delta}=\Omega_0} \int_0^{t\wedge \tau^{\varepsilon,\delta}}\overline{G}(\pi(\Omega_s^{\varepsilon,\delta}))ds\right|=0 \ ,
\end{equation}
where $\overline{G}(z)=\displaystyle{\int_{\mathcal{O}^*(z)}}G(\Omega)dm_z(\Omega)$ is defined in \emph{(\ref{Eq:AveragingLeafwiseCoAdjointsuN})}. 
\end{theorem}

\begin{proof}
Set
\begin{equation}\label{Thm:BoundedFunctionConvergenceProcessToProjectedAverage:Eq:CenteredG}
H(\Omega)
=
G(\Omega)-\overline G(\pi(\Omega)) \ .
\end{equation}
By (\ref{Eq:AveragingLeafwiseCoAdjointsuN}), for every $z\in \mathcal A_{\mathrm{regular}}$,
we have $\displaystyle{\int_{\mathcal O^*(z)}H(\Omega)dm_z(\Omega)=0}$. Let 
$K_0=\pi^{-1}(\mathcal{U}_0\cup \partial \mathcal{U}_0)$, which is a compact set. By Assumption \ref{Assumption:FastVectorMixingFullCoadjoint} and the standard solvability of the Poisson equation on each coadjoint orbit, there exists a smooth function
$\varphi_G$ on $K_0$ with $\displaystyle{\int_{\mathcal{O}^*(z)}}\varphi(\Omega)dm_z(\Omega)=0$ such that
$$L_{\mathrm{fast}}\varphi_G=H \ .$$
Moreover, on $K_0$, the functions $\varphi_G$ and
$L_{\mathrm{slow}}\varphi_G$ are bounded.

By (\ref{Eq:EulerArnoldZeitlinTruncatedDissipatedForcedWithFastNoise:GeneratorFastSlow}) and applying It\^{o}'s formula to
$\varphi_G(\Omega_{t\wedge\tau^{\varepsilon,\delta}}^{\varepsilon,\delta})$, we get
$$
\varphi_G(\Omega_{t\wedge\tau^{\varepsilon,\delta}}^{\varepsilon,\delta})
-\varphi_G(\Omega_0) =
\int_0^{t\wedge\tau^{\varepsilon,\delta}}
\left[
\frac{1}{\varepsilon^2}L_{\mathrm{fast}}\varphi_G
+
L_{\mathrm{slow}}\varphi_G
\right](\Omega_s^{\varepsilon,\delta})ds
+
M_t^\varepsilon \ ,$$
where $M_t^\varepsilon$ is a martingale. Since
$L_{\mathrm{fast}}\varphi_G=H$, this implies
$$
\int_0^{t\wedge\tau^{\varepsilon,\delta}}
H(\Omega_s^{\varepsilon,\delta})ds =
\varepsilon^2
\left[
\varphi_G(\Omega_{t\wedge\tau^{\varepsilon,\delta}}^{\varepsilon,\delta})
-\varphi_G(\Omega_0)
\right]
-
\varepsilon^2
\int_0^{t\wedge\tau^{\varepsilon,\delta}}
L_{\mathrm{slow}}\varphi_G(\Omega_s^{\varepsilon,\delta})ds
-
\varepsilon^2 M_t^\varepsilon \ .
$$

When taking expectations in the above, the martingale term disappears. Therefore,
$$
\left|
\mathbf{E}_{\Omega_0^{\varepsilon,\delta}=\Omega_0}
\int_0^{t\wedge\tau^{\varepsilon,\delta}}
H(\Omega_s^{\varepsilon,\delta})ds
\right| 
\leq 
\varepsilon^2
\mathbf{E}
\left|
\varphi_G(\Omega_{t\wedge\tau^{\varepsilon,\delta}}^{\varepsilon,\delta})
-\varphi_G(\Omega_0)
\right|
+
\varepsilon^2
\mathbf{E}
\int_0^{t\wedge\tau^{\varepsilon,\delta}}
\left|
L_{\mathrm{slow}}\varphi_G(\Omega_s^{\varepsilon,\delta})
\right|ds \ .
$$

Since $\varphi_G$ and $L_{\mathrm{slow}}\varphi_G$ are bounded on $K_0$, there is a
constant $C=C(G,T,K_0)$ such that, uniformly for $0\leq t\leq T$,
$$
\left|
\mathbf{E}_{\Omega_0^{\varepsilon,\delta}=\Omega_0}
\int_0^{t\wedge\tau^{\varepsilon,\delta}}
H(\Omega_s^{\varepsilon,\delta})ds
\right|
\leq
C\varepsilon^2 \ .
$$
Inserting (\ref{Thm:BoundedFunctionConvergenceProcessToProjectedAverage:Eq:CenteredG}) into the above and letting $\varepsilon\rightarrow 0$ we get (\ref{Thm:BoundedFunctionConvergenceProcessToProjectedAverage:Eq}).
\end{proof}

Combining Theorem \ref{Thm:BoundedFunctionConvergenceProcessToProjectedAverage} with Corollary \ref{Cor:WeakConvergenceToZ}, we get 

\begin{corollary}\label{Cor:BoundedFunctionConvergenceProcessToZ}
For any bounded smooth function $G: \su(N)\rightarrow \mathbb{R}$, any $t>0$ we have, as $\varepsilon\rightarrow 0$ and then $\delta \rightarrow 0$,
\begin{equation}\label{Cor:BoundedFunctionConvergenceProcessToZ:Eq}
\mathbf{E}_{\Omega_0^{\varepsilon,\delta}=\Omega_0} \int_0^{t\wedge \tau^{\varepsilon,\delta}} G(\Omega_s^{\varepsilon,\delta})ds \rightarrow \mathbf{E}_{Z_0=\pi(\Omega_0)} \int_0^{t\wedge \tau} \overline{G}(Z_s)ds \ .
\end{equation}
where $\overline{G}(z)=\displaystyle{\int_{\mathcal{O}^*(z)}}G(\Omega)dm_z(\Omega)$ is defined in \emph{(\ref{Eq:AveragingLeafwiseCoAdjointsuN})}. 
\end{corollary}

\begin{proof}
Apply (\ref{Thm:BoundedFunctionConvergenceProcessToProjectedAverage:Eq}) and then (\ref{Cor:WeakConvergenceToZ:Eq:WeakConvergenceBoundedFunctional}) with  $F=\overline{G}$, we get (\ref{Cor:BoundedFunctionConvergenceProcessToZ:Eq}).
\end{proof}

\subsection{Nonlinear fluxes for energy and enstrophy with their averaged limits}\label{Sec:AveragingFramework:EnergyEnstrophyFlux}

In parallel with (\ref{Eq:Kraichnan2dTurbulence:CumulativeLowFreqEnergyEnstropy}), for given $\Omega\in \su(N)$ with decomposition (\ref{Eq:ZeitlinExpansionVorticityAndStreamMatrices}), we introduce the Zeitlin approximation version of cumulative low-frequency energy and enstrophy as follows:

\begin{equation}\label{Eq:FluxZeitlin:CumulativeLowFreqEnergyEnstropy}
\mathcal{E}_{\leq K}(\Omega) =\dfrac{1}{2}\sum\limits_{|k|\leq K , k\in \mathcal{K}_N}\dfrac{|\Omega_k|^2}{|k|^2} 
\ \ , \ \
\mathcal{Z}_{\leq K}(\Omega) = \dfrac{1}{2}\sum\limits_{|k|\leq K, k\in \mathcal{K}_N}|\Omega_k|^2 \ .
\end{equation}

We consider cross-shell fluxes associated with $\mathcal{E}_{\leq K}$ and $\mathcal{Z}_{\leq K}$, which are the change rates of these quantities carried by the Euler-Arnold nonlinear term $[\Psi, \Omega]$ in (\ref{Eq:EulerArnoldZeitlinTruncatedDissipatedForcedWithFastNoise:Specific}) or in (\ref{Eq:EulerArnoldZeitlinTruncated}). This is in parallel with (\ref{Eq:Kraichnan2dTurbulence:EnergyFlux}) and (\ref{Eq:Kraichnan2dTurbulence:EnstrophyFlux}) that are of the $2$-d fluid case. 

\begin{definition}\emph{(Instantaneous Flux)}\label{Def:FluxZeitlin:InstantaneousFluxEnergyEnstropy}
We define the \emph{instantaneous flux} across the
shell level $K$ via the energy and enstrophy rates carried by the nonlinear dynamics \emph{(\ref{Eq:EulerArnoldZeitlinTruncated})}, i.e. 
\begin{equation}\label{Def:FluxZeitlin:InstantaneousFluxEnergyEnstrophy:Eq}
\Pi_\mathcal{E}(K; \Omega)=-[D\mathcal{E}_{\leq K}(\Omega)]([\Psi, \Omega]) \ , \ 
\Pi_\mathcal{Z}(K; \Omega)=-[D\mathcal{Z}_{\leq K}(\Omega)]([\Psi, \Omega]) \ , \end{equation}
where $[DF(\Omega)](\Theta)=\left.\dfrac{d}{ds}\right|_{s=0} F(\Omega+s\Theta)$ is the Fr\'{e}chet derivative of $F: \su(N)\rightarrow \mathbb{R}$ at $\Omega\in \su(N)$ in the direction $\Theta \in \su(N)$; $\Psi = -A^{-1}\Omega$ and $A$ is defined in \emph{(\ref{Eq:InertiaOperatorZeitlinCase})} and the sign convention is chosen so that a positive flux
corresponds to a net transfer of energy or enstrophy from the low modes
$|k|\leq K$ to the complementary modes $|k|>K$, similar as in \emph{(\ref{Eq:Kraichnan2dTurbulence:EnergyFlux})}, \emph{(\ref{Eq:Kraichnan2dTurbulence:EnstrophyFlux})}.
\end{definition}

\begin{definition}[Fast coadjoint leafwise averaged flux]\label{Def:FluxZeitlin:LeafwiseAveragedFluxEnergyEnstropy}
We define the \emph{fast coadjoint leafwise averaged flux} as applying the leafwise averaging operator \emph{(\ref{Eq:AveragingLeafwiseCoAdjointsuN})} to the instantaneous flux \emph{(\ref{Def:FluxZeitlin:InstantaneousFluxEnergyEnstrophy:Eq})}, i.e.
\begin{equation}\label{Def:FluxZeitlin:LeafwiseAveragedFluxEnergyEnstropy:Eq}
\overline{\Pi}_{\mathcal{E}}(K; z)=\int_{\mathcal{O}^*(z)}\Pi_{\mathcal{E}}(K; \Omega)dm_z(\Omega) \ , \ \overline{\Pi}_{\mathcal{Z}}(K; z)=\int_{\mathcal{O}^*(z)}\Pi_{\mathcal{Z}}(K; \Omega)dm_z(\Omega) \ .
\end{equation}
\end{definition}

Under the same notations as in Corollary \ref{Cor:BoundedFunctionConvergenceProcessToZ}, we can connect the instantaneous flux with the fast coadjoint leafwise averaged flux as the following result.

\begin{theorem}[Averaged flux convergence]\label{Thm:AveragedFluxConvergence}
For any $t>0$, as first $\varepsilon\rightarrow 0$ and then $\delta \rightarrow 0$ we have
\begin{equation}\label{Thm:AveragedFluxConvergence:Eq:Convergence:Energy}
\mathbf{E}_{\Omega_0^{\varepsilon,\delta}=\Omega_0}\int_0^{t\wedge \tau^{\varepsilon,\delta}}\Pi_{\mathcal{E}}(K;\Omega_s^{\varepsilon,\delta})ds \rightarrow \mathbf{E}_{Z_0=\pi(\Omega_0)}\int_0^{t\wedge \tau}\overline{\Pi}_{\mathcal{E}}(K; Z_s)ds \ , \ 
\end{equation}
and
\begin{equation}\label{Thm:AveragedFluxConvergence:Eq:Convergence:Enstrophy}
\mathbf{E}_{\Omega_0^{\varepsilon,\delta}=\Omega_0}\int_0^{t\wedge \tau^{\varepsilon,\delta}}\Pi_{\mathcal{Z}}(K; \Omega_s^{\varepsilon,\delta})ds \rightarrow \mathbf{E}_{Z_0=\pi(\Omega_0)}\int_0^{t\wedge \tau}\overline{\Pi}_{\mathcal{Z}}(K; Z_s)ds \ .
\end{equation}
\end{theorem}

\begin{proof}
We can apply Corollary \ref{Cor:BoundedFunctionConvergenceProcessToZ} to $G(\Omega)=\Pi_{\mathcal{E}}(K; \Omega)$ and $G(\Omega)=\Pi_{\mathcal{Z}}(K; \Omega)$, since these functions are bounded within $\mathcal{U}_0$. 
\end{proof}

In parallel with the original stationary nonlinear flux for Kraichnan's scenario, defined in (\ref{Eq:Kraichnan2dTurbulence:EnergyFlux}) and (\ref{Eq:Kraichnan2dTurbulence:EnstrophyFlux}), we also introduce the following

\begin{definition}\emph{(Stationary Flux)}\label{Def:FluxZeitlin:StationaryFluxEnergyEnstropy}
We define the \emph{stationary flux} as the expected instantaneous flux \emph{(\ref{Def:FluxZeitlin:InstantaneousFluxEnergyEnstrophy:Eq})} at $\Omega \sim \mu_{\varepsilon, \delta, \alpha, \nu}^{(N)}$ where $\mu_{\varepsilon, \delta, \alpha, \nu}^{(N)}$ is the invariant measure of $\Omega_t^{\varepsilon ,\delta}$ in \emph{Theorem \ref{Thm:ErgodicityEulerArnoldZeitlinTruncatedDissipatedForcedWithFastNoise}}, i.e.
\begin{equation}\label{Def:FluxZeitlin:StationaryFluxEnergyEnstropy:Eq}
\Pi_\mathcal{E}(K)=\mathbf{E}_{\Omega\sim \mu_{\varepsilon, \delta, \alpha, \nu}^{(N)}}\Pi_\mathcal{E}(K; \Omega)
 \ , \ 
\Pi_\mathcal{Z}(K)=\mathbf{E}_{\Omega\sim \mu_{\varepsilon, \delta, \alpha, \nu}^{(N)}}\Pi_\mathcal{Z}(K; \Omega) \ . 
\end{equation}
\end{definition}

The averaged flux convergence result, i.e., Theorem \ref{Thm:AveragedFluxConvergence} stated above is a local-in-stratum, stopped finite-time result. It should not be confused with a convergence theorem for the stationary fluxes in Definition \ref{Def:FluxZeitlin:StationaryFluxEnergyEnstropy}. Passing to stationary fluxes would require a global averaging result for the invariant measures $\mu^{(N)}_{\varepsilon,\delta,\alpha,\nu}$, including the behavior of the averaged dynamics near the boundary of $\mathcal{A}_{\text{regular}}$  and across possible singular strata parameterized by $\mathcal{A}_{\text{singular}}$. In the present work we do not characterize such boundary conditions. Therefore, the stationary fluxes are introduced as the natural objects corresponding to the Kraichnan setting, while the rigorous averaging statements we have here are formulated only for stopped dynamics inside a regular stratum. We will discuss the boundary setting and the averaging for stationary flux in a future work.

We note that although the above notions of energy and enstrophy fluxes are referring to the cross-shell transfer carried by the Euler-Arnold nonlinearity, in the real dynamics of (\ref{Eq:EulerArnoldZeitlinTruncatedDissipatedForcedWithFastNoise:Specific}), each of the other terms such as the fast stochastic forcing $\mathscr{F}$, the damping $-\alpha\Omega$, the viscosity $\nu(-A)\Omega$, and the external stochastic forcing $f$ all enter the full shell balance of the process and change the shell-wise energy or enstrophy budget. We will analyze the averaged flux contributions of each of these terms in Sections \ref{Sec:HeatingUpAllCoAdjoint} and \ref{Sec:SymmetryBreakingCoAdjoint}.

\section{Fully symmetric heating up of the whole coadjoint orbit} \label{Sec:HeatingUpAllCoAdjoint}

In this section, we consider the full symmetric case, where the fast stochastic forcing directions $X_j(\Omega)$ in (\ref{Eq:EulerArnoldZeitlinTruncatedDissipatedForcedWithFastNoise:FastPart}) are given by Hamiltonian vector fields on the coadjoint orbit, so that they preserve the KKS form and thus the Liouville volume measure $m_{\mathrm{KKS}}$. To this end we make the following assumption throughout this section.

\begin{assumption}[Hamiltonian fast mixing structure on coadjoint orbits]\label{Assumption:FastVectorHamiltonianFullCoadjoint}
Under \emph{Assumption \ref{Assumption:FastVectorMixingFullCoadjoint}}, we further assume that the tangent fast vector fields $X_j(\Omega)$ are Hamiltonian vector fields on coadjoint orbits of $SU(N)$. Given
smooth Hamiltonian functions $h_j$ on $\mathfrak{su}(N)$, we define, by \emph{(\ref{Eq:HamiltonianVectorFieldExplicitFormCoAdjointOrbitSUN})}, that
\begin{equation}\label{Assumption:FastVectorHamiltonianFullCoadjoint:Eq:HamiltonianTangentVector}
X_j(\Omega)=X_{h_j}(\Omega)=[\Omega, \nabla h_j(\Omega)] \ , \ j=1,...,m \ .
\end{equation}
As a special case, when  $h_j(\Omega)=( \beta_j,\Omega)$, this gives
\(X_j(\Omega)=[\Omega, \beta_j]\) for some $\beta_j\in \su(N)$.
\end{assumption}

Assumption \ref{Assumption:FastVectorHamiltonianFullCoadjoint} (which includes Assumption \ref{Assumption:FastVectorMixingFullCoadjoint}) ensures that the $X$ part of the fast stochastic forcing $\mathscr{F}(\Omega_t^{\varepsilon,\delta})dt$ in (\ref{Eq:EulerArnoldZeitlinTruncatedDissipatedForcedWithFastNoise:FastPart}), which is given by the process $\Omega_t^X$ in (\ref{Eq:EulerArnoldZeitlinTruncatedDissipatedForcedWithFastNoise:FastPart:X}), admits the Liouville volume measure $m_{\mathrm{KKS}}$ as the unique invariant measure on the coadjoint orbit. The next Lemma asserts that the leafwise averaged flux of any smooth observable $Q: \su(N)\rightarrow \mathbb{R}$ under the measure $m_{\mathrm{KKS}}$ vanishes.

\begin{lemma}[KKS average of the Euler-Arnold derivative vanishes]\label{Lm:LeafAverageVanishKKSMeasure}
For any $z$ in the interior of $\mathcal{A}_{\text{regular}}$ and any smooth observable $Q: \su(N)\rightarrow \mathbb{R}$, we have
\begin{equation}\label{Lm:LeafAverageVanishKKSMeasure:Eq}
\int_{\mathcal{O}^*(z)}[D Q(\Omega)]([\Psi, \Omega])dm_{\mathrm{KKS},z}(\Omega)=0 \ ,
\end{equation}
where $\Psi = -A^{-1}\Omega$ with the operator $A$ defined as in \emph{(\ref{Eq:InertiaOperatorZeitlinCase})}.
\end{lemma}

\begin{proof}
Fix \(z\in \mathcal A_{\mathrm{regular}}^\circ\). Then
\(\mathcal O^*(z)\) is a compact smooth coadjoint orbit equipped with the
KKS symplectic form and the corresponding normalized Liouville measure
\(m_{\mathrm{KKS},z}\). Equations (\ref{Eq:CoAdjointHamiltonian}) and (\ref{Eq:HamiltonianEqViaKirillovForm}) have shown the Hamiltonian structure of the Euler-Arnold nonlinear term for general Lie group $G$. In the special case $G=SU(N)$, we consider the Hamiltonian
$\mathscr{H}(\Omega)=\dfrac{1}{2}(\Omega,A^{-1}\Omega)\ , \ \Omega\in \mathcal O^*(z)$.
Since \(A\) is symmetric with respect to the inner product on \(\mathfrak{su}(N)\),
we have
$\nabla \mathscr{H}(\Omega)=A^{-1}\Omega$.
By (\ref{Eq:HamiltonianVectorFieldExplicitFormCoAdjointOrbitSUN}),
the Hamiltonian vector field associated with $\mathscr{H}$ is given by 
$X_{\mathscr{H}}(\Omega)
=
[\Omega,\nabla \mathscr{H}(\Omega)]
=
[\Omega,A^{-1}\Omega]
=
[-A^{-1}\Omega,\Omega]
=
[\Psi,\Omega]$ where we have used $\Psi=-A^{-1}\Omega$.
Thus the vector field \([\Psi,\Omega]\) is Hamiltonian on \(\mathcal O^*(z)\). Let \(\Phi_t\) be the flow generated by \(X_{\mathscr{H}}\) on \(\mathcal O^*(z)\). Hamiltonian
flows preserve the KKS symplectic form and therefore preserve the associated
Liouville measure \(m_{\mathrm{KKS},z}\). Hence, for every smooth observable
\(Q:\mathfrak{su}(N)\to \mathbb R\),
\[
\int_{\mathcal O^*(z)} Q(\Phi_t(\Omega))\,dm_{\mathrm{KKS},z}(\Omega)
=
\int_{\mathcal O^*(z)} Q(\Omega)\,dm_{\mathrm{KKS},z}(\Omega).
\]
Differentiating this identity at \(t=0\), and using the compactness of
\(\mathcal O^*(z)\), this gives
\[
0
=
\int_{\mathcal O^*(z)}
[DQ(\Omega)](X_{\mathscr{H}}(\Omega))\,dm_{\mathrm{KKS},z}(\Omega) \ ,
\]
which is (\ref{Lm:LeafAverageVanishKKSMeasure:Eq}).
\end{proof}

\begin{corollary}\label{Cor:LeafAverageVanishKKSMeasure:EnergyEnstrophy}
For any $z$ in the interior of \emph{$\mathcal{A}_{\text{regular}}$}, as $dm_z(\Omega)=dm_{\mathrm{KKS},z}(\Omega)$, the fast coadjoint leafwise averaged fluxes \emph{(\ref{Def:FluxZeitlin:LeafwiseAveragedFluxEnergyEnstropy:Eq})} for energy and enstrophy both vanish, i.e.
\begin{equation}\label{Cor:LeafAverageVanishKKSMeasure:EnergyEnstrophy:Eq}
\overline{\Pi}_{\mathcal{E}}(K; z)=\overline{\Pi}_{\mathcal{Z}}(K; z)=0 \ .
\end{equation}
\end{corollary}

\begin{proof}
We apply Lemma \ref{Lm:LeafAverageVanishKKSMeasure} to $Q(\Omega)=\mathcal{E}_{\leq K}(\Omega)$ and $Q(\Omega)=\mathcal{Z}_{\leq K}(\Omega)$ defined in (\ref{Eq:FluxZeitlin:CumulativeLowFreqEnergyEnstropy}).
\end{proof}

\begin{corollary}\label{Cor:AveragedNonlinearFluxConvergenceZero}
For any $t>0$, we have
\begin{equation}\label{Cor:AveragedNonlinearFluxConvergenceZero:Eq}
\lim\limits_{\delta\rightarrow 0}\lim\limits_{\varepsilon\rightarrow 0}\mathbf{E}_{\Omega_0^{\varepsilon,\delta}=\Omega_0}\int_0^{t\wedge \tau^{\varepsilon,\delta}}\Pi_\mathcal{E}(K; \Omega_s^{\varepsilon,\delta})ds=\lim\limits_{\delta\rightarrow 0}\lim\limits_{\varepsilon\rightarrow 0}\mathbf{E}_{\Omega_0^{\varepsilon,\delta}=\Omega_0}\int_0^{t\wedge \tau^{\varepsilon,\delta}}\Pi_\mathcal{Z}(K; \Omega_s^{\varepsilon,\delta})ds=0 \ .
\end{equation}
\end{corollary}

\begin{proof}
This is a direct consequence by combining Corollary \ref{Cor:LeafAverageVanishKKSMeasure:EnergyEnstrophy} with Theorem \ref{Thm:AveragedFluxConvergence}.
\end{proof}

To investigate the cross-shell transfer of energy and enstrophy contributed by each of the other terms in (\ref{Eq:EulerArnoldZeitlinTruncatedDissipatedForcedWithFastNoise:Specific}), we first establish the following general orbit average result using the idea of Schur's Lemma from representation theory. 

\begin{lemma}[Orbit average of quadratic form on $\su(N)$]\label{Lm:OrbitAveragesuN} 
Let $\mathcal{O}^*(\Lambda)$ be the $\su(N)$ coadjoint orbit passing through $\Lambda=\text{diag}(i\lambda_1,...,i\lambda_N)$ for $\Lambda\in \su(N)$, as defined in \emph{(\ref{Eq:CoAdjointOrbitsuNIsospectral})}. For any $\Omega\in \mathcal{O}^*(\Lambda)$, consider the decomposition \emph{(\ref{Eq:ZeitlinExpansionVorticityAndStreamMatrices})}. Then we have
\begin{equation}\label{Lm:OrbitAveragesuN:Eq:Average}
\int_{\mathcal{O}^*(\Lambda)}\Omega_k\Omega_l dm_{\mathrm{KKS}}(\Omega)=\dfrac{\|\Lambda\|^2}{c_N(N^2-1)}\delta_{k+l, 0} \ ,
\end{equation}
where $dm_{\mathrm{KKS}}(\Omega)=dm_{\mathrm{KKS}, \mathcal{O}^*(\Lambda)}(\Omega)$ is the KKS measure on $\mathcal{O}^*(\Lambda)$ and $\| \Lambda \|^2=(\Lambda, \Lambda)=-\operatorname{tr}(\Lambda^2)$ as in \emph{(\ref{Eq:suNinnerproduct})}, and $c_N>0$ is the normalizing constant defined in \emph{(\ref{Eq:OrthogonalBasisZeitlinInnerProduct})}. 
\end{lemma}

\begin{proof}
For any $X, Y\in \su(N)$, we claim that
\begin{equation}\label{Lm:OrbitAveragesuN:Eq:QuadraticAverageIntegral}
\int_{\mathcal{O}^*(\Lambda)} (\Omega, X)(\Omega, Y)dm_{\mathrm{KKS}}(\Omega) = \dfrac{\|\Lambda\|^2}{N^2-1}(X,Y) \ .
\end{equation}
Given (\ref{Lm:OrbitAveragesuN:Eq:QuadraticAverageIntegral}), by complex bilinearity we extend it to $\su(N; \mathbb{C})$ and let $X=L_{-k,N}$, $Y=L_{-l, N}$. Then $(\Omega, X)=\Omega_k (L_{k,N}, L_{-k,N})=c_N\Omega_k$ and similarly $(\Omega, Y)=c_N\Omega_l$. Since $(L_{-k, N}, L_{-l, N})=c_N\delta_{-k-l, 0}=c_N\delta_{k+l, 0}$, we arrive at (\ref{Lm:OrbitAveragesuN:Eq:Average}). 

It remains to prove (\ref{Lm:OrbitAveragesuN:Eq:QuadraticAverageIntegral}). Define the linear map $C_\Lambda:\ \su(N)\rightarrow \su(N)$ by $$C_\Lambda(X)=\displaystyle{\int_{\mathcal{O}^*(\Lambda)}(\Omega,X)\Omega dm_{\mathrm{KKS}}(\Omega)} \ .$$
Then for any $g\in SU(N)$ we have 
$$\begin{array}{lll}
& C_\Lambda(gXg^{-1}) &
\\
= & \displaystyle{\int_{\mathcal{O}^*(\Lambda)}} (\Omega, gXg^{-1})\Omega dm_{\mathrm{KKS}}(\Omega)&
\\
= & \displaystyle{\int_{\mathcal{O}^*(\Lambda)}} (g^{-1}\Omega g, X)\Omega dm_{\mathrm{KKS}}(\Omega) & 
\\
= & \displaystyle{\int_{\mathcal{O}^*(\Lambda)}} (\Omega', X)g\Omega'g^{-1} dm_{\mathrm{KKS}}(g\Omega'g^{-1}) & (\text{change of variable } \Omega'=g^{-1}\Omega g)
\\
= & \displaystyle{\int_{\mathcal{O}^*(\Lambda)}} g(\Omega', X)\Omega'g^{-1} dm_{\mathrm{KKS}}(\Omega') & (dm_{\mathrm{KKS}}(\Omega') \text{ is invariant under } \Omega' \rightarrow g^{-1}\Omega' g)
\\
= & g C_\Lambda(X) g^{-1} \ . &
\end{array}$$
So $C_\Lambda$ is an intertwining operator (see \cite[Section 4.8]{[RepresentationTheoryGTM222]}) of the adjoint representation, i.e. $C_\Lambda \circ Ad_g = Ad_g \circ C_\Lambda$ for any $g\in SU(N)$.  Using this fact, an argument similar to the proof of Schur's Lemma (see \cite[Theorem 4.26]{[RepresentationTheoryGTM222]}, \cite[Section 8.2, Theorem 1]{[KirillovRepresentationBook]}) helps to prove that \begin{equation}\label{Lm:OrbitAveragesuN:Eq:IntertwiningCIdentity}
C_\Lambda=\dfrac{\|\Lambda\|^2}{N^2-1}I \ ,
\end{equation}
which directly leads to (\ref{Lm:OrbitAveragesuN:Eq:QuadraticAverageIntegral}). 

To show (\ref{Lm:OrbitAveragesuN:Eq:IntertwiningCIdentity}), we first notice that $C_\Lambda$ is self-adjoint, as we have 
$$\begin{array}{lll}
(C_\Lambda X, Y) & =\displaystyle{\int_{\mathcal{O}^*(\Lambda)}}(\Omega, X)(\Omega, Y)dm_{\mathrm{KKS}}(\Omega)& 
\\
& =\displaystyle{\int_{\mathcal{O}^*(\Lambda)}}(\Omega, Y)(\Omega, X)dm_{\mathrm{KKS}}(\Omega)& =(X, C_\Lambda Y) \ .
\end{array}$$ Therefore $C_\Lambda$ is real-diagonalizable on $\su(N)$, i.e., $\su(N)$ can be decomposed into eigenspaces of $C_\Lambda$ with real eigenvalues. Let $E_\lambda \subseteq \su(N)$ be an eigenspace of $C_\Lambda$ with real eigenvalue $\lambda$. Then if $X\in E_\lambda$, we must have $C_\Lambda (gXg^{-1})=\lambda gXg^{-1}$, i.e., $g E_\lambda g^{-1}=E_\lambda$. If we let $g=g(t)=e^{tY}$ for some $Y\in \su(N)$, then $e^{tY}E_\lambda e^{-tY}=E_\lambda$ gives $[Y, X]\in E_\lambda$ for any $X\in E_\lambda$, which means that $E_\lambda$ is an $\su(N)$ ideal. Since $\su(N)$ is a simple Lie algebra, the only nontrivial ideal is $\su(N)$ itself. So this yields $E_\lambda = \{0\}$ or $E_\lambda = \su(N)$. This means that we must have $C_\Lambda = c_\Lambda I$ for some constant $c_\Lambda$.

The constant $c_\Lambda$ can be determined as follows. First we have $\operatorname{Tr}_{\su(N)}(C_\Lambda)=c_\Lambda\operatorname{Tr}_{\su(N)}(I)=c_\Lambda \text{dim}(\su(N))=(N^2-1)c_\Lambda$. On the other hand, 
$$\begin{array}{ll}
\operatorname{Tr}_{\su(N)}(C_\Lambda) & =\displaystyle{\int_{\mathcal{O}^*(\Lambda)}}\operatorname{Tr}_{\su(N)}(X\mapsto (\Omega, X)\Omega) dm_{\mathrm{KKS}}(\Omega)
\\
& =\displaystyle{\int_{\mathcal{O}^*(\Lambda)}}\|\Omega\|^2 dm_{\mathrm{KKS}}(\Omega)
\\
& =\|\Lambda\|^2 \ ,
\end{array}$$
since for any $\Omega\in \mathcal{O}^*(\Lambda)$ we have $\Omega=g\Lambda g^{-1}$ for some $g\in SU(N)$, yielding $\|\Omega\|^2=\|\Lambda\|^2$. Therefore $(N^2-1)c_\Lambda=\|\Lambda\|^2$, giving (\ref{Lm:OrbitAveragesuN:Eq:IntertwiningCIdentity}).
\end{proof}

Lemma \ref{Lm:OrbitAveragesuN} helps us to calculate the orbit average of cumulative low-frequency energy and enstrophy as the following 

\begin{corollary}\label{Cor:OrbitAverageCumulativeEnergyEnstrophy}
We have
\begin{equation}\label{Cor:OrbitAverageCumulativeEnergyEnstrophy:Eq}
\begin{array}{ll}
\displaystyle{\int_{\mathcal{O}^*(\Lambda)}}\mathcal{E}_{\leq K}(\Omega)dm_{\mathrm{KKS}}(\Omega) & = \dfrac{\|\Lambda\|^2}{2c_N(N^2-1)}\cdot \sum\limits_{0<|k|\leq K} \dfrac{1}{|k|^2} \ ,
\\
\displaystyle{\int_{\mathcal{O}^*(\Lambda)}}\mathcal{Z}_{\leq K}(\Omega)dm_{\mathrm{KKS}}(\Omega) & = \dfrac{\|\Lambda\|^2}{2c_N(N^2-1)}\cdot \left|\{k: 0<|k|\leq K\}\right| \ .
\end{array}
\end{equation}
Here the cumulative energy and enstrophy $\mathcal{E}_{\leq K}(\Omega)$, $\mathcal{Z}_{\leq K}(\Omega)$ are defined as in \emph{(\ref{Eq:FluxZeitlin:CumulativeLowFreqEnergyEnstropy})}.
\end{corollary}

Now we use Lemma \ref{Lm:OrbitAveragesuN} to analyze the cross-shell transfer of each of the other terms than the nonlinear term in (\ref{Eq:EulerArnoldZeitlinTruncatedDissipatedForcedWithFastNoise:Specific}). Let $Q: \su(N)\rightarrow \mathbb{R}$ be a quadratic observable defined by 
\begin{equation}\label{Eq:QuadraticObservableGeneralizingEnergyEnstrophy}
Q(\Omega)=\dfrac{1}{2}\sum\limits_{i,j\in \mathcal{K}_N} b_{ij}\Omega_i\Omega_j \ , \ b_{ij}=b_{ji} \ ,
\end{equation} 
where $\Omega=\sum\limits_{k\in \mathcal{K}_N}\Omega_k L_{k,N}$ as in (\ref{Eq:ZeitlinExpansionVorticityAndStreamMatrices}).

When 
\begin{equation}\label{Eq:QuadraticObservableGeneralizingEnergyEnstrophy:Coefficientb:Energy} b_{ij}=\dfrac{1}{|i|^2}\cdot \delta_{i+j, 0}\cdot \delta_{|i|\leq K}
\end{equation} 
we have $Q(\Omega)=\dfrac{1}{2}\sum\limits_{|k|\leq K, k\in \mathcal{K}_N}\dfrac{|\Omega_k|^2}{|k|^2}=\mathcal{E}_{\leq K}(\Omega)$; and when 
\begin{equation}\label{Eq:QuadraticObservableGeneralizingEnergyEnstrophy:Coefficientb:Enstrophy}  b_{ij}=\delta_{i+j, 0}\cdot \delta_{|i|\leq K} \ ,
\end{equation} 
we have $Q(\Omega)=\dfrac{1}{2}\sum\limits_{|k|\leq K, k\in \mathcal{K}_N} |\Omega_k|^2=\mathcal{Z}_{\leq K}(\Omega)$.

The evolution of the quantity $Q(\Omega_t^{\varepsilon,\delta})$ satisfies the Dynkin's formula (see \cite[Lemma 17.21, p383]{[KallenbergProbabilityBook]})

\begin{equation}\label{Eq:DynkinFormulaEvolutionQuadraticObserble}
\begin{array}{ll}
\mathbf{E}_{\Omega_0} Q(\Omega_{t\wedge \tau^{\varepsilon,\delta}}^{\varepsilon,\delta})-Q(\Omega_0)
& =\mathbf{E}_{\Omega_0^{\varepsilon,\delta} =\Omega_0} \displaystyle{\int_0^{t\wedge \tau^{\varepsilon,\delta}}} \left(\dfrac{1}{\varepsilon^2}L_{\text{fast}}+L_{\text{slow}}\right) Q(\Omega_s^{\varepsilon,\delta})ds \\
&
=\dfrac{1}{\varepsilon^2}\mathbf{E}_{\Omega_0} \displaystyle{\int_0^{t\wedge \tau^{\varepsilon,\delta}}} L_{\text{fast}} Q(\Omega_s^{\varepsilon,\delta})ds + \mathbf{E}_{\Omega_0} \displaystyle{\int_0^{t\wedge \tau^{\varepsilon,\delta}}} L_{\text{slow}} Q(\Omega_s^{\varepsilon,\delta})ds \ . 
\end{array}
\end{equation}
where $L=\dfrac{1}{\varepsilon^2}L_{\text{fast}}+L_{\text{slow}}$ is the generator of $\Omega_t^{\varepsilon,\delta}$ as defined in (\ref{Eq:EulerArnoldZeitlinTruncatedDissipatedForcedWithFastNoise:GeneratorFastSlow}) with $L_{\text{fast}}$ defined in (\ref{Eq:EulerArnoldZeitlinTruncatedDissipatedForcedWithFastNoise:FastMotionPartGenerator}) and $L_{\text{slow}}$ defined in (\ref{Eq:EulerArnoldZeitlinTruncatedDissipatedForcedWithFastNoise:EssentialSlowMotionPartGenerator}). For $Q=\mathcal{E}_{\leq K}$ or $Q=\mathcal{Z}_{\leq K}$, the terms on the RHS of (\ref{Eq:DynkinFormulaEvolutionQuadraticObserble}) give the change rates of these quantities that can be used to define the corresponding fluxes. However, two issues must be noticed: 

\begin{itemize}
\item[(1)] The term $$\begin{array}{ll}
-\mathcal{J}^{\varepsilon,\delta}_{\text{fast}}(Q, t)& \stackrel{\text{def}}{=}\displaystyle{\dfrac{1}{\varepsilon^2}\mathbf{E}_{\Omega_0}\int_0^{t\wedge \tau^{\varepsilon,\delta}}L_{\text{fast}}Q(\Omega_s^{\varepsilon,\delta})ds}
\\
& = \displaystyle{\mathbf{E}_{\Omega_0} Q(\Omega^{\varepsilon,\delta}_{t\wedge \tau^{\varepsilon,\delta}})-Q(\Omega_0)-\mathbf{E}_{\Omega_0}\int_0^{t\wedge \tau^{\varepsilon,\delta}}L_{\text{slow}}Q(\Omega_s^{\varepsilon,\delta})ds} \ ,
\end{array}$$
where the second equality is by (\ref{Eq:DynkinFormulaEvolutionQuadraticObserble}). 
This term corresponds to the change rate of $Q$ made by the fast process, and it need not be considered as the physical flux in the sense of Kraichnan's theory. The reason is that $L_{\mathrm{fast}}$ is introduced as an auxiliary fast mixing mechanism on the coadjoint leaves, rather than as a genuine shell-transfer mechanism of the Euler-Arnold dynamics. Although $L_{\mathrm{fast}}Q$ may change a shell observable $Q= \mathcal E_{\le K}$ or $Q=\mathcal Z_{\le K}$, this change represents the rapid thermalization of the leaf variable under the artificial fast process. Since $Q$ is not a Casimir invariant, this thermalization may redistribute the value of $Q$ within the same coadjoint orbit even though the Casimir parameter $z=\pi(\Omega)$ is kept fixed. The role of the fast process is to determine the leafwise statistical ensemble with respect to which this nonlinear flux observable is averaged, not to provide an additional physical flux channel. 

\item[(2)] The term 
\begin{equation}\label{Eq:FluxSlowMotion}
-\mathcal{J}_{\text{slow}}^{\varepsilon,\delta}(Q, t)=\mathbf{E}_{\Omega_0}\int_0^{t\wedge \tau^{\varepsilon,\delta}}L_{\text{slow}}Q(\Omega_s^{\varepsilon,\delta})ds
\end{equation}
gives the change rates of $Q$ made by each of the terms in $L_\text{slow}$ defined in (\ref{Eq:EulerArnoldZeitlinTruncatedDissipatedForcedWithFastNoise:EssentialSlowMotionPartGenerator}). For $Q=\mathcal E_{\le K}$ or $Q=\mathcal Z_{\le K}$, this term should be understood as a shell-budget contribution. It is not yet a single flux quantity. Depending on the particular component of $L_{\mathrm{slow}}$, the corresponding rate may represent a genuine cross-shell transfer across the cutoff $K$, or it may represent a within-shell effect, such as injection, dissipation, damping, stochastic heating, or redistribution that changes the amount of $Q$ contained in the low modes without being a transfer between the two sides of the shell boundary. As $\varepsilon\rightarrow 0$, the interpretation of the averaged rates as cross-shell or within-shell will be made only after decomposing $L_{\mathrm{slow}}$ into its individual dynamical components.
\end{itemize}

By (\ref{Eq:EulerArnoldZeitlinTruncatedDissipatedForcedWithFastNoise:EssentialSlowMotionPartGenerator}), we decompose
\begin{equation}\label{Eq:SlowMotionGeneratorDecoposition}
L_{\mathrm{slow}}
=
L_Y+L_{EA}+L_\alpha+L_\nu+L_\delta+L_f \ ,
\end{equation}
where 
\begin{equation}\label{Eq:SlowMotionGeneratorDecoposition:Termwise}
\begin{array}{lll}
L_Y Q(\Omega) & = & \dfrac{\rho^2}{2}\sum\limits_{k=1}^l\left[[D^2 Q(\Omega)](Y_k(\Omega), Y_k(\Omega))+(\nabla Q(\Omega), [DY_k(\Omega)](Y_k(\Omega)))\right] \ , 
\\
L_{EA} Q(\Omega) & = & [DQ(\Omega)]([\Psi, \Omega]) \ ,
\\
L_\alpha Q(\Omega) & = & [DQ(\Omega)](-\alpha\Omega) \ ,
\\
L_\nu Q(\Omega) & = & [DQ(\Omega)](\nu(-A)\Omega) \ ,
\\
L_\delta Q(\Omega) & = & \dfrac{\delta^2}{2}\sum\limits_{k\in \mathcal{K}_N^+} \left([D^2 Q(\Omega)](L_{k, \cos N}, L_{k, \cos N})+[D^2 Q(\Omega)](L_{k, \sin, N}, L_{k, \sin, N})\right) \ ,
\\
L_f Q(\Omega) & = & \dfrac{1}{2}\sum\limits_{k\in \mathcal{K}_N^+, |k|\approx k_f}q_k^2\left([D^2 Q(\Omega)](L_{k, \cos N}, L_{k, \cos N})+[D^2 Q(\Omega)](L_{k, \sin, N}, L_{k, \sin, N})\right) \ ,
\end{array}
\end{equation}
i.e., these terms correspond respectively to the \(Y\)-noise, the Euler-Arnold
nonlinearity, the damping, the viscosity, the background non-degenerate noise,
and the external stochastic forcing. For each component \(L_r\), \(r\in
\{Y,EA,\alpha,\nu,\delta,f\}\), we define
\begin{equation}\label{Eq:FluxComponent}
\mathcal{J}_r^{\varepsilon,\delta}(Q,t)
\stackrel{\text{def}}{=}
\mathbf{E}_{\Omega_0}\int_0^{t\wedge \tau^{\varepsilon,\delta}}j_Q^r(\Omega_s^{\varepsilon,\delta})ds 
\ , \text{ where } j_Q^r(\Omega)=-L_r Q(\Omega) \ .
\end{equation}
Then using (\ref{Eq:SlowMotionGeneratorDecoposition}), (\ref{Eq:FluxSlowMotion}), (\ref{Eq:FluxComponent}), we can decompose

\begin{equation}\label{Eq:FluxSlowMotionDecompositionIntoIndividualFluxes}
\begin{array}{ll}
\mathcal{J}_{\text{slow}}^{\varepsilon,\delta}(Q,t) & = \sum\limits_{r\in\{Y,EA,\alpha,\nu,\delta,f\}}
\mathcal J_r^{\varepsilon,\delta}(Q,t) \ .
\end{array}
\end{equation}

\begin{lemma}\label{Lm:QuadraticObserbleSlowGeneratorComponentLimit}
We have, for any $t> 0$ and each of $r\in \{Y, EA, \alpha, \nu, \delta, f\}$,
\emph{\begin{equation}\label{Lm:QuadraticObserbleSlowGeneratorComponentLimit:Eq}
\lim\limits_{\delta\rightarrow 0}\lim\limits_{\varepsilon\rightarrow 0} \mathcal{J}_r^{\varepsilon,\delta}(Q,t)=\mathbf{E}_{Z_0=\pi(\Omega_0)}\int_0^{t\wedge \tau}\overline{j_Q^r}(Z_s)ds \ ,
\end{equation}}
where the averaging operation is defined in \emph{(\ref{Eq:AveragingLeafwiseCoAdjointsuN})}.
\end{lemma}

\begin{proof}
This is a direct consequence of Corollary \ref{Cor:BoundedFunctionConvergenceProcessToZ} provided that the process is stopped before leaving the chosen regular stratum neighborhood.
\end{proof}

Note that $j_{\mathcal{E}_{\leq K}}^{EA}(\Omega)=\Pi_\mathcal{E}(K;\Omega)$ and $j_{\mathcal{Z}_{\leq K}}^{EA}(\Omega)=\Pi_\mathcal{Z}(K;\Omega)$ correspond to the nonlinear cross-shell
fluxes that we have discussed in Section \ref{Sec:AveragingFramework:EnergyEnstrophyFlux}.
We have already analyzed the above average for $\overline{j^{EA}_{\mathcal{E}_{\leq K}}}$ and $\overline{j^{EA}_{\mathcal{Z}_{\leq K}}}$ in Corollary \ref{Cor:AveragedNonlinearFluxConvergenceZero} and concluded that they are $0$ with respect to the KKS invariant measure $m_{\mathrm{KKS}}$. The rest of this section is to analyze the average of all the other $j^r_{\mathcal{E}_{\leq K}}$ and $j^r_{\mathcal{Z}_{\leq K}}$'s for $r\in \{Y, \alpha, \nu, \delta, f\}$ and investigate their role played in the shell-budget change, i.e., cross-shell transfer or within-shell effect.

\begin{lemma}\label{Lm:AverageSlowGeneratorOtherTermsjForQuadraticObservableQ}
For the quadratic observable $Q$ introduced in \emph{(\ref{Eq:QuadraticObservableGeneralizingEnergyEnstrophy})}, any $z$ in the interior of $\mathcal{A}_{\mathrm{regular}}$, and the coadjoint orbit invariant measure $dm_z(\Omega)=dm_{\text{KKS},z}(\Omega)$, assume $\mathcal{O}^*(z)=\mathcal{O}^*(\Lambda)$ for $\Lambda=\text{diag}(i\lambda_1,...,i\lambda_N)\in \su(N)$, then we have
\begin{equation}\label{Lm:AverageSlowGeneratorOtherTermsjForQuadraticObservableQ:Eq}
\begin{array}{ll}
\overline{j_Q^\alpha}(z) & = \dfrac{\alpha \|\Lambda\|^2}{c_N(N^2-1)}\sum\limits_{k\in \mathcal{K}_N} b_{k, -k} \ ;
\\
\overline{j_Q^\nu}(z) & = \dfrac{\nu \|\Lambda\|^2}{c_N(N^2-1)}\sum\limits_{k\in \mathcal{K}_N}|k|^2 b_{k, -k} \ ;
\\
\overline{j_Q^\delta}(z) & =-\dfrac{\delta^2}{2}\sum\limits_{k\in \mathcal{K}_N^+} b_{k, -k} \ ;
\\
\overline{j_Q^f}(z) & = -\dfrac{1}{2}\sum\limits_{k\in \mathcal{K}_N^+, |k|\approx k_f}q_k^2 b_{k, -k} \ .
\end{array}
\end{equation}
\end{lemma}

\begin{proof}
Since $Q$ defined in (\ref{Eq:QuadraticObservableGeneralizingEnergyEnstrophy}) is quadratic, we have $j_Q^\alpha(\Omega)=-[DQ(\Omega)](-\alpha\Omega))=2\alpha Q(\Omega)$. Therefore by Lemma \ref{Lm:OrbitAveragesuN} we get 
$$\begin{array}{ll}
\overline{j_Q^\alpha}(z)& =\displaystyle{\int_{\mathcal{O}^*(z)}}2\alpha Q(\Omega) dm_{\mathrm{KKS}, z}(\Omega)
\\
& = \alpha \sum\limits_{i,j\in \mathcal{K}_N}b_{ij}\displaystyle{\int_{\mathcal{O}^*(z)}}\Omega_i\Omega_j dm_{\mathrm{KKS}, z}(\Omega)
\\
& = \alpha \sum\limits_{i,j\in \mathcal{K}_N}b_{ij} \dfrac{\|\Lambda\|^2}{c_N(N^2-1)}\delta_{i+j, 0}
\\
& = \dfrac{\alpha \|\Lambda\|^2}{c_N(N^2-1)}\sum\limits_{k\in \mathcal{K}_N} b_{k, -k} \ .
\end{array}$$

Next, we have $j_Q^\nu(\Omega)=-[DQ(\Omega)](\nu(-A)\Omega)$, where $A$ defined in (\ref{Eq:InertiaOperatorZeitlinCase}) satisfies $AL_{k,N}=|k|^2 L_{k,N}$. From (\ref{Eq:QuadraticObservableGeneralizingEnergyEnstrophy}) we see that $[DQ(\Omega)](\Theta)=\sum\limits_{i,j\in \mathcal{K}_N}\Omega_i b_{ij} \Theta_j$ where $\Theta=\sum\limits_{k\in \mathcal{K}_N}\Theta_k L_{k,N}\in \su(N)$. Setting $\Theta=A\Omega=\sum\limits_{k\in \mathcal{K}_N} |k|^2\Omega_k L_{k,N}$, we get $[DQ(\Omega)](A\Omega)=\sum\limits_{i,j\in \mathcal{K}_N} |j|^2\Omega_ib_{ij}\Omega_j$, so $j_Q^\nu(\Omega)=\nu\sum\limits_{i,j\in \mathcal{K}_N}|j|^2\Omega_ib_{ij}\Omega_j$. Therefore by Lemma \ref{Lm:OrbitAveragesuN} we get
$$\begin{array}{ll}
\overline{j_Q^\nu}(z) & = \nu\sum\limits_{i,j\in \mathcal{K}_N}b_{ij}|j|^2\displaystyle{\int_{\mathcal{O}^*(z)}}\Omega_i\Omega_j dm_{\mathrm{KKS}, z}(\Omega)
\\
& = \nu \sum\limits_{i,j\in \mathcal{K}_N}b_{ij}|j|^2 \dfrac{\|\Lambda\|^2}{c_N(N^2-1)}\delta_{i+j, 0}
\\
& = \dfrac{\nu \|\Lambda\|^2}{c_N(N^2-1)}\sum\limits_{k\in \mathcal{K}_N}|k|^2 b_{k, -k} \ .
\end{array}$$

Next, we have $$j_Q^\delta(\Omega)=-\dfrac{\delta^2}{2}\sum\limits_{k\in \mathcal{K}_N^+} \left([D^2 Q(\Omega)](L_{k, \cos N}, L_{k, \cos N})+[D^2 Q(\Omega)](L_{k, \sin, N}, L_{k, \sin, N})\right) \ .$$
Since $D^2 Q(\Omega)$ is bilinear and $[D^2 Q(\Omega)](L_{i,N}, L_{j,N})=b_{ij}$, using (\ref{Eq:RealSinCosZeitlinBasis}), we get
$$\begin{array}{ll}
& [D^2Q(\Omega)](L_{k, \cos, N}, L_{k, \cos, N})
\\
= & [D^2 Q(\Omega)]\left(\dfrac{L_{k, N}+L_{-k,N}}{2}, \dfrac{L_{k,N}+L_{-k,N}}{2}\right)
\\
= & \dfrac{1}{4}\left([D^2 Q(\Omega)](L_{k,N}, L_{k,N})+[D^2 Q(\Omega)](L_{-k,N}, L_{k,N}) \right.
\\
& \qquad \left. +[D^2 Q(\Omega)](L_{k,N}, L_{-k,N})+ [D^2 Q(\Omega)](L_{-k,N}, L_{-k,N})\right)
\\
= & \dfrac{1}{4}(b_{kk}+b_{-k,k}+b_{k,-k}+b_{-k,-k}) \ .
\end{array}$$
Similarly we have $[D^2Q(\Omega)](L_{k,\sin,N}, L_{k,\sin,N})=\dfrac{1}{4}(-b_{kk}+b_{-k,k}+b_{k,-k}-b_{-k,-k})$. Thus we get $j_Q^\delta(\Omega)=-\dfrac{\delta^2}{2}\sum\limits_{k\in \mathcal{K}_N^+} b_{k, -k}$, which is a constant. Therefore $\overline{j_Q^\delta}(\Omega)=-\dfrac{\delta^2}{2}\sum\limits_{k\in \mathcal{K}_N^+} b_{k, -k}$.

Using the same computation as the above we also get $\overline{j_Q^f}(\Omega)=-\dfrac{1}{2}\sum\limits_{k\in \mathcal{K}_N^+, |k|\approx k_f}q_k^2 b_{k, -k}$.
\end{proof}

Using (\ref{Eq:QuadraticObservableGeneralizingEnergyEnstrophy:Coefficientb:Energy}) and (\ref{Eq:QuadraticObservableGeneralizingEnergyEnstrophy:Coefficientb:Enstrophy}), Lemma \ref{Lm:AverageSlowGeneratorOtherTermsjForQuadraticObservableQ} enables us to analyze the shell-budget contributions for the energy and enstrophy by the damping, viscosity, background noise and the external forcing terms. 

\begin{corollary}\label{Cor:AverageSlowGeneratorOtherTermsjForDamping}
Under the same assumptions of \emph{Lemma \ref{Lm:AverageSlowGeneratorOtherTermsjForQuadraticObservableQ}}, we have
\begin{equation}\label{Cor:AverageSlowGeneratorOtherTermsjForDamping:Eq}
\overline{j_{\mathcal{E}_{\leq K}}^\alpha}(z)=2\alpha \overline{\mathcal{E}_{\leq K}}(z) \ , \ \overline{j_{\mathcal{Z}_{\leq K}}^\alpha}(z)=2\alpha \overline{\mathcal{Z}_{\leq K}}(z) \ .
\end{equation}
Thus the damping term $-\alpha \Omega$ in \emph{(\ref{Eq:EulerArnoldZeitlinTruncatedDissipatedForcedWithFastNoise:Specific})} only contributes within-shell budget change.
\end{corollary}

\begin{proof}
Setting $b$ as in (\ref{Eq:QuadraticObservableGeneralizingEnergyEnstrophy:Coefficientb:Energy}), using Lemma \ref{Lm:AverageSlowGeneratorOtherTermsjForQuadraticObservableQ}, we get $\overline{j_{\mathcal{E}_{\leq K}}^\alpha}(z)=\dfrac{\alpha\|\Lambda\|^2}{c_N(N^2-1)}\sum\limits_{0<|k|\leq K}\dfrac{1}{|k|^2}$,
which by Corollary \ref{Cor:OrbitAverageCumulativeEnergyEnstrophy} gives $\overline{j_{\mathcal{E}_{\leq K}}^\alpha}(z)=2\alpha \overline{\mathcal{E}_{\leq K}}(z)$. Similarly, setting $b$ as in (\ref{Eq:QuadraticObservableGeneralizingEnergyEnstrophy:Coefficientb:Enstrophy}), using Lemma \ref{Lm:AverageSlowGeneratorOtherTermsjForQuadraticObservableQ}, we get $\overline{j_{\mathcal{Z}_{\leq K}}^\alpha}(z)=\dfrac{\alpha \|\Lambda\|^2}{c_N(N^2-1)}|\{k: 0< |k|\leq K\}|$, which by Corollary \ref{Cor:OrbitAverageCumulativeEnergyEnstrophy} gives $\overline{j_{\mathcal{Z}_{\leq K}}^\alpha}(z)=2\alpha \overline{\mathcal{Z}_{\leq K}}(z)$. 
\end{proof}

\begin{corollary}\label{Cor:AverageSlowGeneratorOtherTermsjForViscosity}
Under the same assumptions of \emph{Lemma \ref{Lm:AverageSlowGeneratorOtherTermsjForQuadraticObservableQ}}, we have
\begin{equation}\label{Cor:AverageSlowGeneratorOtherTermsjForViscosity:Eq}
\overline{j_{\mathcal{E}_{\leq K}}^\nu}(z)=2\nu \overline{\mathcal{Z}_{\leq K}}(z) \ , \ \overline{j_{\mathcal{Z}_{\leq K}}^\nu}(z)=\dfrac{\nu \|\Lambda\|^2}{c_N(N^2-1)}\sum\limits_{k\in \mathcal{K}_N, |k|\leq K}|k|^2 \ .
\end{equation}
Thus the viscosity term $\nu(-A)\Omega$ in \emph{(\ref{Eq:EulerArnoldZeitlinTruncatedDissipatedForcedWithFastNoise:Specific})} only contributes within-shell budget change.
\end{corollary}

\begin{proof}
Setting $b$ as in (\ref{Eq:QuadraticObservableGeneralizingEnergyEnstrophy:Coefficientb:Energy}), using Lemma \ref{Lm:AverageSlowGeneratorOtherTermsjForQuadraticObservableQ}, we get $\overline{j_{\mathcal{E}_{\leq K}}^\nu}(z)=\dfrac{\nu\|\Lambda\|^2}{c_N(N^2-1)}|\{k: 0<|k|\leq K\}|$,
which by Corollary \ref{Cor:OrbitAverageCumulativeEnergyEnstrophy} gives $\overline{j_{\mathcal{E}_{\leq K}}^\nu}(z)=2\nu \overline{\mathcal{Z}_{\leq K}}(z)$. 

Setting $b$ as in (\ref{Eq:QuadraticObservableGeneralizingEnergyEnstrophy:Coefficientb:Enstrophy}), using Lemma \ref{Lm:AverageSlowGeneratorOtherTermsjForQuadraticObservableQ}, we get $\overline{j_{\mathcal{Z}_{\leq K}}^\nu}(z)=\dfrac{\nu \|\Lambda\|^2}{c_N(N^2-1)}\sum\limits_{k\in \mathcal{K}_N, |k|\leq K}|k|^2$. 
\end{proof}

\begin{corollary}\label{Cor:AverageSlowGeneratorOtherTermsjForBackgroundNoiseExternalForcing}
Under the same assumptions of \emph{Lemma \ref{Lm:AverageSlowGeneratorOtherTermsjForQuadraticObservableQ}}, we have
\begin{equation}\label{Cor:AverageSlowGeneratorOtherTermsjForBackgroundNoiseExternalForcing:Eq}
\begin{array}{l}
\overline{j_{\mathcal{E}_{\leq K}}^\delta}(z)= -\dfrac{\delta^2}{2}\sum\limits_{k\in \mathcal{K}_N^+, |k|\leq K} \dfrac{1}{|k|^2}\ , \ \overline{j_{\mathcal{Z}_{\leq K}}^\delta}(z)= -\dfrac{\delta^2}{2}|\{k\in \mathcal{K}_N^+: 0<|k|\leq K\}| \ .
\\
\overline{j_{\mathcal{E}_{\leq K}}^f}(z)= -\dfrac{1}{2} \sum\limits_{k\in \mathcal{K}_N^+, |k|\approx k_f, |k|\leq K} q_k^2\dfrac{1}{|k|^2}\ , \ \overline{j_{\mathcal{Z}_{\leq K}}^f}(z)=-\dfrac{1}{2} \sum\limits_{k\in \mathcal{K}_N^+, |k|\approx k_f, |k|\leq K} q_k^2 \ .
\end{array}
\end{equation}
Thus the background noise $\delta d\mathfrak{b}_t$ and the external forcing term $f(t)dt$ in \emph{(\ref{Eq:EulerArnoldZeitlinTruncatedDissipatedForcedWithFastNoise:Specific})} only contribute within-shell budget change.
\end{corollary}

It remains to analyze the shell-budget contribution of the $Y$ term in (\ref{Eq:EulerArnoldZeitlinTruncatedDissipatedForcedWithFastNoise:Specific}). By (\ref{Eq:SlowMotionGeneratorDecoposition:Termwise}) and (\ref{Eq:FluxComponent}) we know that
$$\overline{j_Q^Y}(z)
= -\dfrac{\rho^2}{2}\sum\limits_{k=1}^l \int_{\mathcal{O}^*(z)} \left[[D^2 Q](Y_k, Y_k)+(\nabla Q, [DY_k](Y_k))\right](\Omega) dm_{\mathrm{KKS},z}(d\Omega) \ ,
$$
which is generally non-zero.
For general vector fields $Y_k$ and $Q=\mathcal{E}_{\leq K}$ or $Q=\mathcal{Z}_{\leq K}$, usually there is no closed-form formula for the above quantity. However, by (\ref{Assumption:FastVectorEnergyPreservingPerturbation:Eq:PerturbationParameter}) and (\ref{Assumption:FastVectorEnergyPreservingPerturbation:Eq:PerturbationParameter:ConstraintK}), on the regular compact region considered above, the order of $\overline{j_Q^Y}(z)$ is at $O(\alpha+\nu)$ and thus by Corollary \ref{Cor:AverageSlowGeneratorOtherTermsjForDamping}, \ref{Cor:AverageSlowGeneratorOtherTermsjForViscosity} it is comparable with the shell-budget contributions made by the damping and viscosity terms. In summary, $\overline{j^Y_{\mathcal{E}_{\leq K}}}(z)$ and $\overline{j^Y_{\mathcal{Z}_{\leq K}}}(z)$ should be interpreted as a small transverse stochastic
shell-budget effect, or leakage effect, rather than as a nonlinear cross-shell flux
carried by the Euler-Arnold dynamics.

\section{Symmetry breaking and emergence of non-trivial flux}\label{Sec:SymmetryBreakingCoAdjoint}

In the previous section we considered the fully symmetric situation in which the fast leafwise dynamics preserves the KKS measure on each coadjoint orbit. In that case the averaged contribution of the Euler-Arnold nonlinearity to every shell energy and shell enstrophy flux vanishes. This shows that complete coadjoint thermalization, although it produces rapid mixing along the orbit, does not by itself create a non-equilibrium transfer of energy or enstrophy across shells.

We now explain how non-trivial flux can emerge once this KKS symmetry is broken. The basic mechanism is simple. The Euler-Arnold nonlinear term may produce instantaneous transfer between shells, but under the KKS orbital ensemble the positive and negative contributions cancel exactly. If the invariant measure of the fast tangent dynamics is tilted away from the KKS measure, this cancellation is no longer exact. The resulting statistical bias can select a preferred direction of nonlinear shell transfer. The tilted invariant measure is made precise in the following Assumption, which is made throughout this section.

\begin{assumption}[Symmetry breaking invariant measure on coadjoint orbits]\label{Assumption:FastVectorSymetryBreakingInvariantMeasureFullCoadjoint}
Under \emph{Assumption \ref{Assumption:FastVectorMixingFullCoadjoint}}, we further assume that the invariant measure of the fast process $\Omega_t^X$ is given by tilting the KKS measure, i.e.,
\begin{equation}\label{Assumption:FastVectorSymetryBreakingInvariantMeasureFullCoadjoint:Eq:Measure}
dm_z(\Omega)\stackrel{\mathrm{def}}{=}dm_{z, \mathrm{KKS}}^\kappa(\Omega)=(1+\kappa q(\Omega)+O(\kappa^2))\cdot dm_{z, \mathrm{KKS}}(\Omega) \ ,
\end{equation}
where $0<\kappa<\!\!<1$ is a small tilting parameter and $q(\Omega)\in \mathbf{C}^{(2)}_{\mathcal{O}^*(z)}(\mathbb{R})$ is not identically equal to $0$ such that \begin{equation}\label{Assumption:FastVectorSymetryBreakingInvariantMeasureFullCoadjoint:Eq:qCentering}
\int_{\mathcal{O}^*(z)}q(\Omega) dm_{z,\mathrm{KKS}}(\Omega)=0 \ .
\end{equation}
Moreover, the remainder term $O(\kappa^2)$ is uniform on $\mathcal{O}^*(z)$.
\end{assumption}

\begin{lemma}\label{Lm:SymetryBreakingAverageFlux}
Under \emph{Assumption \ref{Assumption:FastVectorSymetryBreakingInvariantMeasureFullCoadjoint}}, for any $z$ in the interior of $\mathcal{A}_{\text{regular}}$, as $dm_z(\Omega)=dm^\kappa_{z,\mathrm{KKS}}(\Omega)$, we have
\begin{equation}\label{Lm:SymetryBreakingAverageFlux:Eq}
\begin{array}{l}
\displaystyle{\overline{\Pi}_\mathcal{E}(K;z)=\kappa \int_{\mathcal{O}^*(z)}\Pi_\mathcal{E}(K; \Omega)q(\Omega)dm_{z, \mathrm{KKS}}(\Omega)}+O(\kappa^2) \ , \ 
\\
\displaystyle{\overline{\Pi}_\mathcal{Z}(K;z)=\kappa \int_{\mathcal{O}^*(z)}\Pi_\mathcal{Z}(K; \Omega)q(\Omega)dm_{z, \mathrm{KKS}}(\Omega)}+O(\kappa^2) \ .
\end{array}
\end{equation}
In particular, if the integrals in \emph{(\ref{Lm:SymetryBreakingAverageFlux:Eq})} are nonzero, then the leafwise averaged fluxes $\overline{\Pi}_{\mathcal{E}}(K; z)$, $\overline{\Pi}_\mathcal{Z}(K;z)$ are nonzero for small enough $\kappa>0$. 
\end{lemma}

\begin{proof}
We simply combine (\ref{Eq:AveragingLeafwiseCoAdjointsuN}), (\ref{Assumption:FastVectorSymetryBreakingInvariantMeasureFullCoadjoint:Eq:Measure}) with Corollary \ref{Cor:LeafAverageVanishKKSMeasure:EnergyEnstrophy} to get
$$\begin{array}{ll}
\overline{\Pi}_{\mathcal{E}}(K;z) & = \displaystyle{\int_{\mathcal{O}^*(z)} \Pi_{\mathcal{E}}(K; \Omega)dm^\kappa_{z, \mathrm{KKS}}}(\Omega)
\\
& = \displaystyle{\int_{\mathcal{O}^*(z)} \Pi_{\mathcal{E}}(K; \Omega)(1+\kappa q(\Omega)+O(\kappa^2))dm_{z, \mathrm{KKS}}}(\Omega)
\\
& = \displaystyle{\int_{\mathcal{O}^*(z)} \Pi_{\mathcal{E}}(K; \Omega)dm_{z, \mathrm{KKS}}}(\Omega)+\kappa \displaystyle{\int_{\mathcal{O}^*(z)} \Pi_{\mathcal{E}}(K; \Omega) q(\Omega)dm_{z, \mathrm{KKS}}}(\Omega) + O(\kappa^2)
\\
& = \kappa \displaystyle{\int_{\mathcal{O}^*(z)} \Pi_{\mathcal{E}}(K; \Omega) q(\Omega)dm_{z, \mathrm{KKS}}}(\Omega) + O(\kappa^2) \ .
\end{array}$$
The case of $\overline{\Pi}_{\mathcal{Z}}(K; z)$ follows a similar argument.
\end{proof}

In the rest of this section we discuss, from the construction of appropriate fast vector fields $X_j(\Omega)$ in (\ref{Eq:EulerArnoldZeitlinTruncatedDissipatedForcedWithFastNoise:FastMotionPart}), when can Assumption \ref{Assumption:FastVectorSymetryBreakingInvariantMeasureFullCoadjoint} be satisfied, and when the integrals in (\ref{Lm:SymetryBreakingAverageFlux:Eq}) are nonzero.

Suppose we have Hamiltonian vector fields $X_j^0(\Omega)=[\Omega, \nabla h_j(\Omega)]$, $j=1,...,m$ on the coadjoint orbit $\mathcal{O}^*(z)$, $z\in \mathcal{A}_{\text{regular}}$, that satisfies Assumption \ref{Assumption:FastVectorHamiltonianFullCoadjoint} (and thus Assumption \ref{Assumption:FastVectorMixingFullCoadjoint}). Let $\boldsymbol{v}_j=\boldsymbol{v}_j(\Omega)$, $j=1,...,m$ be non-Hamiltonian smooth vector fields on $\mathcal{O}^*(z)$ and consider the combined vector fields
\begin{equation}\label{Eq:SymmetryBreakingFastVectorFieldsCoadjoint}
X_j(\Omega)=X_j^0(\Omega)+\kappa \boldsymbol{v}_j(\Omega) \ , \ 0<\kappa <\!\!<1 \ , \ j=1,...,m \ .
\end{equation}
Inserting (\ref{Eq:SymmetryBreakingFastVectorFieldsCoadjoint}) into (\ref{Eq:EulerArnoldZeitlinTruncatedDissipatedForcedWithFastNoise:FastMotionPartGenerator}), the corresponding fast generator takes the form
\begin{equation}\label{Eq:SymmetryBreakingFastMotionGenerator}
L_{\text{fast}} V(\Omega)=L_{\text{fast}}^0 V(\Omega) + \kappa \mathcal{R} V(\Omega) + \kappa^2 \mathcal{S} V(\Omega) \ ,
\end{equation}
where
\begin{equation}\label{Eq:SymmetryBreakingFastMotionGenerator:L0}
L_{\text{fast}}^0 V(\Omega)= \dfrac{1}{2}\sum\limits_{j=1}^m \left[[D^2 V(\Omega)](X_j^0(\Omega), X_j^0(\Omega))+(\nabla V(\Omega), [DX_j^0(\Omega)](X_j^0(\Omega)))\right] \ ,
\end{equation}

\begin{equation}\label{Eq:SymmetryBreakingFastMotionGenerator:R}
\begin{array}{ll}
\mathcal{R} V(\Omega) & = \sum\limits_{j=1}^m [D^2 V(\Omega)](X_j^0(\Omega), \boldsymbol{v}_j(\Omega))
\\
& \qquad + \dfrac{1}{2}\sum\limits_{j=1}^m\left(\nabla V(\Omega), [DX_j^0(\Omega)](\boldsymbol{v}_j(\Omega))+[D\boldsymbol{v}_j(\Omega)](X_j^0(\Omega))\right) \ ,
\end{array}
\end{equation}

\begin{equation}\label{Eq:SymmetryBreakingFastMotionGenerator:S}
\mathcal{S} V(\Omega)= \dfrac{1}{2}\sum\limits_{j=1}^m \left[[D^2 V(\Omega)](\boldsymbol{v}_j(\Omega), \boldsymbol{v}_j(\Omega))+(\nabla V(\Omega), [D\boldsymbol{v}_j(\Omega)](\boldsymbol{v}_j(\Omega)))\right] \ .
\end{equation}

We consider the invariant measure on $\mathcal{O}^*(z)$ calculated from the fast generator (\ref{Eq:SymmetryBreakingFastMotionGenerator}) and we write it in terms of the density function $r^\kappa(\Omega)$ with respect to $m_{\mathrm{KKS}}$. This gives us
$$L_{\text{fast}}^* r^\kappa(\Omega) = 0 \ ,$$
where for $f, g\in \mathbf{C}^{(2)}_{\mathcal{O}^*(z)}(\mathbb{R})$ we have
$$\int_{\mathcal{O}^*(z)} (L_\text{fast}f(\Omega))g(\Omega)dm_{\mathrm{KKS}, z}(\Omega)=\int_{\mathcal{O}^*(z)} f(\Omega)(L^*_\text{fast}g(\Omega))dm_{\mathrm{KKS}, z}(\Omega) \ .$$
Write 
$r^\kappa(\Omega)=1+\kappa q(\Omega)+O(\kappa^2)$ such that $q(\Omega)$ satisfies (\ref{Assumption:FastVectorSymetryBreakingInvariantMeasureFullCoadjoint:Eq:qCentering}),
and make use of (\ref{Eq:SymmetryBreakingFastMotionGenerator}) we get from the above that 
\begin{equation}\label{Eq:SymmetryBreakingFastMotionInvariantMeasureEquationForq}
(L_\text{fast}^0)^* q = -\mathcal{R}^* 1 \ ,
\end{equation}
where $1$ is the constant one function. Equation (\ref{Eq:SymmetryBreakingFastMotionInvariantMeasureEquationForq}) is the equation for $q(\Omega)$, which, under the hypoellipticity assumption of $L_\text{fast}^0$ admits a unique solution $q(\Omega)\in \mathbf{C}^{(2)}_{\mathcal{O}^*(z)}(\mathbb{R})$ that satisfy the constraint (\ref{Assumption:FastVectorSymetryBreakingInvariantMeasureFullCoadjoint:Eq:qCentering}), given by the formula
\begin{equation}\label{Eq:SymmetryBreakingFastMotionInvariantMeasureSolutionFormulaForq}
q = -\left[(L_\text{fast}^0)^*\right]^{-1}\left(\mathcal{R}^* 1\right) \ .
\end{equation}
Here the inverse operators are understood on the corresponding zero-mean subspaces.

Consider the solutions $\phi_{\mathcal{E}, K, z}(\Omega)$ and $\phi_{\mathcal{Z}, K, z}(\Omega)$ satisfying the equations

\begin{equation}\label{Eq:SymmetryBreakingFluxFastGenerator0PoissonEquation:Energy}
L^0_{\text{fast}}\phi_{\mathcal{E}, K, z}(\Omega)=\Pi_\mathcal{E}(K; \Omega) \ , \ \int_{\mathcal{O}^*(z)}\phi_{\mathcal{E}, K, z}(\Omega)dm_{z, \mathrm{KKS}}(\Omega)=0 \ ,
\end{equation}

\begin{equation}\label{Eq:SymmetryBreakingFluxFastGenerator0PoissonEquation:Enstrophy}
L^0_{\text{fast}}\phi_{\mathcal{Z}, K, z}(\Omega)=\Pi_\mathcal{Z}(K; \Omega) \ , \ \int_{\mathcal{O}^*(z)}\phi_{\mathcal{Z}, K, z}(\Omega)dm_{z, \mathrm{KKS}}(\Omega)=0 \ .
\end{equation}
These equations admit unique solutions $\phi_{\mathcal{E}, K, z}(\Omega), \phi_{\mathcal{Z}, K, z}(\Omega)\in \mathbf{C}^{(2)}_{\mathcal{O}^*(z)}(\mathbb{R})$ due to Corollary \ref{Cor:LeafAverageVanishKKSMeasure:EnergyEnstrophy} and the hypoellipticity assumption of $L^0_{\text{fast}}$. They are given by the formulae
\begin{equation}\label{Eq:SymmetryBreakingFluxFastGenerator0PoissonEquation:Solution}
\phi_{\mathcal{E}, K, z}(\Omega)=\left[L^0_{\text{fast}}\right]^{-1} (\Pi_\mathcal{E}(K; \Omega)) \ , \ \phi_{\mathcal{Z}, K, z}(\Omega)=\left[L^0_{\text{fast}}\right]^{-1} (\Pi_\mathcal{Z}(K; \Omega)) \ . 
\end{equation}
 
\begin{corollary}\label{Cor:SymetryBreakingAverageFluxNotZeroCondition}
If
\begin{equation}\label{Cor:SymetryBreakingAverageFluxNotZeroCondition:Eq:Energy}
\int_{\mathcal{O}^*(z)}\mathcal{R} \phi_{\mathcal{E}, K, z}(\Omega)dm_{z, \mathrm{KKS}}(\Omega) \neq 0 \ ,
\end{equation}
then the averaged nonlinear energy flux $\overline{\Pi}_{\mathcal{E}}(K; z)$ is nonzero for small enough $\kappa$. Similarly, if 
\begin{equation}\label{Cor:SymetryBreakingAverageFluxNotZeroCondition:Eq:Enstrophy}
\int_{\mathcal{O}^*(z)}\mathcal{R} \phi_{\mathcal{Z}, K, z}(\Omega)dm_{z, \mathrm{KKS}}(\Omega) \neq 0 \ ,
\end{equation}
then the averaged nonlinear enstrophy flux $\overline{\Pi}_{\mathcal{Z}}(K; z)$ is nonzero for small enough $\kappa >0$.
\end{corollary}

\begin{proof}
Using (\ref{Eq:SymmetryBreakingFastMotionInvariantMeasureSolutionFormulaForq}), (\ref{Eq:SymmetryBreakingFluxFastGenerator0PoissonEquation:Solution}) and the definition of the adjoint operation, the integrals in (\ref{Lm:SymetryBreakingAverageFlux:Eq}) can be written as
\begin{equation}\label{Cor:SymetryBreakingAverageFluxNotZeroCondition:Eq:EnergyMainIntegral}
\begin{array}{rl} \displaystyle{\int_{\mathcal{O}^*(z)}}\Pi_{\mathcal{E}}(K; \Omega)q(\Omega)dm_{z, \mathrm{KKS}}(\Omega)
= & - \displaystyle{\int_{\mathcal{O}^*(z)}}\Pi_{\mathcal{E}}(K; \Omega)\left[(L_\text{fast}^0)^*\right]^{-1}\left(\mathcal{R}^* 1\right)(\Omega)dm_{z, \mathrm{KKS}}(\Omega)
\\
= & - \displaystyle{\int_{\mathcal{O}^*(z)}}\left[L_{\text{fast}}^0\right]^{-1}\left(\Pi_{\mathcal{E}}(K; \Omega)\right) \mathcal{R}^* 1(\Omega) dm_{z, \mathrm{KKS}}(\Omega)
\\
= & - \displaystyle{\int_{\mathcal{O}^*(z)}}\phi_{\mathcal{E}, K, z}(\Omega) \mathcal{R}^* 1(\Omega) dm_{z, \mathrm{KKS}}(\Omega)
\\
= & - \displaystyle{\int_{\mathcal{O}^*(z)}}\mathcal{R}\phi_{\mathcal{E}, K, z}(\Omega) dm_{z, \mathrm{KKS}}(\Omega) \ ,
\end{array}
\end{equation}
and similarly
\begin{equation}\label{Cor:SymetryBreakingAverageFluxNotZeroCondition:Eq:EnstrophyMainIntegral}
\int_{\mathcal{O}^*(z)}\Pi_{\mathcal{Z}}(K;\Omega)q(\Omega)dm_{z,\mathrm{KKS}}(\Omega)=-\int_{\mathcal{O}^*(z)}\mathcal{R} \phi_{\mathcal{Z}, K, z}(\Omega)dm_{z, \mathrm{KKS}}(\Omega) \ .
\end{equation}
Therefore the conclusions of this Corollary follow from Lemma \ref{Lm:SymetryBreakingAverageFlux}.
\end{proof}

Using (\ref{Eq:SymmetryBreakingFastMotionGenerator:R}), the integrals of (\ref{Cor:SymetryBreakingAverageFluxNotZeroCondition:Eq:Energy}) and (\ref{Cor:SymetryBreakingAverageFluxNotZeroCondition:Eq:Enstrophy}) can be written as 
\begin{equation}\label{Eq:SymetryBreakingAverageFluxNotZeroIntergalGeneralForm:Energy}
\begin{array}{ll}
& \displaystyle{\int_{\mathcal{O}^*(z)}}\mathcal{R} \phi_{\mathcal{E}, K, z}(\Omega)dm_{z, \mathrm{KKS}}(\Omega)
\\
= & \sum\limits_{j=1}^m \displaystyle{\int_{\mathcal{O}^*(z)}}[D^2 \phi_{\mathcal{E}, K, z}(\Omega)](X_j^0(\Omega), \boldsymbol{v}_j(\Omega)) dm_{z, \mathrm{KKS}}(\Omega)
\\
& \qquad + \dfrac{1}{2}\sum\limits_{j=1}^m \displaystyle{\int_{\mathcal{O}^*(z)}} \left(\nabla \phi_{\mathcal{E}, K, z}(\Omega), [DX_j^0(\Omega)](\boldsymbol{v}_j(\Omega))+[D\boldsymbol{v}_j(\Omega)](X_j^0(\Omega))\right)dm_{z, \mathrm{KKS}}(\Omega) \ ,
\end{array}
\end{equation}
and
\begin{equation}\label{Eq:SymetryBreakingAverageFluxNotZeroIntergalGeneralForm:Enstrophy}
\begin{array}{ll}
& \displaystyle{\int_{\mathcal{O}^*(z)}}\mathcal{R} \phi_{\mathcal{Z}, K, z}(\Omega)dm_{z, \mathrm{KKS}}(\Omega)
\\
= & \sum\limits_{j=1}^m \displaystyle{\int_{\mathcal{O}^*(z)}}[D^2 \phi_{\mathcal{Z}, K, z}(\Omega)](X_j^0(\Omega), \boldsymbol{v}_j(\Omega))dm_{z, \mathrm{KKS}}(\Omega)
\\
& \qquad + \dfrac{1}{2}\sum\limits_{j=1}^m \displaystyle{\int_{\mathcal{O}^*(z)}}\left(\nabla \phi_{\mathcal{Z}, K, z}(\Omega), [DX_j^0(\Omega)](\boldsymbol{v}_j(\Omega))+[D\boldsymbol{v}_j(\Omega)](X_j^0(\Omega))\right)dm_{z, \mathrm{KKS}}(\Omega) \ .
\end{array}
\end{equation}

By Corollary \ref{Cor:SymetryBreakingAverageFluxNotZeroCondition}, nonzero energy and enstrophy flux happens if the RHS of (\ref{Eq:SymetryBreakingAverageFluxNotZeroIntergalGeneralForm:Energy}) or (\ref{Eq:SymetryBreakingAverageFluxNotZeroIntergalGeneralForm:Enstrophy}) is non-zero.

\bibliographystyle{plain}
\bibliography{bibliography_coadjoint_averaging_flux}

\end{document}